\documentclass[11pt]{amsart}
\usepackage{geometry}
\usepackage{booktabs}
\usepackage[greek,english]{babel}
\usepackage[dvipsnames]{xcolor}
\usepackage[colorlinks=true,linkcolor=red,urlcolor=blue]{hyperref}
\usepackage{amsthm}
\usepackage{graphicx}
\usepackage{times}
\usepackage{xcolor}
\graphicspath{ {./images/} }
\usepackage{tikz}
\usetikzlibrary{shapes.geometric, arrows.meta, positioning, calc}
\usepackage{multirow}
\usetikzlibrary{arrows.meta, positioning}
\usepackage{caption}
\usepackage{subcaption}
\usepackage[square,numbers]{natbib}
\usepackage[ruled,vlined]{algorithm2e}
\usepackage{hyperref}
\usepackage{comment}
\usepackage{enumitem}
\hypersetup{
    colorlinks=true,
    linkcolor=blue,
    filecolor=magenta,      
    urlcolor=cyan,
    pdftitle={Campi Cannerozzi Tzouanas - Optimal mean-field CCE},
    pdfpagemode=FullScreen,
}

\usepackage[title]{appendix}
\usepackage{fancyhdr}	
\usepackage[foot]{amsaddr}
\usepackage[colorinlistoftodos]{todonotes}

\numberwithin{equation}{section}
\newtheorem{theorem}{Theorem}[section]
\newtheorem{corollary}{Corollary}[section]
\newtheorem{lemma}{Lemma}[section]
\newtheorem{proposition}{Proposition}[section]
\newtheorem{remark}{Remark}[section]

\newtheorem{assumption}{Assumption}[section]
\newtheorem{definition}{Definition}[section]
\newtheorem{assumptions}{Assumptions}[section]

\newcommand{\bbA}{{\ensuremath{\mathbb A}} }

\newcommand{\bbF}{{\ensuremath{\mathbb F}} }

\newcommand{\cA}{{\ensuremath{\mathcal A}} }
\newcommand{\cB}{{\ensuremath{\mathcal B}} }

\newcommand{\cD}{{\ensuremath{\mathcal D}} }
\newcommand{\cE}{{\ensuremath{\mathcal E}} }
\newcommand{\cF}{{\ensuremath{\mathcal F}} }
\newcommand{\cG}{{\ensuremath{\mathcal G}} }

\newcommand{\cK}{{\ensuremath{\mathcal K}} }
\newcommand{\cL}{{\ensuremath{\mathcal L}} }
\newcommand{\cM}{{\ensuremath{\mathcal M}} }
\newcommand{\cN}{{\ensuremath{\mathcal N}} }

\newcommand{\cP}{{\ensuremath{\mathcal P}} }

\newcommand{\cR}{{\ensuremath{\mathcal R}} }

\newcommand{\cW}{{\ensuremath{\mathcal W}} }

\newcommand{\cZ}{{\ensuremath{\mathcal Z}} }

\newcommand{\bfJ}{{\ensuremath{\mathbf J}} }

\newcommand{\frm}{{\ensuremath{\mathfrak m}} }

\newcommand{\dC}{{\ensuremath{\mathrm C}} }

\newcommand{\dM}{{\ensuremath{\mathrm M}} }

\newcommand{\F}{\mathbb{F}}
\newcommand{\R}{\mathbb{R}}

\newcommand{\N}{\mathbb{N}}
\newcommand{\Q}{\mathbb{Q}}

\renewcommand{\P}{\mathbb{P}}
\newcommand{\E}{\mathbb{E}}

\newcommand{\1}{\ensuremath{\mathbf{1}}}

\title[Optimal Mean Field CCEs]{Optimal Coarse Correlated Equilibria in Mean Field Games: \\ Linear Programming and No-Regret Learning}
\date{\today}
\author[Campi]{Luciano Campi $^{\ast,1}$}
\email{\href{mailto:luciano.campi@unimi.it}{$^{1}$luciano.campi@unimi.it}}
\author[Cannerozzi]{Federico Cannerozzi $^{\dagger,2}$}
\email{\href{mailto:federico.cannerozzi@uni-bielefeld.de}{$^{2}$federico.cannerozzi@uni-bielefeld.de}}
\author[Tzouanas]{Ioannis Tzouanas $^{\dagger,3}$}
\email{\href{mailto:ioannis.tzouanas@uni-bielefeld.de}{$^{3}$ioannis.tzouanas@uni-bielefeld.de}}
\address{$^{\ast}$Department of Mathematics “Federigo Enriques”, University of Milan, Via Saldini 50, 20133, Milan, Italy.}
\address{$^{\dagger}$Center for Mathematical Economics (IMW), Bielefeld University, Universit\"atsstrasse 25, 33615, Bielefeld, Germany.}

\begin{document}

\begin{abstract}
We introduce optimal coarse correlated equilibria for continuous-time mean field games. A coarse correlated equilibrium is a randomized recommendation scheme from which no player can gain by ignoring the recommendation and switching to an alternative strategy. The problem is as follows: a moderator selects, among all mean-field coarse correlated equilibria, one that optimizes a prescribed performance criterion, which may differ from the representative player’s objective. After formulating the problem, we develop a linear programming (LP) formulation, prove the existence of optimal LP coarse correlated equilibria, and relate the LP characterization to the original probabilistic setting. Building on this characterization, we design a no-regret primal-dual algorithm, based on an equivalent Lagrangian formulation of the external-regret constraint, for learning such equilibria. We provide explicit convergence rates for the learning algorithm, and numerical examples illustrate the method.
\end{abstract}

\maketitle


\smallskip

{\noindent \small \textbf{AMS 2020:} 91A16, 91A26, 68T05.}

\smallskip

\section{Introduction}

Mean field games (MFGs), introduced independently by Lasry and Lions \cite{LasryLions2007} and by Huang, Malham\'e and Caines \cite{HuangMalhameCaines2006}, provide a tractable limit framework for symmetric games with many players and interactions of mean-field type.
The standard MFG solution is the infinitely many players analogue of a Nash equilibrium: a representative player optimizes against a fixed population flow, and this flow must be generated when all players use an optimal response.
This connection can be made rigorous by proving convergence of Nash equilibria to MFG solutions or by showing that MFG solutions generate approximate Nash equilibria with vanishing error (see \cite{carmona2018probabilistic,carmona2018probabilistic_vol2} and the references therein).

The Nash paradigm, however, is only one possible notion of equilibrium.
In many situations a moderator can send recommendations to agents, thereby correlating their behavior without imposing actions.
This leads naturally to coarse correlated equilibria (CCEs), which appeared implicitly in Hannan's work on repeated play \cite{hannan1957} and explicitly in Moulin and Vial \cite{moulin_vial1978}.
In a CCE, a moderator draws a strategy profile from a publicly known distribution and privately recommends to each player her component.
Each player decides whether to commit before the lottery is realized, assuming that the others commit.
If she commits, she follows her own recommendation; if she deviates, she does so without observing the moderator's realization.
The lottery is a CCE if every player prefers to commit rather than deviate in this ex ante sense.
CCEs generalize both Nash equilibria and Aumann's correlated equilibria (CEs) \cite{aumann1970,aumann1987}.
Among their main advantages, CCEs may lead to higher payoffs than Nash equilibria \cite{moulin2022nplayer,moulinraysengupta2014}, may improve on correlated equilibria in relevant classes of games \cite{neyman1997correlated}, and can arise naturally as the outcome of no-regret learning algorithms and dynamics \cite{Blum_Mansour_2007,MasCollel,roughgarden_2016}.

The recent MFG literature has started to explore these other-than-Nash equilibrium concepts.
CEs for finite-state discrete-time MFGs were studied in \cite{bonesini2025correlated,campifischer2021}, while learning procedures for CEs and CCEs were considered in \cite{muller2022learningcorrelatedequilibriameanfield,muller2022learningequilibriameanfieldgames}.
In continuous time, CCEs were introduced in \cite{campi2024coarse}, and further constructive analyses were developed in \cite{campi2025LQ,cannerozzi2026cooperation}.
These works show that CCEs may be abundant: infinitely many CCEs can coexist with a unique MFG solution, and may even exist when no MFG solution exists \cite[Chapter 2]{cannerozzi2025thesis}.
This abundance of equilibria naturally raises a selection problem: among all coarse correlated equilibria, one may wish to choose an optimal CCE according to a prescribed criterion.
The aim of the present paper is to address the CCE selection problem by developing a general optimization approach and providing an associated learning algorithm.

Following \cite{campi2025LQ,campi2024coarse}, our notion of mean-field CCE is given by a recommended strategy together with a stochastic flow of measures.
The recommended strategy is modeled as a strict control, i.e., a stochastic process taking values in the action space.
The stochasticity of the flow comes from the moderator's randomization over recommendations, and the recommended strategy may be correlated with this flow.
Such a pair is required to satisfy two conditions.
First, consistency: at each time, the stochastic flow coincides with the conditional law of the representative player's state and recommended action, conditional on the whole stochastic flow.
Second, optimality: the representative player has no incentive to deviate before the recommendation is realized.
In particular, admissible deviations are chosen ex ante and cannot depend on the realization of the moderator's recommendation, but only on the player's own source of randomness.
In this context, we introduce an explicit objective criterion for the moderator, possibly different from the objective of the representative player, and we study the problem of selecting an optimal mean-field CCE according to this criterion.
We refer to this framework as the probabilistic formulation of mean-field CCEs.

The paper is built around two main contributions, which are complementary and closely connected.
First, we introduce a linear programming (LP) formulation of optimal mean-field CCEs.
The LP formulation provides a convenient framework for existence results and for the moderator's optimization problem over the set of mean-field CCEs.
Second, and equally importantly, we develop a no-regret learning algorithm to compute CCEs within this LP framework.
The no-regret formulation gives the constructive counterpart: it reformulates the CCE optimality condition as an external-regret constraint, and uses primal-dual updates to compute an optimal mean-field CCE.

In stochastic control, the linear programming approach is a classical relaxation technique, well suited to existence results and numerical approximation \cite{Kurtz_cond_markov,Kurtz_Stockbridge_SICON_98,kurtz_stockbridge2017singular}.
In the MFG literature, linear programming formulations have been developed in \cite{bouveret_dumitrescu2020sicon,dumitrescu2021EJP,LinProgFictDumitrescu,guo2025continuous,leutscher_thesis}.
Linear programming methods are also classical in the computation of CEs and CCEs in finite-player finite-action games and their compact representations, where the equilibrium conditions can be written as linear inequalities; see, for instance, \cite{jiangleyton2015,paparough2008} and \cite[Problem~18.3]{roughgarden_2016}.
The linear programming approach prescribes to formulate the controlled dynamics directly in terms of occupation measures.
In our setting, an occupation measure describes the joint distribution of time, state and recommended action along the trajectory of the representative player, together with the terminal law.
The controlled dynamics are encoded by linear martingale constraints.
The moderator is allowed to randomize over the space of admissible occupation measures, instead of selecting a deterministic equilibrium flow of occupation measures, thus extending the LP formulation developed for MFG Nash equilibria in \cite{dumitrescu2021EJP,LinProgFictDumitrescu,leutscher_thesis} to the CCE framework.
A correlated recommendation is therefore represented by a probability measure on the space of admissible occupation measures.
The consistency condition is imposed at the level of the induced flow, and the CCE no-deviation condition becomes a family of linear inequalities comparing the recommended payoff with the payoff of every admissible ex ante deviation.
This representation has two important consequences: it gives a tractable relaxed formulation of the moderator's optimization problem over mean-field CCEs, and it provides the compactness and convexity properties needed to guarantee existence of an optimal LP-CCE.
Under our assumptions, the set of LP-CCEs is compact and convex, while the moderator's criterion is continuous and linear on this set.
As a consequence, an optimal LP-CCE exists.
We also relate the LP and probabilistic formulations.
Under convexity assumptions, we show that an optimal LP-CCE can be represented by a mean-field CCE in the probabilistic formulation; when the dynamics are independent of the mean-field term, this representation preserves optimality for the moderator among mean-field CCEs.

The second main contribution is a no-regret primal-dual algorithm built on the LP formulation.
In line with the classical connection between no-regret learning and CCEs \cite{Blum_Mansour_2007,Foster_Vohra,hannan1957,MasCollel,roughgarden_2016}, we define the external regret as the maximal gain obtainable by replacing the recommendation with a fixed admissible deviation that does not observe the realization of the moderator's randomization.
A correlated flow is a mean-field CCE precisely when its external regret is non-positive, that is, when the representative player cannot benefit from deciding ex ante to ignore the recommendation and use a different admissible strategy.
This reformulation turns the moderator's selection problem into a constrained optimization problem over LP-correlated flows.
Introducing a Lagrange multiplier for the external-regret constraint yields an equivalent saddle-point formulation.
Our algorithm alternates between a primal update and a dual update.
The primal variable is a probability measure over admissible martingale flows, while the dual variable penalizes positive regret.
The primal step minimizes a Bregman-regularized Lagrangian, in the spirit of mirror-descent and primal-dual methods \cite{Beck_Teboulle,Chambolle_Contreras,Shavel_Shwartz}, and the dual step updates the multiplier according to the regret of the current iterate.
Under a uniqueness condition for the best deviation, the external-regret functional admits an envelope-type derivative.
This differentiability is key to prove the convergence, because it yields the three-point inequalities used to control the primal-dual gap.
We obtain an $O(N^{-1/2})$ convergence bound for averaged iterates, and the averaged regret is controlled at the same order up to the moderator objective gap.
Consequently, subsequential limits of the averaged iterates are saddle points of the Lagrangian and yield optimal LP-CCEs.
The same algorithm can also be used when the moderator has no additional performance criterion, in which case it reduces to a procedure for computing mean-field CCEs by driving external regret to zero, with rates consistent with the classical no-regret literature \cite{Cesa-Bianchi_Lugosi_2006}.

To the best of our knowledge, this is the first no-regret learning algorithm for CCE selection in a continuous-time stochastic MFG.
Existing learning algorithms for MFGs are mostly Nash-oriented.
Fictitious play for MFGs was introduced by Cardaliaguet and Hadikhanloo \cite{CardaliaguetHadikhanloo} and further developed in \cite{hadikhanloo2017,PerrinLauriereEliePerolat}.
Dumitrescu, Leutscher and Tankov \cite{LinProgFictDumitrescu} combined fictitious play with the LP formulation for MFGs with optimal stopping and absorption.
A complementary line of work uses order-theoretic methods based on Tarski's fixed point theorem and monotone best-response iteration in submodular MFGs \cite{dianetti2021submodular,dianetti2023unifying}.
Our method differs from these approaches in both target and mechanism: it targets CCEs rather than Nash equilibria, and it uses external regret rather than repeated best response to a mean-field history.
In fictitious play, agents repeatedly compute best responses to an estimated population behavior; here, instead, the algorithm searches for a recommendation whose average performance cannot be improved by any fixed ex ante deviation.
Moreover, the convergence result does not rely on a Lasry--Lions monotonicity condition, which is consistent with the findings of \cite{muller2022learningcorrelatedequilibriameanfield,muller2021learning} on learning algorithms for mean-field CEs and CCEs.
Similar no-regret and primal-dual ideas have been used for optimal CCE computation in finite or extensive-form games \cite{AnagnostidesFarinaPanageasSandholm,CaiDaskalakis2025proximal,FarinaOptimalCCE}, while the present paper adapts this philosophy to continuous-time mean-field dynamics through the LP formulation.

We also provide a parametrized implementation of the learning scheme.
Since the infinite dimensional primal-dual algorithm cannot be implemented directly, we introduce a family of randomized policies parametrized by neural networks and a finite-dimensional correlation device.
This is related in spirit to recent computational approaches to incentive design and Stackelberg-type MFGs, where an external agent optimizes her own criterion subject to the equilibrium behavior of a large population \cite{DayanikliLauriere,MarrisGempThomasTacchettiTuyls,MarrisMullerLancotTuylsGraepel,ThomaPiliourasMarrisDeepIncentiveLearning}.
The induced measure-valued flows are constrained through the martingale formulation of the dynamics, while the moderator's objective and the external regret are evaluated on the parametrized class.
We also refer to recent machine-learning methods for MFGs and mean-field control \cite{LauriereActorCriticJMLR}, and to related deep reinforcement learning approaches \cite{GuoHuZhang_MF_OMO,GuoXuZariphopoulou,HuZhang_ML-OML,MagninoShaoWuLauriere}.
The resulting primal-dual gradient method is tested on a simple flocking model and on the emission abatement game introduced, in its mean-field formulation, in \cite{campi2025LQ}.
The experiments illustrate how correlated recommendations can interpolate between Nash-type behavior and more coordinated outcomes, and how the moderator's objective can select equilibria with different welfare or aggregate-performance properties.

\subsection*{Structure of the paper}
The rest of the paper is organized as follows:
Section~\ref{sec:assumptions} states the standing assumptions.
Section~\ref{section: Prob formulation} introduces the probabilistic formulation of mean-field CCEs.
Section~\ref{section: Linear Programming} develops the linear programming formulation and defines LP-CCEs.
Section~\ref{section: solution to LP problem} proves existence of an optimal LP-CCE and studies the relation between the LP and the probabilistic formulations.
Section~\ref{section: Learning algorithm} introduces the external-regret formulation and the primal-dual learning algorithm, together with its convergence analysis.
Section~\ref{Section: Implementation and Examples} presents the parametrized implementation of the algorithm.
Section~\ref{section: examples} contains the numerical experiments, while Sections~\ref{sec:appendix:technical_results_existence}, ~\ref{section: Appendix: Learning Part: Technical Results} and~\ref{app:strictCCE} collect the technical results.

\section{Assumptions}\label{sec:assumptions}

In the following, fix a finite time horizon $T>0$, an action space $A\subseteq\R$, a constant $\sigma>0$, an initial distribution $m_0^*\in\cP(\R)$, and the following functions:
\begin{align*}
    (b,f,f^0): [0,T] \times \R \times \cP_2(\R \times A) \times A \to \R \times \R \times \R,\\ 
    (g,g^0): \R \times \cP_2(\R \times A) \to \R \times \R.
\end{align*}

The following assumptions will be in force throughout the whole paper. Throughout, $c > 0$ denotes a generic constant, possibly changing from line to line, uniform over the relevant variables.

\begin{assumptions}\label{standing_assumptions}
\begin{enumerate}[label=(S\arabic*)]
    \item $A$ is a compact subset of $\R$.
    \item $m^*_0 \in \cP_{\bar{p}}(\R)$, $\bar{p} > 2$.
    
    \item The function $(t,x,m,a) \mapsto b(t,x,m,a)$ is measurable and jointly continuous in $(x,m,a)$ for any $t$ fixed, and it is defined by 
    \[
        b(t,x,m,a) = \bar{b}\left(t,x,\int_{\R \times A}\hat{b}(t,y)m(dy,du),a \right),
    \]
    where $\bar{b}:[0,T] \times \R \times \R \times A \to \R$ is measurable and continuous in $(x,y,a)$ for every fixed $t$, and $\hat{b} :[0,T] \times \R \to \R$ is continuous.
    The functions $\bar{b}$ and $\hat{b}$ are bounded.
    Moreover, $\bar b$ is Lipschitz continuous in its second variable $x$, uniformly with respect to the other variables.

    \item The functions $(t,x,m,a) \mapsto f(t,x,m,a)$ and $(t,x,m,a) \mapsto f^0(t,x,m,a)$ are jointly measurable and continuous in $(x,m,a)$ for any fixed $t \in [0,T]$
    and are defined by 
    \begin{align*}
        f(t,x,m,a) & = \bar{f}\left(t,x, \int_{\R \times A}\hat{f}(t,y)m(dy,du),a \right), \\
        f^0(t,x,m,a) & = \bar{f}^0\left(t,x, \int_{\R \times A}\hat{f}^0(t,y)m(dy,du),a \right),
    \end{align*}
    where $(\bar{f},\bar{f}^0):[0,T] \times \R \times \R \times A \to \R \times \R$ are measurable and continuous in $(x,y,a)$ for every fixed $t$, and $(\hat{f},\hat{f}^0) :[0,T] \times \R \to \R$ are continuous.
    $\bar{f}(t,x,y,a)$ and $\bar{f}^0(t,x,y,a)$ satisfy a quadratic growth condition, i.e., for any $(t,x,y,a) \in [0,T]\times \R \times \R \times A$, it holds
    \begin{align*}
        & \vert \bar{f}(t,x,y,a) \vert \leq c(1 + \vert x \vert^2 + \vert a \vert^2 + \vert y \vert^2)  \\
        & \vert \bar{f}^{0}(t,x,y,a) \vert \leq c(1 + \vert x \vert^2 + \vert a \vert^2 + \vert y \vert^2).
    \end{align*}
    $\bar{f}(t,x,y,a)$ and $\bar{f}^0(t,x,y,a)$ are locally Lipschitz in $y$, i.e., for any $y,y' \in \R$, $(t,x,a) \in [0,T]\times \R \times A$, it holds
    \begin{align*}
        & \vert \bar{f}(t,x,y,a) - \bar{f}(t,x,y',a)\vert \leq c(1 + \vert x \vert + \vert a \vert + \vert y \vert + \vert y' \vert  ) \vert y - y' \vert, \\
        & \vert \bar{f}^0(t,x,y,a) - \bar{f}^0(t,x,y',a)\vert \leq c(1 + \vert x \vert + \vert a \vert + \vert y \vert + \vert y' \vert  ) \vert y - y' \vert.
    \end{align*}
    $\hat{f}, \hat{f}^0 \in \dC^{1,2}([0,T] \times \R)$ and have bounded derivatives.
    
    \item $(g,g^0):\R \times \cP_2(\R \times A) \to \R \times \R$ are measurable and  jointly continuous in $(x,m)$, and such that
    \begin{align*}
        g(x,m)  = \bar{g}\left(x,\int_{\R \times A}\hat{g}(y)m(dy,du)\right), \quad g^0(x,m)  = \bar{g}^0\left(x,\int_{\R \times A}\hat{g}^0(y)m(dy,du)\right),
    \end{align*}
    where $(\bar{g},\bar{g}^0):\R \times \R \to \R \times \R$ and $(\hat{g},\hat{g}^0): \R \to \R$ are continuous.
    $\bar{g}(x,y)$ and $\bar{g}^0(x,y)$ have quadratic growth in $(x,y)$, in the sense that
 \begin{align*}
        & \vert \bar{g}(x,y) \vert \leq c(1 + \vert x \vert^2 + \vert y \vert^2),  \quad \vert \bar{g}^{0}(x,y) \vert \leq c(1 + \vert x \vert^2+ \vert y \vert^2).
    \end{align*}    $\bar{g}(x,y)$ and $\bar{g}^0(x,y)$ are locally Lipschitz in $y$, in the sense that, for any $y,y' \in \R$, $x \in \R$, it holds
    \begin{align*}
        &     \vert \bar{g}(x,y) - \bar{g}(x,y')\vert \leq c(1 + \vert x \vert + \vert y \vert + \vert y' \vert  ) \vert y - y' \vert, \\
        &     \vert \bar{g}^0(x,y) - \bar{g}^0(x,y')\vert \leq c(1 + \vert x \vert + \vert y \vert + \vert y' \vert  ) \vert y - y' \vert.
    \end{align*}
    Moreover, $\hat{g}, \hat{g}^0 \in \dC(\R)$ and have linear growth.
\end{enumerate}
\end{assumptions}

Before we proceed, we make some comments on the Assumption \ref{standing_assumptions}:
\begin{itemize}
    \item The compactness of the action space $A$ is mainly used in the linear programming formulation. It is standard in compactification arguments for relaxed controls and occupation measures, as in \cite{dumitrescu2021EJP,LinProgFictDumitrescu,Kurtz_Stockbridge_SICON_98,kurtz_stockbridge2017singular}.

    \item We take the volatility coefficient to be a positive constant. This avoids introducing martingale measures. Moreover, the constant nondegenerate volatility and the boundedness and Lipschitz continuity of the drift in the state variable allow us to rely on strong existence and pathwise uniqueness results for SDEs.

    \item Although the measure argument $m$ belongs to $\cP_2(\R\times A)$, the interactions in (S3)--(S5) are of scalar state-marginal type, in the sense that
    \[
    \int_{\mathbb R\times A}\hat h(t,y)m(dy,du) = \int_{\mathbb R}\hat h(t,y)m^x(dy),
    \qquad m^x(\cdot):=m(\cdot\times A).
    \]
    We keep the joint measure $m$ in the formulation because the consistency condition (see \eqref{eq:CCE:consistency}) is imposed on the joint law of the state and the recommended action. This is coherent with the assumptions made in \cite[Assumption 3.1]{dumitrescu2021EJP}.
    
\end{itemize}

\section{Probabilistic Formulation}\label{section: Prob formulation}
Let $(\Omega,\cF,\bbF,\P)$ be a complete filtered probability space satisfying the usual assumptions.
Let $W$ be a $\bbF$-Brownian motion and $\xi$ an independent real-valued $\cF_0$-measurable random variable with law $m_0^*$. 
We denote as $\bbF^{\xi,W} = (\cF^{\xi,W}_t)_{t \in [0,T]}$ the filtration generated by $(\xi,W)$ and satisfying the usual conditions.
Moreover, we impose the following structural condition on $\cF_0$:
\begin{assumption}\label{assumption:F_0}
The $\sigma$-algebra $\mathcal F_0$ is large enough to support a uniform random variable $U \sim \mathrm{Unif}(0,1)$, independent of $\xi$ and $W$.
\end{assumption}
Assumption \ref{assumption:F_0} ensures that $\cF_0$ is big enough to support the extra-randomization that the mediator uses to randomize the players' strategies, independently of the idiosyncratic shocks that determine the representative player's state dynamics.

\smallskip
In order to define the mean-field CCE, we adopt the following definitions, which are inspired by \cite{campi2025LQ,campi2024coarse}. We denote by $\dM(\cP_2)$ the set of all measurable maps $m:[0,T] \to \cP_2(\R \times A)$, where $\cP_2(\R \times A)$ is equipped with the topology of weak convergence of probability measures with continuous test functions of at most quadratic growth. We equip $\dM(\cP_2)$ with the $\sigma$-field generated by the maps $m\mapsto m_t(B)$, for any $t \in [0,T]$, $B\in\cB_{\R\times A}$.

\begin{definition}[Correlated measure flow]\label{probabilistic:def_corr_flow}
A correlated flow is a pair $(\lambda,\mu)$ satisfying the following properties:
\begin{enumerate}[label=\roman*)]
    \item $\lambda=(\lambda_t)_{t \in [0,T]}$ is an $\bbF$-progressively measurable process with values in $A$.
    \item $\mu=(\mu_t)_{t \in [0,T]}$ is a $\cF_0$-measurable random variable with values in $\dM(\cP_2)$, such that the map $(t,\omega)\mapsto \mu_t(\omega)$ is $\cB([0,T])\otimes\cF_0$-measurable.
    \item $\mu$ is independent of $\xi$ and $W$.
\end{enumerate}
We refer to $\lambda$ as the \textit{recommended strategy} and to $\mu$ as the \textit{random flow of measures}. 
\end{definition}

We now assign the dynamics and the cost functional, to be minimized, associated to the representative player.
We distinguish two cases: if the representative player decides to play according to the recommendation $\lambda$, the dynamics is given by
\begin{equation}\label{strong_formulation:repr_player:dynamics}
    dX_{t}= b(t,X_{t},\mu_{t},\lambda_{t})dt+\sigma dW_{t}, \quad X_0 = \xi,
\end{equation}
and the cost functional is given by
\begin{equation}
    J(\lambda,\mu):= \mathbb{E} \bigg[ \int_{0}^{T}f(t,X_{t},\mu_{t},\lambda_{t})dt+g(X_{T},\mu_{T}) \bigg].
\end{equation}

If instead the representative player ignores the recommendation, she will do so without having any information on the outcome of the moderator's lottery. She will then use a strategy $\beta\in\mathbb{A}$, where
\begin{equation}\label{probabilistic:admissible_deviations}
    \bbA := \{\beta:[0,T] \times \Omega \to A: \, \text{$\beta$ is $\bbF^{\xi,W}$-progressively measurable}\}.
\end{equation}
If the representative player ignores the mediator's recommendation, we refer to her as the deviating player.
For any admissible deviation $\beta \in \bbA$, the dynamics of the deviating player is given by
\begin{equation}\label{strong_formulation:dev_player:dynamics}
    dX_{t}= b(t,X_{t},\mu_{t},\beta_{t})dt+\sigma dW_{t}, \quad X_0 = \xi,
\end{equation}
and the cost functional is given by
\begin{equation}
    J(\beta,\mu):= \mathbb{E} \bigg[ \int_{0}^{T}f(t,X_{t},\mu_{t},\beta_{t})dt+g(X_{T},\mu_{T}) \bigg].
\end{equation}

Observe that any $\beta\in\bbA$ is $\bbF^{\xi,W}$-progressively measurable.
Since $\mu$ is independent of $(\xi,W)$, every admissible deviation $\beta\in\bbA$ is independent of the random flow of measures $\mu$.
Moreover, under Assumptions \ref{standing_assumptions}, for any correlated flow $(\lambda,\mu)$ and any deviation $\beta \in \bbA$, there exists a pathwise unique solution to both equations \eqref{strong_formulation:repr_player:dynamics} and \eqref{strong_formulation:dev_player:dynamics}.

\begin{definition}[Mean-field Coarse Correlated Equilibrium]\label{def:CCE}
A correlated flow $(\lambda,\mu)$ is a mean-field CCE if the following properties hold:
\begin{enumerate}
\item \label{Optimality Strong sense} (Optimality) For every deviation $\beta\in\mathbb{A}$, it holds
\begin{equation}\label{Optimality Strong Inequality}
    J(\lambda,\mu)\leq J(\beta,\mu).
\end{equation}
\item \label{Consistency Strong sense} (Consistency) It holds
    \begin{equation}\label{eq:CCE:consistency}
    \mu_{t}(\cdot)=\P((X_{t},\lambda_t) \in \cdot|\mu), \quad \P\text{-a.s., } \;\forall t\in [0,T].
    \end{equation}
\end{enumerate}
We refer to a mean-field CCE also as coarse correlated solution to the MFG.
\end{definition}

Observe that Definition \ref{def:CCE} implies the one given in \cite{campi2024coarse}, as the consistency condition \eqref{eq:CCE:consistency} implies
\[
\mu^x_t(\cdot) = \P( X_t \in \cdot | \mu^x).
\]
Sometimes we will use the notation $X^\lambda$ when we want to stress the dependence of the state on the recommended strategy $\lambda$, and likewise for the deviations.

\begin{remark}[Information asymmetry]\label{rmk:information_asymmetry}
We stress that the optimality condition \eqref{Optimality Strong Inequality} is formulated by comparing the costs yielded by the correlated measure flow $(\lambda,\mu)$ and by admissible deviations, which belong to the smaller class of $\bbF^{\xi,W}$-adapted processes.
This is due to the information asymmetry between the committing player and the deviating player. Indeed, if a player deviates from moderator's recommendation $\lambda$, she will do so without knowing the outcome of mediator's lottery over the set of strategy profiles. 
In the mean-field formulation, this is reflected by the independence between the random flow of measures $\mu$ and any admissible deviation $\beta \in \bbA$, as $\mu$'s stochasticity comes from the mediator's lottery only. In other terms, when a player deviates, only the Brownian filtration $\bbF^{\xi,W}$ will be available to her, and so the deviating strategy will be independent of the flow of measures $\mu$.
\end{remark}

\begin{remark}[Relation with coarse correlated solution of \cite{campi2024coarse}]
Differently from \cite{campi2024coarse} and more in line with \cite{campi2025LQ}, we do not introduce explicitly the extra randomness of the moderator through the construction of a  product probability space.
Instead, we suppose that $\cF_{0}$ is large enough to provide room for moderator's lottery over the representative player's strategies. 
In line with \cite{campi2024coarse} and differently from \cite{campi2025LQ}, we consider a (stochastic) flow of measures instead of just averages. We notice here that the measure flows take values in $\cP(\R \times A)$ and not in $\cP(\R)$, as we consider the joint measure of the state and the recommendation, in the same spirit of \cite{dumitrescu2021EJP,LinProgFictDumitrescu,leutscher_thesis}.
Finally, as (at the equilibrium) we consider the flow of joint laws of $(X^{\lambda}_t,\lambda_t)$, we cannot suppose that $\mu$ takes values in $\dC([0,T],\cP_2(\R \times A) )$, but we must consider only measurable flows of probability measures. 
\end{remark}

We now introduce an additional cost, that the moderator aims at minimizing over all coarse correlated solutions $(\lambda,\mu)$ of the mean field game.
For each correlated flow $(\lambda,\mu)$, define the objective functional
\begin{equation}
    J^{0}(\lambda,\mu):=\mathbb{E} \bigg[ \int_{0}^{T}f^{0}(t,X_{t},\mu_{t},\lambda_{t})dt+g^0(X_{T},\mu_{T}) \bigg],
\end{equation}
where $f^0, g^0$ are as in Assumptions \ref{standing_assumptions} and $X$ satisfies \eqref{strong_formulation:repr_player:dynamics}.
\begin{definition}\label{def:probabilistic_formulation:optimal_CCE}
A correlated flow $(\lambda^{*},\mu^{*})$ is optimal for the moderator if $(\lambda^{*},\mu^{*})$ is a mean-field CCE and it satisfies
\[
J^{0}(\lambda^{*},\mu^{*}) \leq J^{0}(\lambda,\mu) \quad \forall (\lambda,\mu) \text{ mean-field CCE}. 
\]
\end{definition}

\begin{remark}
For the sake of completeness, we recall from \cite[Definition 3.2]{dumitrescu2021EJP} the definition of MFG solutions in the strong formulation, in the case where there is no absorption and the representative player solves only a control problem with regular controls.
We say that a pair $(\alpha^\star,\mu^\star)$, with $\alpha^\star\in \bbA$ and $\mu^\star$ a deterministic measurable flow of probability measures on $\R \times A$, is a solution to the MFG problem if
\begin{enumerate}
\item \label{Optimality Strong sense Nash} (Optimality) For every deviation $\beta \in \bbA$, it holds
\begin{equation}\label{probabilistic:def:MFG:optimality}
    J(\alpha^\star,\mu^\star)\leq J(\beta,\mu^\star).
\end{equation}
\item \label{probabilistic:def:MFG:consistency} (Consistency) It holds
    \begin{equation}\label{eq:NE:consistency}
    \mu^\star_{t}(\cdot)=\P((X^\star_{t},\alpha^\star_t) \in \cdot), \;\forall t\in [0,T],
    \end{equation}
    where $X^\star$ denotes the solution to \eqref{strong_formulation:repr_player:dynamics} for $(\lambda,\mu)=(\alpha^\star,\mu^\star)$.
\end{enumerate}
It is straightforward to see that 
Definition \ref{def:CCE} of mean-field CCEs extends the notion of solution to the MFG.
\end{remark}

\section{Linear Programming}\label{section: Linear Programming}
Our next goal is to adopt a linear programming formulation for our problem.
Following \cite{dumitrescu2021EJP,LinProgFictDumitrescu,leutscher_thesis}, the key idea is to formulate the dynamics of the representative player in terms of martingale problems, as done by \cite{Kurtz_cond_markov,Kurtz_Stockbridge_SICON_98,kurtz_stockbridge2017singular}.

\subsection*{Preliminaries}

Let $V_2$ be the set of measurable flows $m=(m_t)_{t\in[0,T]}$ from $[0,T]$ to $\cP_2(\R\times A)$ such that
\[
    \int_0^T\int_{\R\times A} |x|^2 m_t(dx,da)dt<\infty.
\]
We identify two flows which agree $dt$-a.e. on $[0,T]$. To each $m\in V_2$, we associate the Borel finite positive measure on $[0,T]\times\R\times A$ with time marginal $dt$ defined by $m_t(dx,da)dt$.
We endow $V_2$ with the following weak convergence topology:
We say that $m^n \to m$ in $V_2$ if 
\begin{equation}\label{eq:def:topology}
    \int_0^T \int_{\R \times A} g(t,x,a) m^n_t(dx,da) dt \overset{n \to \infty}{\longrightarrow} \int_0^T \int_{\R \times A} g(t,x,a) m_t(dx,da) dt
\end{equation}
for any $g:[0,T] \times \R \times A \to \R$ continuous and with at most quadratic growth in $(t,x,a)$.
We denote such a topology by $\tau^{(2)}$.
We remark that, since $[0,T]$ and $A$ are compact, it is enough to consider as test functions continuous functions with at most quadratic growth in $x$, uniformly in $(t,a)$.

We will consider also the stable topology on $V_2$, defined as follows: We say that $m^n \to m$ in $V_2$ in the stable topology if 
\begin{equation}\label{eq:def:stable_topology}
    \int_0^T \int_{\R \times A} g(t,x,a) m^n_t(dx,da) dt  \overset{n \to \infty}{\longrightarrow} \int_0^T \int_{\R \times A} g(t,x,a) m_t(dx,da) dt
\end{equation}
for any measurable function $g:[0,T]\times\R\times A\to\R$ such that, for every $t\in[0,T]$, $g(t,\cdot,\cdot)$ is continuous and has at most quadratic growth in $x$.
We denote this topology by $\bar{\tau}^{(2)}$.

On $\cP_2(\R)$ consider the topology generated by weak convergence with continuous test functions of at most quadratic growth, which we still denote by $\tau^{(2)}$, with little abuse of notation.
We will consider the product space $V_2 \times \cP_2(\R)$, endowed either with the product topology $\tau^{(2)}\otimes \tau^{(2)}$ or with the product topology $\bar{\tau}^{(2)}\otimes\tau^{(2)}$.
We denote $(\eta,\overline{\eta}) \in V_2 \times \cP_2(\R)$ by $\boldsymbol{\eta}$ as well.

\smallskip
\begin{definition}[Martingale property for fixed flow of measures]\label{LP:def:pre_martingale_constraint}
Let $m \in V_2$. We say that $\boldsymbol{\eta} = (\eta,\overline{\eta})$ in $V_2 \times \cP_2(\R) $ satisfies the martingale property relative to $m$ if, for any test function $u \in \dC^{1,2}_{b}([0,T]\times \R)$, it holds
\begin{equation}\label{eq:def:pre_martingale_property}
    \int_{\R} u(T,x) \overline{\eta}(dx) -\int_{\R} u(0,x)m^*_0(dx) = \int_{0}^{T}\int_{\R \times A} (\partial_t+\cL ) u(t,x,m_t,a)\eta_{t}(dx,da)dt,
\end{equation}
where $\cL u$ is defined by
\begin{multline}\label{LP:kolmogorov_operator}
    \cL u(t,x,m,a):=\frac{\sigma^2}{2}\partial^{2}_x u(t,x) +b(t,x,m,a)\partial_x u(t,x), \\
    (t,x,m,a) \in [0,T] \times \R \times \cP_2(\R \times A) \times A.
\end{multline}
We denote by $\cR[m]\subseteq V_2 \times \cP_2(\R)$ the set of $\boldsymbol{\eta}$ which satisfy the martingale property relative to $m$.
\end{definition}

Define $\cR_0$ as the $\tau^{(2)} \otimes \tau^{(2)}$-closure of the convex hull of the set
\[
\bigcup_{m \in V_2} \cR[m] \subseteq V_2 \times \cP_2(\R).
\]
The following holds:
\begin{lemma}[\cite{leutscher_thesis}, Proposition 3.8, Lemma 3.10, Lemma 3.13]\label{lemma:compactness_a_priori_R0}
 $(\cR_0,\tau^{(2)} \otimes \tau^{(2)})$ is compact and convex.
Moreover, on $\cR_0$, the topologies $\tau^{(2)} \otimes \tau^{(2)}$ and $\bar{\tau}^{(2)} \otimes \tau^{(2)}$ coincide.
Finally, there exists a constant $C>0$ so that, for any $(\eta,\overline{\eta}) \in \cR_0$, it holds
\begin{equation}\label{uniform_bound}
    \int_{\R \times A} (\vert x \vert^2 + \vert a \vert^2) \eta_{t}(dx,da) \leq C, \quad \forall \, t \in [0,T], \quad \int_{\R} \vert x \vert^2 \overline{\eta}(dx) \leq C.
\end{equation}
\end{lemma}

\subsection*{CCEs in the linear programming formulation}

We now introduce the LP-analogue of correlated flows and deviations from the moderator's suggestion.
Following \cite[Section 6]{campi2024coarse} and \cite[Section 3]{campi2025LQ}, we restrict to the LP-analogue of consistent recommendations, that is, we consider only correlated measure flows which satisfy the consistency condition \eqref{eq:CCE:consistency}. In this way, we will have to deal with the optimality property only.

\begin{definition}[Consistent Martingale Property]\label{LP:def:consistent_mtg_property}
Let $\rho \in \cP\big( V_2 \times \cP_2(\R) \big)$.
We say that $\rho$ satisfies the consistent martingale property if, for any test functions $u \in \dC^{1,2}_{b}([0,T]\times \R)$ and $\phi \in \dC_{b}(V_2 \times \cP_2(\R),\tau^{(2)}\otimes\tau^{(2)})$, the following equality is satisfied:
\begin{multline}\label{eq:consistent_mtg_property}
    \int_{V_2 \times \cP_2(\R)} \phi(m,\overline{m}) \bigg( \int_{\R} u(T,x) \overline{m}(dx) -\int_{\R} u(0,x)m_0^*(dx) \\
    - \int_{0}^{T}\int_{\R \times A} (\partial_t+\cL ) u(t,x,m_t,a)m_{t}(dx,da)dt \bigg) \rho(dm,d\overline{m}) = 0.
\end{multline}
We denote by $\cM$ the set of measures $\rho$ that satisfy the consistent martingale property \eqref{eq:consistent_mtg_property}.
We refer to $\rho$ as a consistent LP-correlated flow.
\end{definition}

We now consider admissible deviations.
To this extent, for a Polish metric space $(E,d_E)$, denote by $\cW_2$ the $2$-Wasserstein distance on $\cP_2(E)$ (see, e.g, \cite[Definition 6.1]{villani2008optimal}).
In what follows, $\cW_2$ is used either on $\R$ or on $\R\times A$, endowed with the corresponding product metric.
For $m,m'\in V_2$, define the pseudo-distance
\[
    d_b(m,m'):=\operatorname*{ess\,sup}_{t\in[0,T]}\sup_{(x,a)\in\R\times A}\left|b(t,x,m_t,a)-b(t,x,m'_t,a)\right|.
\]
Notice that $d_b$ is finite under Assumption \ref{standing_assumptions}, since $b$ is bounded.
For $\boldsymbol{\eta}=(\eta,\overline{\eta})$ and $\boldsymbol{\eta}'=(\eta',\overline{\eta}')$ in $\cR_0$, set
\[
    d_{\cR_0}(\boldsymbol{\eta},\boldsymbol{\eta}') :=\int_0^T\cW_2(\eta_t,\eta'_t)dt+\cW_2(\overline{\eta},\overline{\eta}').
\]
Notice that $d_{\cR_0}(\boldsymbol{\eta},\boldsymbol{\eta}') = 0$ if and only if $\overline{\eta} = \overline{\eta}'$ and $\eta_t = \eta'_t$ $dt$-a.e., i.e. $\eta = \eta'$ in $V_2$.

\begin{definition}[Admissible deviations in the LP formulation]\label{LP:def:admissible_deviations}
A pair of functions $(\kappa,\overline{\kappa}): V_2 \to \cR_0$ is an admissible deviation in the LP formulation if the following holds:
\begin{enumerate}
    \item\label{LP:def:admissible_deviations:dynamics} For any $m \in V_2$, $(\kappa,\overline{\kappa})(m)$ belongs to $\cR[m]$.
    
    \item\label{LP:def:admissible_deviations:independence} The marginal over $A$ of $\kappa(m)$ is constant in $m$, i.e., letting $\kappa_t(dx,da \vert m ) := \kappa_t (m)(dx,da)$, it holds
\begin{equation}\label{LP:deviations:independence}
    \begin{aligned}
        & \int_{\R \times C} \kappa_t(dx,da \vert m_1) = \int_{\R \times C} \kappa_t(dx,da \vert m_2) \quad \forall\, C \in \cB_{A}, \, \forall\, m_1, m_2 \in V_2, \, dt\text{-a.e. }t\in[0,T].
    \end{aligned}
    \end{equation}
    \item \label{LP:def:admissible_deviations:stability}
     There exists a constant $C_{\kappa}>0$ such that, for every $m,m'\in V_2$,
    \[
        d_{\cR_0}\big((\kappa(m),\overline{\kappa}(m)),(\kappa(m'),\overline{\kappa}(m'))\big)\leq C_{\kappa}d_b(m,m').
    \]
    \item\label{LP:def:admissible_deviations:continuity} The map
    \[
        \cR_0 \ni (m,\overline m) \mapsto (\kappa(m),\overline{\kappa}(m)) \in \cR_0
    \]
    is continuous from $(\cR_0,\tau^{(2)}\otimes\tau^{(2)})$ into itself.
\end{enumerate}
The set of all admissible deviations is denoted by $\cD$.
\end{definition}
Properties \eqref{LP:def:admissible_deviations:dynamics}, \eqref{LP:def:admissible_deviations:independence} and \eqref{LP:def:admissible_deviations:stability} are the LP analogue of the deviation constraint for CCEs in the probabilistic formulation.
In particular, property \eqref{LP:def:admissible_deviations:dynamics} means that the deviation is dynamically admissible: once the flow $m$ is fixed, the deviating flow must be compatible with the dynamics associated with $m$.
Property \eqref{LP:def:admissible_deviations:independence} requests that the deviating player fixes the distribution of alternative actions without observing the outcome of the moderator's lottery, and therefore the $A$-marginal of the deviation cannot depend on the realized flow $m$.
Property \eqref{LP:def:admissible_deviations:stability} requires that the admissible deviation can depend on the flow $m$ only through the way in which $m$ affects the controlled dynamics.
Regarding the continuity property \eqref{LP:def:admissible_deviations:continuity}, only the induced map on $\cR_0$ is relevant for the LP-CCE condition, because Lemma \ref{LP:lemma:restriction_to_R_0} below shows that every $\rho\in\cM$ is supported on $\cR_0$. We stress that this requirement is not restrictive for deviations coming from the probabilistic formulation, since Proposition \ref{LP:relation:prop_deviations} shows that any admissible deviation $\beta\in\bbA$ induces a pair $(\kappa,\overline{\kappa})\in\cD$ according to Definition \ref{LP:def:admissible_deviations}.

\smallskip
Under Assumption \ref{standing_assumptions}, the terminal costs depend on $m\in\cP_2(\R\times A)$ only through its first marginal. Hence, with a slight abuse of notation, for $\overline{m} \in\cP_2(\R)$, we write
\[
    g(x,\overline{m} ) := \bar g\left(x,\int_{\R}\hat g(y)\overline{m} (dy)\right), \qquad     g^0(x,\overline{m} ) := \bar g^0\left(x,\int_{\R}\hat g^0(y)\overline{m} (dy)\right).
\]
For any $\rho \in \cM$ and $(\kappa,\overline{\kappa}) \in \cD$, we introduce the functionals
\begin{align}
    & \Gamma[\rho] :=\int_{V_2 \times \cP_2(\R)} \left( \int_{0}^{T}\int_{\R \times A }f(t,x,m_t, a)m_t(dx,da)dt + \int_{\R}g(x,\overline{m})\overline{m}(dx) \right) \rho(dm,d\overline{m}), \label{LP:repr_player:cost_functional} \\
    & \Gamma^{dev}[\rho](\kappa,\overline{\kappa}) :=\int_{V_2 \times \cP_2(\R)} \bigg( \int_{0}^{T}\int_{\R \times A }f(t,x,m_t, a)\kappa_t(dx,da \vert m)dt \label{LP:dev_player:cost_functional} \\
    & \qquad \qquad \qquad \qquad + \int_{\R} g(x,\overline{m})\overline{\kappa}(dx \vert m )\bigg)\rho(dm,d\overline{m}) \notag.
\end{align}

\begin{definition}[Coarse correlated equilibrium in the LP formulation]\label{LP:def:CCE}
A probability measure $\rho \in \cM$ is a coarse correlated equilibrium for the mean-field game in the linear programming formulation if
\begin{equation}\label{eq:LP:def_CCE}
\Gamma[\rho] \leq \Gamma^{dev}[\rho](\kappa,\overline{\kappa}) \quad \forall (\kappa,\overline{\kappa}) \in \cD.
\end{equation}
We refer to such equilibria also as LP-coarse correlated equilibria and LP-CCEs.
We denote the set of all LP-CCEs by $\cE$.
\end{definition}

The mediator's objective functional  $\Gamma^{0}:\cE \to \R$ is defined by
\begin{multline}\label{LP:mediator_functional}
    \Gamma^{0}(\rho):=\int_{V_2 \times \cP_2(\R)} \bigg( \int_{0}^{T}\int_{\R\times A} f^{0}(t,x,m_t,a) m_t(dx,da) dt \\
    + \int_{\R}g^{0}(x,\overline{m})\overline{m}(dx) \bigg) \rho(dm,d\overline{m}).
\end{multline}
\begin{definition}\label{LP:def:optimal_LP_CCE}
An LP-CCE $\rho^* \in \cE$ is optimal for the moderator's optimization problem in the LP formulation if 
\begin{equation}\label{LP:mediator_value}
    V^{\text{LP}}:=\inf_{\rho \in \cE} \Gamma^{0}(\rho) = \Gamma^{0}(\rho^*).
\end{equation}
We refer to $\rho^*$ as an optimal LP-CCE.
\end{definition}

\smallskip
We remark that our notion of LP-CCEs extends the notion of LP-NE considered in \cite{dumitrescu2021EJP,LinProgFictDumitrescu,leutscher_thesis}.
We recall the analogue of \cite[Definition 3.4]{dumitrescu2021EJP} in our context, so neglecting absorption and optimal stopping:
\begin{definition}[LP formulation of the MFG, Definition 3.4 in \cite{dumitrescu2021EJP}]\label{LP:def:LPNE}
Fix a pair $(m,\overline{m}) \in V_2 \times \cP_2(\R)$, and consider the set $\cR[m]$ as given by Definition \ref{LP:def:pre_martingale_constraint}.
Let $F[m,\overline{m}]:V_2 \times \cP_2(\R) \to \R$ be defined as
\begin{equation}\label{eq:function_F}
    F[m,\overline{m}](\eta,\overline{\eta}) = \int_0^T \int_{\R \times A} f(t,x,m_t,a)\eta_t(dx,da)dt + \int_{\R} g(x,\overline{m})\overline{\eta}(dx).
\end{equation}
We say that $(m^\star,\overline{m}^\star) \in V_2 \times \cP_2(\R)$ is an LP-MFG Nash equilibrium (or simply LP-NE) if $(m^\star,\overline{m}^\star) \in \cR[m^\star]$ and 
\begin{equation}\label{LP:def:LPNE:optimality}
F[m^\star,\overline{m}^\star](m^\star,\overline{m}^\star)  \leq F[m^\star,\overline{m}^\star](\eta,\overline{\eta}) \quad \forall\, (\eta,\overline{\eta}) \in \cR[m^\star].
\end{equation}
\end{definition}

\begin{proposition}\label{LP:prop:NE_is_CCE}
Let $(m^\star,\overline{m}^\star) \in V_2 \times \cP_2(\R)$ be an LP-NE. Then, $\delta_{(m^\star,\overline{m}^\star)}(dm,d\overline{m})$ is an LP-CCE.
\end{proposition}
\begin{proof}
Set $\rho(dm,d\overline{m}) = \delta_{(m^\star,\overline{m}^\star)}(dm,d\overline{m})$.
The consistent martingale property \eqref{eq:consistent_mtg_property} reduces to property \eqref{eq:def:pre_martingale_property} with fixed flow of measures $m^\star \in V_2$.
As $(m^\star,\overline{m}^\star) \in \cR[m^\star]$, we have $\rho \in \cM$ and the cost $\Gamma[\rho]$ reduces to $F[m^\star,\overline{m}^\star](m^\star,\overline{m}^\star)$.
Let $(\kappa,\overline{\kappa}) \in \cD$.
By definition, $(\kappa,\overline{\kappa})(m^\star) \in \cR[m^\star]$, so that the optimality property for the LP-NE \eqref{LP:def:LPNE:optimality} implies
\begin{equation*}
\Gamma[\rho] = F[m^\star,\overline{m}^\star](m^\star,\overline{m}^\star) \leq F[m^\star,\overline{m}^\star](\kappa(m^\star),\overline{\kappa}(m^\star)) = \Gamma^{dev}[\rho](\kappa,\overline{\kappa}),
\end{equation*}
where, in the second equality, we exploited again the fact that $\rho$ is equal to a Dirac delta.
As $(\kappa,\overline{\kappa})$ is arbitrary in $\cD$, this concludes the proof.
\end{proof}

\subsection{Relation with coarse correlated solution of the MFG}\label{sec:LP:relation}
In this section, we clarify the relation between LP-CCEs in the sense of Definition \ref{LP:def:CCE} and coarse correlated solutions to the MFG in the sense of Definition \ref{def:CCE}.
We start by showing that, starting from a correlated flow which satisfies the consistency condition \eqref{eq:CCE:consistency}, it is possible to define a measure $\rho$ that satisfies the consistent martingale property.

\begin{proposition}\label{LP:relation:prop:consistent_strategies}
Let $(\lambda,\mu)$ be a correlated measure flow in the sense of Definition \ref{probabilistic:def_corr_flow} and suppose that the consistency condition \eqref{eq:CCE:consistency} is satisfied.
Then, $\rho = \P \circ (\mu,\mu^x_T)^{-1} \in \cM$, where $\mu^x_T$ denotes the marginal over $\R$ of $\mu_T$.
\end{proposition}
\begin{proof}
We first notice that $\rho$ is indeed a probability measure on $V_2\times\cP_2(\R)$.
Indeed, identify the measurable flow $\mu$ with its $dt$-a.e. equivalence class in $V_2$.
This identification is measurable thanks to the joint measurability assumption in Definition \ref{probabilistic:def_corr_flow}: for every continuous function $g:[0,T]\times\R\times A\to\R$ with at most quadratic growth in $x$, uniformly in $(t,a)$, the map
\[
    \omega\mapsto \int_0^T\int_{\R\times A} g(t,x,a)\mu_t(\omega)(dx,da)dt
\]
is $\cF_0$-measurable.
Moreover, by the consistency condition \eqref{eq:CCE:consistency} and the standard moment estimate for \eqref{strong_formulation:repr_player:dynamics},
\[
    \E\left[\int_0^T\int_{\R\times A}|x|^2\mu_t(dx,da)dt\right]=\E\left[\int_0^T|X_t|^2dt\right]<\infty.
\]
Hence $\mu$ is $V_2$-valued, up to the $dt$-a.e. identification.
Since $\mu_T^x$ is an $\cF_0$-measurable $\cP_2(\R)$-valued random variable, $\rho=\P\circ(\mu,\mu_T^x)^{-1}$ belongs to $\cP(V_2 \times \cP_2(\R))$.

Let $X$ be the state of the representative player when she plays the recommendation $\lambda$.
Notice that the consistency condition \eqref{eq:CCE:consistency} implies that, for any pair of functions $\phi: V_2 \times \cP_2(\R) \to \R$, $\psi:\R \times A \to \R$ bounded and measurable, it holds
\begin{align*}
    \E[\phi(\mu,\mu^x_T)\psi(X_t,\lambda_t)] &= \E[\phi(\mu,\mu^x_T) \E[\psi(X_t,\lambda_t) \vert \mu] ] =   \E\left[ \int_{\R \times A}\psi(x,a)\mu_t(dx,da)\phi(\mu,\mu^x_T) \right] \\
   & = \int_{V_2 \times \cP_2(\R)} \int_{\R \times A}\psi(x,a)m_t(dx,da)\phi(m,\overline{m}) \rho(dm,d\overline{m}).
\end{align*}
By the arbitrariness of $\phi$ and $\psi$, it follows that the joint law of $(X_t,\lambda_t,\mu,\mu^x_T)$ at any time $t$ can be expressed in terms of $\rho$ as
\begin{equation}\label{LP:relation:prop:consistent_strategies:consistency_laws}
    \P\circ(X_t,\lambda_t,\mu,\mu^x_T)^{-1}(dx,da,dm,d\overline{m}) = m_t(dx,da) \rho(dm,d\overline{m}).
\end{equation}
Analogously, we have
\[
\P\circ(X_T,\mu,\mu^x_T)^{-1}(dx,dm,d\overline{m}) = \overline{m}(dx) \rho(dm,d\overline{m}).
\]
Let $u \in \dC^{1,2}_b([0,T] \times \R)$ and $\phi \in \dC_b(V_2 \times \cP_2(\R))$.
As $\partial_x u$ is bounded, by applying It\^{o}'s formula and taking expectations, we have
\begin{align*}
    \E & \Big[ \phi(\mu,\mu^x_T) \Big( \int_{\R}u(T,x)\mu^{x}_T(dx) - \int_{\R}u(0,x)m^{*}_{0}(dx) \Big) \Big] = \E[\phi(\mu,\mu^x_T) \E[ u(T,X_T) - u(0,\xi) \vert \mu ]  ]  \\
    & = \E [\phi(\mu,\mu^x_T) \big( u(T,X_T) - u(0,\xi) \big)] = \E\left[ \phi(\mu,\mu^x_T) \int_0^T (\partial_t + \cL )u(t,X_t,\mu_t,\lambda_t)dt \right] \\
    & = \int_0^T\E\left[ \phi(\mu,\mu^x_T) \E\left[ (\partial_t + \cL )u(t,X_t,\mu_t,\lambda_t) \Big\vert \mu \right] \right]dt \\
    & =  \int_0^T \E\left[ \phi(\mu,\mu^x_T) \int_{\R \times A} (\partial_t + \cL )u(t,\cdot,\mu_t,\cdot) \mu_t(dx,da) \right]dt  \\
    & =   \E\left[ \phi(\mu,\mu^x_T) \int_0^T\int_{\R \times A} (\partial_t + \cL )u(t,\cdot,\mu_t,\cdot) \mu_t(dx,da) dt \right] \\
    & = \int_{V_2 \times \cP_2(\R)} \phi(m,\overline{m}) \left( \int_0^T \int_{\R \times A} (\partial_t + \cL )u(t,x,m_t,a)m_t(dx,da)dt \right) \rho(dm,d\overline{m})
\end{align*}
where we applied Fubini's theorem and \eqref{LP:relation:prop:consistent_strategies:consistency_laws}. Hence, we have \eqref{eq:consistent_mtg_property} and we conclude.
\end{proof}
Notice that, in Proposition \ref{LP:relation:prop:consistent_strategies}, whenever a correlated flow $(\lambda,\mu)$ in the probabilistic formulation is considered in the LP formulation, we identify the measurable flow $(\mu_t)_{t\in[0,T]}$ with its $dt$-a.e. equivalence class in $V_2$. The terminal marginal $\mu^x_T$ is kept separately, since it is not encoded in this equivalence class.

Turning to the admissible deviations, we show that any $\beta \in \bbA$ induces an admissible deviation in the LP formulation $(\kappa,\overline{\kappa})$.
\begin{proposition}\label{LP:relation:prop_deviations}
Let $\beta \in \bbA$ be an admissible deviation.
For any $m \in V_2$, let $X^m$ be the unique strong solution of
\begin{equation}\label{eq:relation:prop_deviations:eq_fixed_flow}
dX^m_t = b(t,X^m_t,m_t,\beta_t)dt + \sigma dW_t, \quad X^m_0 = \xi.
\end{equation}
Define $(\kappa,\overline{\kappa})$ by
\begin{equation}\label{eq:relation:prop_deviations:kernels_fixed_flow}
\begin{aligned}
    \kappa_t(C \times D \vert m) & = \E[ \1_{C}(X^m_t)\1_{D}(\beta_t)], \quad && C \in \cB_{\R}, \, D \in \cB_{A}, \\
    \overline{\kappa}( C \vert m) & = \E[ \1_{C}(X^m_T)], && C \in \cB_{\R}.
\end{aligned}
\end{equation}
Then, $(\kappa,\overline{\kappa})$ is an admissible deviation in the sense of Definition \ref{LP:def:admissible_deviations}.
\end{proposition}
\begin{proof}
By Assumption \ref{standing_assumptions}, for any $m \in V_2$ there exists a unique strong solution of \eqref{eq:relation:prop_deviations:eq_fixed_flow}, so that the family of kernels \eqref{eq:relation:prop_deviations:kernels_fixed_flow} is well defined.
Consider the map $(\kappa,\overline{\kappa}):V_2 \to V_2 \times \cP_2(\R)$ so that $\kappa(m) = (\kappa_t(dx,da \vert m) )_{t \in [0,T]}$ and $\overline{\kappa}(m) = \overline{\kappa}(dx \vert m)$.
We show that it satisfies the requirements of Definition \ref{LP:def:admissible_deviations}.
As for point \eqref{LP:def:admissible_deviations:dynamics}, let $u \in \dC^{1,2}_b([0,T] \times \R)$.
By It\^{o}'s formula and Fubini's theorem, exploiting the boundedness of $\partial_x u$ and $\sigma$, we deduce
\begin{multline*}
    \int_{\R} u(T,x) \overline{\kappa} (dx\vert m) - \int_{\R} u(0,x) m_0^*(dx) = \E [u(T,X^m_T)] - \E[u(0,\xi)] \\
    = \int_0^T \E\left[  (\partial_t + \cL )u(t,X^m_t,m_t,\beta_t)\right] dt =  \int_0^T \int_{\R \times A} (\partial_t + \cL )u(t,x,m_t,a)\kappa_t(dx,da \vert m) dt,
\end{multline*}
so that $\kappa(m)$ satisfies \eqref{eq:def:pre_martingale_property} for any $m \in V_2$ and thus $(\kappa,\overline{\kappa})(m) \in \cR[m]$.
Property \eqref{LP:def:admissible_deviations:independence} is straightforward.
As for property \eqref{LP:def:admissible_deviations:stability}, fix $m,m'\in V_2$ and set $\Delta_t:=X^m_t-X^{m'}_t$. By coupling the laws $\kappa_t(\cdot|m)$ and $\kappa_t(\cdot|m')$ through the joint law of $(X^m_t,\beta_t,X^{m'}_t,\beta_t)$, we get
\[
    \cW_2(\kappa_t(\cdot|m),\kappa_t(\cdot|m'))\leq \E[|\Delta_t|^2]^{1/2}, \qquad \cW_2(\overline{\kappa}(\cdot|m),\overline{\kappa}(\cdot|m'))\leq \E[|\Delta_T|^2]^{1/2}.
\]
Moreover,
\[
    \Delta_t=\int_0^t\left(b(s,X^m_s,m_s,\beta_s)-b(s,X^{m'}_s,m'_s,\beta_s)\right)ds.
\]
By adding and subtracting $b(s,X^{m'}_s,m_s,\beta_s)$, we obtain
\[
\begin{aligned}
    &\left|b(s,X^m_s,m_s,\beta_s)-b(s,X^{m'}_s,m'_s,\beta_s)\right| \\
    &\leq \left|b(s,X^m_s,m_s,\beta_s)-b(s,X^{m'}_s,m_s,\beta_s)\right|+\left|b(s,X^{m'}_s,m_s,\beta_s)-b(s,X^{m'}_s,m'_s,\beta_s)\right|.
\end{aligned}
\]
By the Lipschitz continuity of $b$ in the state variable and by the definition of $d_b$, for $dt$-a.e. $s\in[0,T]$,
\[
    \left|b(s,X^m_s,m_s,\beta_s)-b(s,X^{m'}_s,m'_s,\beta_s)\right|\leq L|\Delta_s|+d_b(m,m').
\]
Therefore, Gronwall's lemma gives
\[
    \sup_{t\in[0,T]}|\Delta_t|\leq C d_b(m,m')\quad \P\text{-a.s.}
\]
In particular,
\[
    \sup_{t\in[0,T]}\E[|\Delta_t|^2]^{1/2}\leq C d_b(m,m').
\]
Consequently,
\[
    d_{\cR_0}\big((\kappa(m),\overline{\kappa}(m)),(\kappa(m'),\overline{\kappa}(m'))\big) \leq C d_b(m,m'),
\]
up to changing the constant $C$. This proves condition \eqref{LP:def:admissible_deviations:stability}.
It remains to verify the continuity property \eqref{LP:def:admissible_deviations:continuity}.
Let $((m^n,\overline m^n))_{n\geq1}$ and $(m,\overline m)$ be elements of $\cR_0$ such that $(m^n,\overline m^n)\to(m,\overline m)$ in $\tau^{(2)}\otimes\tau^{(2)}$.
Let $q_t(da) = \delta_{\beta_t}(da)$ be the relaxed control associated to the deviation $\beta$.
Then, by \cite[Lemma 3.9]{dumitrescu2021EJP}, it holds
\begin{equation}\label{eq:LP:relation:prop_deviations:strong_convergence}
\E \left[\sup_{t \in [0,T]} \vert X^{m^n}_t - X^m_t \vert^2 \right] \overset{n \to \infty}{\longrightarrow} 0,
\end{equation}
Let $\varphi:[0,T]\times\R\times A\to\R$ be continuous with at most quadratic growth in $x$, uniformly in $(t,a)$.
By definition of $\kappa$, we have
\[
    \int_0^T\int_{\R\times A}\varphi(t,x,a)\kappa_t(dx,da\vert m^n)dt=\E\left[\int_0^T\varphi(t,X^{m^n}_t,\beta_t)dt\right].
\]
We notice that \eqref{eq:LP:relation:prop_deviations:strong_convergence} implies $X^{m^n}\to X^m$ in probability uniformly on $[0,T]$.
Hence
\[
    \int_0^T\left|\varphi(t,X^{m^n}_t,\beta_t)-\varphi(t,X^m_t,\beta_t)\right|dt\to0
\]
in probability.
By the quadratic growth of $\varphi$, the compactness of $A$, and standard moment estimates for SDEs with bounded drift, the family $(\int_0^T\left|\varphi(t,X^{m^n}_t,\beta_t)\right|dt )_{n \geq 1}$ is uniformly integrable.
Indeed, $m_0^*\in\cP_{\bar{p}}(\R)$ with $\bar{p}>2$ gives
\[
    \sup_{n\geq1}\E\left[\sup_{t\in[0,T]}|X^{m^n}_t|^{\bar p}\right]<\infty.
\]
Thus, Vitali's theorem gives
\[
    \E\left[\int_0^T\left|\varphi(t,X^{m^n}_t,\beta_t)-\varphi(t,X^m_t,\beta_t)\right|dt\right]\to0.
\]
Hence $\kappa(m^n)\to\kappa(m)$ in $\tau^{(2)}$.
The same argument applied at terminal time implies $\overline{\kappa}(m^n)\to\overline{\kappa}(m)$ in $\tau^{(2)}$.
\end{proof}

The next proposition is a technical result which is needed in the proof of Theorem \ref{theorem: equivalence of formulations} (see Step 2 therein).

\begin{proposition}\label{LP:relation:prop:full_martingale_property}
Assume, without loss of generality, that $(\Omega,\cF,\P)$ is Polish.
Let $(\lambda,\mu)$ be a correlated measure flow with $\rho = \P \circ (\mu,\mu^x_T)^{-1}$, and $\beta$ an admissible deviation.
Let $X^\beta$ be the solution of \eqref{strong_formulation:dev_player:dynamics} with deviation $\beta$ and flow of measures $\mu$.
Then, there exists an admissible deviation $(\kappa,\overline{\kappa}) \in \cD$ so that
\begin{equation}\label{eq:relation:prop:deviation:disintegration}
\begin{aligned}
    \E[\1_B(X^\beta_t) \1_C(\beta_t) \1_D(\mu,\mu^x_T)] & = \int_D \kappa_t(B \times C \vert m) \rho(dm,d\overline{m}),  && \forall B \in \cB_{\R}, \, C \in \cB_{A}, \, D \in \cB_{V_2 \times \cP_2(\R)}, \\
    \E[\1_B(X^\beta_T) \1_D(\mu,\mu^x_T)] & = \int_D \overline{\kappa}(B \vert m) \rho(dm,d\overline{m}), && \forall B \in \cB_{\R}, \, D \in \cB_{V_2 \times \cP_2(\R)}.
\end{aligned}
\end{equation}
\end{proposition}
\begin{proof}
Let $\P^m(\cdot):=\P(\cdot \, \vert \, \mu=m)$ be a regular conditional probability of $\P$ given $\mu$.
Since $\mu$ is independent of $(\xi,W,\beta)$, under $\P^m$ the law of $(\xi,W,\beta)$ coincides with its law under $\P$, for $\P\circ\mu^{-1}$-a.e. $m\in V_2$.
For every $m\in V_2$, let $X^m$ be the unique strong solution of \eqref{eq:relation:prop_deviations:eq_fixed_flow} and define $(\kappa,\overline\kappa)$ as in \eqref{eq:relation:prop_deviations:kernels_fixed_flow}. By Proposition
\ref{LP:relation:prop_deviations}, $(\kappa,\overline\kappa)\in\cD$.
We now identify these kernels with the conditional laws of the deviating state process.
Under $\P^m$, the process $X^\beta$ satisfies
\[
dX^\beta_t=b(t,X^\beta_t,m_t,\beta_t)dt+\sigma dW_t,
\qquad X^\beta_0=\xi.
\]
Therefore, by uniqueness in law for the fixed-flow equation, for every
$t\in[0,T]$ and for $\rho$-a.e. $(m,\overline{m} )$,
\[
    \kappa_t(B\times C\vert m) = \P^m\big((X^\beta_t,\beta_t)\in B\times C\big) = \E\left[\1_B(X^\beta_t)\1_C(\beta_t)\mid \mu=m\right],
\]
and similarly
\[
    \overline\kappa(B\vert m) = \P^m\big(X^\beta_T\in B\big) = \E\left[\1_B(X^\beta_T)\mid \mu=m\right].
\]
Let now $D\in\cB_{V_2\times\cP_2(\R)}$. Since $\rho=\P\circ(\mu,\mu^x_T)^{-1}$, we obtain
\begin{multline*}
    \int_D \kappa_t(B\times C\vert m)\rho(dm,d\overline{m} ) = \E\left[ \1_D(\mu,\mu^x_T)\kappa_t(B\times C\vert \mu) \right] \\
    = \E\left[\1_D(\mu,\mu^x_T) \E\left[\1_B(X^\beta_t)\1_C(\beta_t)\mid \mu\right] \right] = \E\left[ \1_B(X^\beta_t)\1_C(\beta_t)\1_D(\mu,\mu^x_T) \right].
\end{multline*}
This proves the first identity in \eqref{eq:relation:prop:deviation:disintegration}. The second one follows in the same way, using the conditional law of $X^\beta_T$.
This concludes the proof.
\end{proof}

\section{Solution to the LP problem}\label{section: solution to LP problem}
The goal of this section is to prove the following result:
\begin{theorem}\label{LP:thm:solution_moderator_optimization_problem}
Under Assumptions \ref{standing_assumptions}, there exists an optimal LP-CCE for the moderator's optimization problem.
\end{theorem}

The roadmap for proving Theorem \ref{LP:thm:solution_moderator_optimization_problem} is the following: we first prove in Lemma \ref{LP:lemma:restriction_to_R_0} that $\cM \subseteq \cP(\cR_0)$.
Thus, we consider $\cM$ with the weak topology of convergence of probability measures on $\cR_0$.
Notice that, since $\cR_0$ is compact, we have $\dC_b(\cR_0) = \dC(\cR_0)$, since any continuous function on the compact set $\cR_0$ is also bounded.
Then, in Lemma \ref{lemma:M_convex_compact}, we prove that the set $\cM$ is compact and convex, and in Lemma \ref{lemma:continuity} we prove that the functionals $\rho \mapsto \Gamma[\rho]$ and $\rho \mapsto \Gamma^{dev}[\rho](\kappa,\overline{\kappa})$ are continuous.
This will imply that the set of CCEs $\cE$ is compact and convex as well.
As any LP-NE is also an LP-CCE by Proposition \ref{LP:prop:NE_is_CCE} and there exists at least one LP-NE by \cite[Theorem 3.4]{leutscher_thesis}, $\cE$ is nonempty.
By relying on the continuity of the mediator's payoff $\Gamma^0$, this will imply the existence of an optimal LP-CCE for the mediator.

\begin{lemma}\label{LP:lemma:continuity_G_u}
For any $u \in C^{1,2}_b([0,T]\times\R)$, define the map $\cG^u:V_2 \times \cP_2(\R) \to \R$ by
\begin{multline}\label{eq:LP:map_G_u}
    \cG^u(m,\overline{m}):= \int_{\R } u(T,x) \overline{m}(dx) -\int_{\R }u(0,x)m_0^*(dx) \\
    - \int_{0}^{T}\int_{\R \times A} (\partial_t+\cL ) u(t,x,m_t,a)m_{t}(dx,da)dt.
\end{multline}
Then, the restriction of $\cG^u$ to $\cR_0$ is a continuous map from $(\cR_0,\tau^{(2)}\otimes\tau^{(2)})$ to $\R$.
\end{lemma}
\begin{proof}
Let $(m^n,\overline m^n),(m,\overline m)\in\cR_0$ be such that $(m^n,\overline m^n)\to(m,\overline m)$ in $\tau^{(2)}\otimes\tau^{(2)}$.
By Lemma \ref{lemma:compactness_a_priori_R0}, the same convergence holds with respect to $\bar\tau^{(2)}\otimes\tau^{(2)}$.
To show the continuity of $\cG^u(m,\overline{m})$, the only problematic term is given by
\[
    \int_0^T \int_{\R \times A} b(t,x,m^n_t,a) \partial_x u(t,x) m^n_t(dx,da)dt.
\]
Since $m^n \to m$ in $\overline{\tau}^{(2)}$, \cite[Lemma 2.11]{dumitrescu2021EJP} implies
\[
    \psi^n(t):=\int_{\R\times A}\hat{b}(t,y)m^n_t(dy,du) \to \psi(t):=\int_{\R\times A}\hat{b}(t,y)m_t(dy,du)
\]
in $L^1([0,T])$.
Since $\bar{b}(t,x,y,a)$ is bounded by assumption and $m^n \to m$ in $\overline{\tau}^{(2)}$, \cite[Lemma F.2]{dumitrescu2021EJP} implies
\begin{multline*}
    \int_0^T \int_{\R \times A} \bar b\big(t,x,\psi^n(t),a\big)\partial_x u(t,x)m^n_t(dx,da)dt \\
    \to \int_0^T \int_{\R \times A} \bar b\big(t,x,\psi(t),a\big)\partial_x u(t,x)m_t(dx,da)dt.
\end{multline*}
This implies the continuity of the map $\cG^u(m,\overline{m})$.
\end{proof}

\begin{lemma}\label{LP:lemma:restriction_to_R_0}
Let $\rho \in \cM$.
We have $\mathrm{supp}(\rho) \subseteq \cR_0$ and $ (m,\overline{m}) \in \cR[m]$ for any $(m,\overline{m}) \in \mathrm{supp}(\rho)$.
Consequently, $\cM \subseteq \cP(\cR_0)$.
\end{lemma}
\begin{proof}
For any $u \in C^{1,2}_b([0,T]\times\R)$, consider the map $\cG^u$ defined by \eqref{eq:LP:map_G_u}.
Let $\mathcal{U} \subseteq \dC_c^{1,2}([0,T]\times\R)$ be a countable determining class for the martingale property, i.e. such that the identities $\cG^u(m,\overline{m} )=0$ for every $u\in\mathcal{U}$ imply $\cG^u(m,\overline{m} )=0$ for every $u\in\dC_b^{1,2}([0,T]\times\R)$.
Fix $u\in\mathcal{U}$.
By the definition of $\cM$, for every $\phi\in\dC_b(V_2\times\cP_2(\R),\tau^{(2)} \otimes \tau^{(2)})$,
\[
\int_{V_2\times\cP_2(\R)}\phi(m,\overline{m} )\cG^u(m,\overline{m} )\rho(dm,d\overline{m} )=0.
\]
Since $\cG^u$ is measurable, this implies $\cG^u(m,\overline{m} )=0$ for $\rho$-a.e. $(m,\overline{m} )$.
Intersecting the corresponding full-measure sets over $u\in\mathcal{U}$, we obtain a set of full $\rho$-measure on which $\cG^u(m,\overline{m} )=0$ for every $u\in\mathcal{U}$.
By the determining property of $\mathcal{U}$, it follows that $\cG^u(m,\overline{m} )=0$ for every $u\in\dC_b^{1,2}([0,T]\times\R)$, for $\rho$-a.e. $(m,\overline{m} )$.
Hence $(m,\overline{m} )\in\cR[m]$ for $\rho$-a.e. $(m,\overline{m} )$, and therefore $(m,\overline{m} )\in\cR[m]\subseteq\cR_0$ for $\rho$-a.e. $(m,\overline{m} )$.
This implies $\rho(\cR_0) = 1$. As $\cR_0$ is compact in $\tau^{(2)} \otimes \tau^{(2)}$, it is closed, and thus we conclude $\mathrm{supp}(\rho) \subseteq \cR_0$.
It remains to prove that every point $(m,\overline m)\in \mathrm{supp}(\rho)$ belongs to $\cR[m]$.
Fix $u\in\dC_b^{1,2}([0,T]\times\R)$.
By the same argument as above, $\cG^u(m,\overline{m} )=0$ for $\rho$-a.e. $(m,\overline{m} )$.
Since the topologies $\tau^{(2)}\otimes\tau^{(2)}$ and $\bar\tau^{(2)}\otimes\tau^{(2)}$ coincide on $\cR_0$, Lemma \ref{LP:lemma:continuity_G_u} implies that the restriction of $\cG^u$ to $\cR_0$ is continuous with respect to $\tau^{(2)}\otimes\tau^{(2)}$.
As $\mathrm{supp}(\rho)\subseteq\cR_0$, continuity yields $\cG^u(m,\overline{m} )=0$ for every $(m,\overline{m} )\in\mathrm{supp}(\rho)$.
Since $u\in\dC_b^{1,2}([0,T]\times\R)$ was arbitrary, we conclude that $(m,\overline{m} )\in\cR[m]$ for every $(m,\overline{m} )\in\mathrm{supp}(\rho)$.
\end{proof}

\begin{lemma}\label{lemma:M_convex_compact}
The set $\cM$ is convex and compact.
\end{lemma}
\begin{proof}
Convexity is obvious from linearity of equation \eqref{eq:consistent_mtg_property} with respect to $\rho$.
As for compactness, it is enough to recall that $\cM \subseteq \cP(\cR_0)$, with $\cR_0$ being compact in $\tau^{(2)} \otimes \tau^{(2)}$. Thus, every sequence in $\cM$ is tight, which implies that $\cM$ is relatively compact.
We just have to show that $\cM$ is closed.
Let $u \in C^{1,2}_b([0,T]\times\R)$ and recall the definition of $\cG^u$ in \eqref{eq:LP:map_G_u}.
Since $\cG^u$ is continuous on $(V_2\times\cP_2(\R),\overline{\tau}^{(2)}\otimes\tau^{(2)})$, it is in particular continuous on $(\cR_0,\overline{\tau}^{(2)}\otimes\tau^{(2)})$, and so it is continuous on  $(\cR_0,\tau^{(2)}\otimes\tau^{(2)})$.
This is enough to conclude that, for any $\phi \in \dC_b(\cR_0)$,
\[
0 = \lim_{n \to \infty} \int_{\cR_0} \phi(m,\overline{m})\cG^u(m,\overline{m}) \rho^n(dm,d\overline{m}) = \int_{\cR_0} \phi(m,\overline{m}) \cG^u(m,\overline{m}) \rho(dm,d\overline{m})
\]
for any $(\rho^n)_{n\geq 1} \subseteq \cM$ converging to $\rho$ in the weak topology on $\cR_0$, where the first equality follows from the consistent martingale property in Definition \ref{LP:def:consistent_mtg_property}.
This is enough to conclude the closedness of $\cM$ and thus the compactness in the weak topology.
\end{proof}

The following lemma extends \cite[Lemma 2.11]{dumitrescu2021EJP} to the unbounded case, provided that we request some additional regularity.
\begin{lemma}\label{lemma:L1_convergence}
Let $h \in \dC^{1,2}([0,T] \times \R)$ such that its derivatives are bounded.
Let $((m^n,\overline{m}^n))_{n \geq 1}$, $(m,\overline{m}) \in \cR_0$ such that $m^n \to m$ in $\tau^{(2)}$.
Then, it holds
\begin{equation}\label{eq:L1_convergence}
    \int_{\R \times A}h(\cdot,x)m^n_\cdot(dx,da) \to \int_{\R \times A}h(\cdot,x)m_\cdot(dx,da) 
\end{equation}
in $L^1([0,T])$.
\end{lemma}
\begin{proof}
Define the maps
\[
    H^n(t) := \int_{\R \times A} h(t,x)m^n_t(dx,da), \quad H(t) := \int_{\R \times A} h(t,x)m_t(dx,da).
\]
Denote the set of $\dC^1$ functions with compact support in $(0,T)$ by $\dC^1_c((0,T))$.
We recall that, for any function $\phi:[0,T] \to \R$, the bounded variation norm of $\phi$ in $(0,T)$ is defined by
\[
\Vert \varphi \Vert_{BV((0,T))} := \Vert \varphi \Vert_{L^1([0,T])} + V(\varphi,(0,T)),
\]
where
\[
V(\varphi,(0,T)) :=\sup\left\{ \int_0^T \psi'(t)\phi(t)dt: \, \psi \in \dC^1_c((0,T)), \, \Vert \psi \Vert_\infty \leq 1 \right\},
\]
We claim that $(H^n)_{n \geq 1}$ are functions of bounded variation on the open interval $(0,T)$ such that
\begin{equation}\label{eq:uniform_BV}
    \sup_{n \geq 1} \Vert H^n \Vert_{BV((0,T))} < \infty.
\end{equation}
To see our claim, first notice that, since $h(t,x)$ has bounded derivatives then it has linear growth in $x$ uniformly in $t$. Therefore, we have
\begin{multline*}
    \int_0^T \vert H^n(t) \vert dt = \int_0^T \left\vert \int_{\R \times A} h(t,x)m^n_t(dx,da) \right\vert dt \leq \int_0^T \int_{\R \times A} \left\vert h(t,x)\right\vert m^n_t(dx,da)  dt \\
    \leq C\left( 1 + \int_0^T \int_{\R \times A} \vert x \vert m^n_t(dx,da)dt  \right) \leq C,
\end{multline*}
where the last inequality follows from \eqref{uniform_bound} in Lemma \ref{lemma:compactness_a_priori_R0}.
This shows that the sequence is uniformly bounded in $L^1$ with respect to $n$.
To show that $V( H^n ,(0,T)) \leq C$ for any $n \geq 1$, with $C$ independent on $n$, we first suppose that, for each $n \geq 1$, $(m^n,\overline{m}^n) \in \cR[\widetilde{m}^n]$, $\widetilde{m}^n \in V_2$.
In virtue of the uniform bound \eqref{uniform_bound} on the moments of $m \in \cR_0$, the martingale constraint \eqref{eq:def:pre_martingale_property} holds true for any $u \in \dC^{1,2}([0,T] \times \R)$ with linear growth in $x$ uniformly in $t$.
Take $\psi \in \dC^1_c((0,T))$. Exploiting the boundedness of the coefficient $b$ and of the derivatives of $h$ and the fact that $\psi(0) = \psi(T) = 0$, the martingale constraint to $u(t,x) = -\psi(t) h(t,x)$ yields, for any $n \geq 1$,
\begin{multline}\label{eq:lemma:L1_convergence:bound_psi}
    \int_0^T \psi'(t) H^n(t)dt = \int_0^T \psi'(t) \int_{\R \times A} h(t,x)m^n_t(dx,da)dt \\
    = - \int_0^T \int_{\R \times A} \psi(t)(\partial_t h(t,x) + \cL h(t,x,\widetilde{m}^n_t,a))m^n_t(dx,da)dt \\
    + \int_{\R }\psi(T)h(T,x)\overline{m}^n(dx) - \int_{\R} \psi(0)h(0,x)m^*_0(dx) \leq C \Vert \psi \Vert_\infty.
\end{multline}
This shows $V( H^n ,(0,T)) \leq C$ for any $n \geq 1$, with $C$ independent of $n$, i.e., that  \eqref{eq:uniform_BV} holds true for any $(m^n,\overline{m}^n) \in \cup_{m \in V_2} \cR[m]$.
By linearity, it holds true also on the convex hull of $\cup_{m \in V_2} \cR[m]$.
It remains to show that the bound $V( H^n ,(0,T)) \leq C$ extends to the $\tau^{(2)} \otimes \tau^{(2)}$ closure of the convex hull of $\cup_{m \in V_2} \cR[m]$.
For any $n \geq 1$, there exists $((\eta^k,\overline{\eta}^k))_{k \geq 1}$ in the convex hull of $\cup_{m \in V_2} \cR[m]$ such that $(\eta^k,\overline{\eta}^k) \to (m^n,\overline{m}^n)$ in $\tau^{(2)} \otimes \tau^{(2)}$.
Let $\psi \in \dC^1_c((0,T))$.
We have
\begin{multline*}
    \int_0^T \psi'(t) H^n(t) dt = \int_0^T \int_{\R \times A} \psi'(t) h(t,x) m^n_t(dx,da) dt \\
    = \lim_{k \to \infty} \int_0^T \int_{\R \times A} \psi'(t) h(t,x) \eta^k_t(dx,da) dt \leq C \Vert \psi \Vert_\infty
\end{multline*}
where we have used the fact that $\psi'(t) h(t,x)$ is an admissible test function for convergence in $\tau^{(2)}$, and the bound \eqref{eq:lemma:L1_convergence:bound_psi} established in the convex hull of $\cup_{m \in V_2} \cR[m]$.
By taking the supremum over $\psi \in \dC^1_c((0,T))$ so that $\Vert \psi \Vert_\infty \leq 1$, \eqref{eq:uniform_BV} follows.

Having established \eqref{eq:uniform_BV}, we show the claimed convergence of $H^n$ to $H$ in $L^1([0,T])$.
To this extent, it is enough to show that from an arbitrary subsequence we can extract a further subsequence that converges to $H(t)$ in $L^1([0,T])$.
Since the bounded variation norm is uniformly bounded, by \cite[Theorem 3.23]{ambrosio2000boundedvariation}, there exists a subsequence $n_k$ and a function $z$ such that $(H^{n_k})_{n_k \geq 1}$ converges to $z$ in $L^1([0,T])$.
To show that $z(t) = H(t)$ for $dt$-a.e. $t \in [0,T]$, we notice that, for any $\phi \in \dC_b([0,T])$, it holds
\begin{multline*}
    \int_0^T \phi(t) z(t) dt = \lim_{k \to \infty} \int_0^T \phi(t) H^{n_k}(t) = \lim_{k \to \infty} \int_0^T \int_{\R \times A} \phi(t)h(t,x) m^{n_k}_t(dx,da)dt \\
    =  \int_0^T \int_{\R \times A} \phi(t)h(t,x) m_t(dx,da)dt,
\end{multline*}
where the first equality holds since $H^{n_k} \to z$ in $L^1([0,T])$ and the last equality holds since $m^{n_k}$ converges to $m$ in $\tau^{(2)}$ and the function $\phi(t)h(t,x)$ is a test function for convergence in $\tau^{(2)}$.
Since the above equality holds true for any $\phi \in \dC_b([0,T])$, we conclude $z(t) = H (t)$ $dt$-a.e.
Since this reasoning holds for any subsequence, we conclude that the whole sequence converges to the limit in the statement.
\end{proof}

Recall the definition of the function $F[m,\overline{m}](\eta,\overline{\eta})$ in \eqref{eq:function_F} and define 
\begin{equation}\label{eq:cost_functionals:function_F0}
    F^{0}(m,\overline{m}) := \int_{0}^{T}\int_{\R\times A}f^{0}(t,x,m_{t},a)m_{t}(dx,da)dt+\int_{\R}g^{0}(x,\overline{m})\overline{m}(dx),
\end{equation}
so that, for any $\rho \in \cM$ and $(\kappa,\overline{\kappa}) \in \cD$, we have
\begin{multline}\label{eq:cost_functionals_integrated_form}
    \Gamma^{0}(\rho) = \int_{\cR_0} F^{0}(m,\overline{m})\rho(dm,d\overline{m}), \quad \Gamma[\rho] = \int_{\cR_0} F[m,\overline{m}](m,\overline{m})\rho(dm,d\overline{m}), \\
    \Gamma^{dev}[\rho](\kappa,\overline{\kappa}) = \int_{\cR_0} F[m,\overline{m}]( \kappa(m),\overline{\kappa}(m) )\rho(dm,d\overline{m}).
\end{multline}

\begin{lemma}\label{lemma:continuity_integrands}
The map $F:\cR_0 \times \cR_0 \to \R$ is continuous, and so is $F^0:\cR_0 \to \R$. 
\end{lemma}
\begin{proof}
We study the continuity of $F[m,\overline{m}](\eta,\overline{\eta})$, since the continuity of $F^{0}(m,\overline{m})$ is completely analogous.
Let $((\eta^n,\overline{\eta}^n))_{n \geq 1} \subseteq \cR_0$, $(m^n,\overline{m}^n)_{n \geq 1} \subseteq \cR_0$, $(\eta,\overline{\eta}) \in \cR_0$ and $(m,\overline{m}) \in \cR_0$ so that $(\eta^n,\overline{\eta}^n) \to (\eta,\overline{\eta})$ and $(m^n,\overline{m}^n) \to (m,\overline{m})$ in $\tau^{(2)}\otimes\tau^{(2)}$.
Since $\hat{f} \in \dC^{1,2}([0,T] \times \R)$ by assumption, Lemma \ref{lemma:L1_convergence} implies
\begin{equation}\label{eq:convergence_L1_f_hat}
        \int_{\R \times A}\hat{f}(\cdot,x)m^n_\cdot(dx,da) \to \int_{\R \times A}\hat{f}(\cdot,x)m_\cdot(dx,da) 
\end{equation}
in $L^1([0,T])$.
Then, we have the estimates
\begin{align*}
     & \Big\vert \int_{0}^T \int_{\R \times A} f(t,x,m^n_t,a)\eta^n_t(dx,da)dt - \int_{0}^T \int_{\R \times A} f(t,x,m_t,a)\eta_t(dx,da)dt \Big\vert \\
     & \leq \Big\vert \int_{0}^T \int_{\R \times A} f(t,x,m_t,a)(\eta^n_t(dx,da) - \eta_t(dx,da))dt \Big\vert \\
     & \quad  + \Big\vert \int_{0}^T \int_{\R \times A} \big( f(t,x,m^n_t,a) - f(t,x,m_t,a)\big) \eta^n_t(dx,da)dt \Big\vert.
\end{align*}
The first term in the bound above goes to $0$ by convergence in $\bar{\tau}^{(2)}$, as $\eta^n \to \eta$ in $\tau^{(2)}$ and the topologies $\tau^{(2)} \otimes \tau^{(2)}$ and $\bar{\tau}^{(2)} \otimes \tau^{(2)}$ coincide on $\cR_0$ by Lemma \ref{lemma:compactness_a_priori_R0}.
As for the second term, by relying on the locally Lipschitz property of $\bar{f}$ with respect to $y$ (cf. Assumption \ref{standing_assumptions}(S4)), we have
\begin{align*}
    & \Big\vert \int_{0}^T \int_{\R \times A} \big( f(t,x,m^n_t,a) - f(t,x,m_t,a)\big) \eta^n_t(dx,da) dt \Big\vert \\
    & \leq \int_{0}^T \int_{\R \times A} \big\vert f(t,x,m^n_t,a) - f(t,x,m_t,a)\big\vert \eta^n_t(dx,da)dt \\
    & \leq C \int_{0}^T \left( 1 + \int_{\R \times A} (\vert x \vert + \vert a \vert)\eta^n_t(dx,da) + \Big\vert \int_{\R \times A} \hat{f}(t,y)m^n_t(dy,da)  \Big\vert +  \Big\vert \int_{\R \times A} \hat{f}(t,y)m_t(dy,da)  \Big\vert \right) \cdot \\
    & \quad \cdot \Big \vert  \int_{\R \times A} \hat{f}(t,y)m^n_t(dy,da) - \int_{\R \times A} \hat{f}(t,y)m_t(dy,da)  \Big \vert  dt \\
    & \leq C \int_{0}^T \left( 1 + \int_{\R \times A} (\vert x \vert + \vert a \vert ) \eta^n_t(dx,da) + \int_{\R \times A} \vert y \vert m^n_t(dy,da)  + \int_{\R \times A} \vert y \vert m_t(dy,da) \right) \cdot \\
    & \quad \cdot \Big \vert  \int_{\R \times A} \hat{f}(t,y)m^n_t(dy,da) - \int_{\R \times A} \hat{f}(t,y)m_t(dy,da)  \Big \vert  dt \\
    & \leq C \int_0^T \Big \vert  \int_{\R \times A} \hat{f}(t,y)m^n_t(dy,da) - \int_{\R \times A} \hat{f}(t,y)m_t(dy,da)  \Big \vert  dt
\end{align*}
which goes to $0$ by \eqref{eq:convergence_L1_f_hat}, where we exploited the fact that $\hat{f}(t,y)$ has linear growth since it belongs to $\dC^{1,2}([0,T] \times \R)$ and has bounded derivatives (cf. Assumption \ref{standing_assumptions}(S4)) and the uniform bound \eqref{uniform_bound} on the moments of any $(m,\overline{m}) \in \cR_0$ to bound the first term in the product above.
The term dependent on $g$ can be handled in a similar way.
Indeed, we have
\begin{multline*}
    \left\vert \int_{\R}g(x,\overline{m}^n)\overline{\eta}^n(dx)  -  \int_{\R}g(x,\overline{m})\overline{\eta}(dx)  \right\vert  \\
    \leq \left\vert \int_{\R}g(x,\overline{m})(\overline{\eta}^n-\overline{\eta})(dx) \right\vert + \left\vert \int_{\R}(g(x,\overline{m}) - g(x,\overline{m}^n))\overline{\eta}^n(dx) \right\vert.
\end{multline*}
The first term converges to $0$ as $\overline{\eta}^n$ converges to $\overline{\eta}$ in $\tau^{(2)}$.
As for the second term, we have 
\begin{equation}\label{eq:lemma_continuity:estimate}
\begin{aligned}
    & \left\vert \int_{\R} ( g(x,\overline{m}) - g(x,\overline{m}^n) )\overline{\eta}^n(dx)  \right\vert \leq \int_{\R} \left\vert g(x,\overline{m}) - g(x,\overline{m}^n) \right\vert \overline{\eta}^n(dx) \\
    & \leq C \left( 1 + \int_{\R} \vert x \vert \overline{\eta}^n(dx) + \Big\vert \int_{\R} \hat{g}(y)\overline{m}(dy)  \Big\vert +  \Big\vert \int_{\R} \hat{g}(y)\overline{m}^n(dy)  \Big\vert \right) \cdot \\
    & \quad \cdot \Big \vert  \int_{\R} \hat{g}(y)\overline{m}(dy) - \int_{\R} \hat{g}(y)\overline{m}^n(dy)  \Big\vert \leq C \Big\vert  \int_{\R} \hat{g}(y)(\overline{m} -\overline{m}^n)(dy) \Big\vert,
\end{aligned}
\end{equation}
where we used the local Lipschitz property of $\bar{g}$ in its second variable (cf. Assumption \ref{standing_assumptions}(S5)), the linear growth of $\hat g$, and the uniform bound \eqref{uniform_bound}.
Since $\hat{g}$ is continuous with at most linear growth and $\overline{m}^n \to \overline{m}$ in $\tau^{(2)}$, the last term in \eqref{eq:lemma_continuity:estimate} converges to $0$.
\end{proof}

\begin{lemma}\label{lemma:continuity}
Let $(\kappa,\overline{\kappa}) \in \cD$. Define the map $\Delta^{(\kappa,\overline{\kappa})}: \cP(\cR_0) \to \R$ by setting
\begin{equation}\label{eq:difference_optimality}
    \Delta^{(\kappa,\overline{\kappa})}(\rho) := \Gamma[\rho] - \Gamma^{dev}[\rho](\kappa,\overline{\kappa}). 
\end{equation}
Then, $\Delta^{(\kappa,\overline{\kappa})}$ is continuous and linear.
\end{lemma}
\begin{proof}
Linearity is obvious. For the continuity, we show first that the map
\begin{equation}\label{test_function}
    (m,\overline{m}) \mapsto \int_{0}^{T}\int_{\R \times A }f(t,x,m_t,a)\kappa_t(dx,da \vert m)dt +\int_{\R }g(x,\overline{m})\overline{\kappa}(dx \vert m )
\end{equation}
is continuous from $(\cR_0,\tau^{(2)} \otimes \tau^{(2)})$ into $\R$.
Let $(m^n,\overline m^n)\to(m,\overline m)$ in $\cR_0$.
Set $k^n_t(dx,da)=\kappa_t(dx,da\vert m^n)$, $\overline k^n(dx)=\overline\kappa(dx\vert m^n)$, $k_t(dx,da)=\kappa_t(dx,da\vert m)$ and $\overline k(dx)=\overline\kappa(dx\vert m)$.
Notice that, by Definition \ref{LP:def:admissible_deviations}, we have $(k^n,\overline{k}^n) \to (k,\overline{k})$ in $\tau^{(2)} \otimes \tau^{(2)}$.
Then, the map in \eqref{test_function} is continuous if and only if
\begin{multline*}
    \int_{0}^{T}\int_{\R \times A }f(t,x,m^n_t,a)k^n_t(dx,da)dt +\int_{\R}g(x,\overline{m}^n)\overline{k}^n(dx) \\
    \to \int_{0}^{T}\int_{\R \times A }f(t,x,m_t,a)k_t(dx,da)dt +\int_{\R}g(x,\overline{m})\overline{k}(dx),
\end{multline*}
which is granted by the continuity of the map $F[m,\overline{m}](\eta,\overline{\eta})$ given by Lemma \ref{lemma:continuity_integrands}.
This proves the desired continuity of the map $\rho \mapsto \Gamma^{dev}[\rho](\kappa,\overline{\kappa})$. Indeed, as $\cM \subseteq \cP(\cR_0)$ is endowed with the weak topology and $\dC(\cR_0)=\dC_b(\cR_0)$ since $\cR_0$ is compact in $\tau^{(2)} \otimes \tau^{(2)}$, the continuity of \eqref{test_function} implies the continuity of $\Gamma^{dev}[\rho](\kappa,\overline{\kappa})$.
Finally, to see the continuity of $\rho \mapsto \Gamma[\rho]$, it is enough to notice that, by the same reasoning as above, the map
\begin{equation*}
    (m,\overline{m}) \mapsto \int_{0}^{T}\int_{\R \times A }f(t,x,m_t,a)m_t(dx,da)dt + \int_{\R}g(x,\overline{m})\overline{m}(dx)
\end{equation*}
is a continuous map from $(\cR_0,\tau^{(2)} \otimes \tau^{(2)})$ into $\R$ as well.
This yields the continuity of $\Delta^{(\kappa,\overline{\kappa})}(\rho)$.
\end{proof}

\begin{proposition}\label{prop:properties_set_E}
$\cE$ is nonempty, convex and compact.
\end{proposition}
\begin{proof}
We notice that the set of LP-CCEs $\cE$ admits the representation
\begin{equation}
    \cE = \cM \cap \left( \bigcap_{(\kappa,\overline{\kappa}) \in \cD} \{ \rho \in \cP(\cR_0): \, \Delta^{(\kappa,\overline{\kappa})}(\rho) \leq 0 \} \right),
\end{equation}
with $\Delta^{(\kappa,\overline{\kappa})}$ defined by \eqref{eq:difference_optimality}.
We notice that $\cE$ is not empty, as there exists at least one LP-NE by \cite[Theorem 3.4]{leutscher_thesis} and any LP-NE is an LP-CCE by Proposition \ref{LP:prop:NE_is_CCE}.
$\cM$ is compact and convex by Lemma \ref{lemma:M_convex_compact}.
Since $\Delta^{(\kappa,\overline{\kappa})}$ is a linear and continuous mapping, the sublevel sets $\{ \rho \in \cP(\cR_0): \, \Delta^{(\kappa,\overline{\kappa})}(\rho) \leq 0 \}$ are convex and closed.
Then, $\cE$ is closed and convex as it is the intersection of closed and convex sets.
As $\cE \subseteq \cM$, it is compact as well.
This concludes the proof.
\end{proof}

We are now ready to prove Theorem \ref{LP:thm:solution_moderator_optimization_problem}.

\begin{proof}[Proof of Theorem \ref{LP:thm:solution_moderator_optimization_problem}]
Recall the definition of the map $F^{0}(m,\overline{m})$ in \eqref{eq:cost_functionals:function_F0}.
By Lemma \ref{lemma:continuity_integrands}, the map $(m,\overline{m}) \mapsto F^{0}(m,\overline{m})$ is continuous from $\cR_0$ to $\R$.
Thus, the map $\Gamma^{0}$ defined in \eqref{LP:mediator_functional} is continuous from the compact set $\cE$ to $\R$, which implies that it admits a minimum on $\cE$.
\end{proof}

\subsection{Relation with the probabilistic formulation of Section \ref{section: Prob formulation}}
We conclude the existence part by studying the equivalence of the probabilistic formulation of Section \ref{section: Prob formulation} and the linear programming formulation of Section \ref{section: Linear Programming}. To establish such equivalence, we need the following extra assumption:

\begin{assumption}\label{ass: convexity relaxed control assumption}
For any $(t,x,m)\in [0,T]\times \R\times \cP_{2}(\R \times A)$, the set 
\begin{equation}
    K(t,x,m):=\{(b(t,x,m,a),z,z^0): \; a\in A,\; f(t,x,m,a) \leq z , \; f^{0}(t,x,m,a) \leq z^{0} \},
\end{equation}
is convex and closed.
\end{assumption}

Since, by Assumption \ref{standing_assumptions}(S3)--(S5), the functions $b$, $f$, $f^0$, $g$ and $g^0$ depend on $m\in\cP_2(\R\times A)$ only through its marginal law $m^x \in \cP_2(\R)$, in the following, with a slight abuse of notation, whenever $\nu\in\cP_2(\R)$, we write
\begin{multline}\label{eq:relation:abuse_of_notation}
    b(t,x,\nu,a):=\bar{b}\left(t,x,\int_{\R}\hat{b}(t,y)\nu(dy),a\right), \\ f(t,x,\nu,a):=\bar{f}\left(t,x,\int_{\R}\hat{f}(t,y)\nu(dy),a\right), \quad 
    g(x,\nu) = \bar{g}\left(x,\int_{\R}\hat{g}(y)\nu(dy)\right),
\end{multline}
and, analogously, we write $f^0(t,x,\nu,a)$ and $g^0(x,\nu)$ when the measure argument is a state marginal $\nu\in\cP_2(\R)$.

\begin{theorem}\label{theorem: equivalence of formulations}
Suppose that Assumptions \ref{standing_assumptions} and \ref{ass: convexity relaxed control assumption} hold.
Then, there exists a mean-field CCE $(\lambda^*,\mu^*)$ such that
\begin{equation}\label{eq:thm_relation:inequality_value}
    J^{0}(\lambda^*,\mu^*) \leq V^{LP}.
\end{equation}
\end{theorem}
\begin{proof}
By Theorem \ref{LP:thm:solution_moderator_optimization_problem}, there exists an optimal LP-CCE  $\rho^* \in \cE$.
We show that there exists a probability space $(\Omega,\cF,\bbF,\P)$, equipped with an $\bbF$-Brownian motion $W$, an independent $\cF_0$-measurable random variable $\xi$, an $\cF_0$-measurable flow of measures $\mu^*=(\mu^*_t)_{t \in [0,T]}$ with values in $\dM(\cP_2)$ and an $\bbF$-progressively measurable process $\lambda^*$ with values in $A$ such that $(\lambda^*,\mu^*)$ is a mean-field CCE and \eqref{eq:thm_relation:inequality_value} holds.

\textbf{Step 1}. We start by building the candidate optimal correlated flow $(\lambda^*,\mu^*)$.
By Theorem \ref{theorem:consistent_mtg_representation}, there exist a measurable function $q:[0,T] \times \R \times V_{2}\times \cP_{2}(\R) \to \cP(A)$, a complete filtered probability space $(\Omega,\cF,\F,\P)$, equipped with a Brownian motion $W$, an independent $\cF_0$-measurable random variable $\xi$ with law $m^{*}_0$ and an independent pair $(\mu,\overline{\mu})$ consisting of a random flow of measures $\mu\in V_{2}$ and a measure $\overline{\mu}\in \cP_{2}(\R)$ such that $\rho^* = \P \circ (\mu,\overline{\mu})^{-1}$, and an $\F$-adapted process $X$ such that
\begin{equation}\label{eq:theorem_relation:sde}
    dX_{t}=\int_{A}b(t,X_{t},\mu_{t},a)q_{t}(X_{t},\mu,\overline{\mu})(da)dt+\sigma dW_{t},\quad  X_{0} =\xi.
\end{equation}
Moreover, for any $C\in \cB_{\R},\; D\in \cB_{A}$ we have
\begin{equation}\label{eq:thm_relation:consistency}
\begin{aligned}
    & \mu_{t}(C\times D)=\E \big[ \boldsymbol{1}_{C}(X_{t})q_{t}(X_{t},\mu,\overline{\mu})(D) |\mu,\overline{\mu}\big],\; dt\text{-a.e.},\; \P\text{-a.s.}, \\
    & \overline{\mu}(C)=\E \big[ \boldsymbol{1}_{C}(X_{T})\big| \mu,\overline{\mu} \big],\; \P\text{-a.s.}
\end{aligned}
\end{equation}
Since $K(t,x,m_t)$ is closed and convex by Assumption \ref{ass: convexity relaxed control assumption}, the triple 
\begin{multline*}
    \bigg(\int_A b(t,x,m_t,a)q_t(x,m,\overline{m} )(da), \int_A f(t,x,m_t,a)q_t(x,m,\overline{m} )(da), \\
    \int_A f^{0}(t,x,m_t,a)q_t(x,m,\overline{m} )(da) \bigg)
\end{multline*}
still belongs to $K(t,x,m_t)$ (see, e.g., \cite[Theorem I.6.13]{WargaBookOptimalControl}).
By a standard measurable selection argument (see, e.g., \cite[Lemma A.9]{HaussmannRelaxedControls}), there exists a measurable function $\alpha:[0,T]\times \R\times V_2 \times \cP_2(\R) \to A$ such that 
\begin{align}
    & b(t,x,m_{t}, \alpha(t,x,m,\overline{m})) = \int_{A}b(t,x,m_{t},a)q_{t}(x,m,\overline{m})(da), \label{eq:thm_relation:drift_equality} \\
    & f(t,x,m_{t},\alpha(t,x,m,\overline{m}))\leq \int_{A}f(t,x,m_{t},a)q_{t}(x,m,\overline{m})(da), \label{eq:thm_relation:cost_inequality1} \\
    & f^{0}(t,x,m_{t},\alpha(t,x,m,\overline{m}))\leq \int_{A}f^{0}(t,x,m_{t},a)q_{t}(x,m,\overline{m})(da). \label{eq:thm_relation:cost_inequality2}
\end{align}
In particular, \eqref{eq:thm_relation:drift_equality} implies that the process $X$ defined by \eqref{eq:theorem_relation:sde} satisfies
\[
    dX_{t}=b(t,X_{t},\mu_{t},\alpha(t,X_{t},\mu,\overline{\mu}))dt+\sigma dW_{t},\quad X_{0}=\xi.
\]
Set
\begin{equation}\label{eq:thm_relation:recommendation}
    \lambda^*_t = \alpha(t,X_{t},\mu,\overline{\mu})
\end{equation}
and notice that $\lambda^*=(\lambda^*_t)_{t \geq 0}$ takes values in $A$ and it is $\bbF$-progressively measurable since $X$ is $\bbF$-progressively measurable, $(\mu,\overline{\mu})$ are $\cF_0$-measurable and $\alpha$ is a deterministic measurable function.
Thus, the process $X$ satisfies
\[
    dX_{t}=b(t,X_{t},\mu_{t},\lambda^*_t)dt+\sigma dW_{t},\quad X_{0}=\xi.
\]
We claim that there exists an $\cF_0$-measurable flow of measures $\mu^* = (\mu^*_t)_{t \in [0,T]}$ with values in $\dM(\cP_2)$ such that 
\begin{align}
    & dX_t = b(t,X_t,\mu^*_t,\lambda^*_t)dt + \sigma dW_t, \quad X_0 = \xi, \\
    & \mu^*_t ( \cdot ) = \P( (X_t,\lambda^*_t) \in \cdot \vert \mu^*), \quad \forall \, t \in [0,T], \\
    & \P \circ (X_t,\lambda^*_t,\mu^{*,x}_t)^{-1} = \P \circ (X_t,\lambda^*_t,\mu^x_t)^{-1}, \quad dt\text{-a.e.,} \label{eq:thm_relation:equality_laws}, \\
    & \P \circ (X_T,\mu^{*,x}_T)^{-1} = \P \circ (X_T,\overline{\mu})^{-1}, \label{eq:thm_relation:equality_laws_terminal_time}
\end{align}
where $\mu^{*,x} = (\mu^{*,x}_t)_{t \in [0,T]}$ denotes the flow of marginals of $\mu^*$ over $\R$.
Indeed, by Lemma \ref{lemma:continuous_version_state_marginal}, there exists an $\cF_0$-measurable random variable $\tilde{\mu}^x=(\tilde{\mu}^x_t)_{t\in[0,T]}$ with values in $\dC([0,T];\cP_2(\R))$ such that $\tilde{\mu}^x_t=\mu^x_t$, $dt$-a.e., $\tilde{\mu}^x_T=\overline{\mu}$ and $\tilde{\mu}^x_t=\P(X_t\in\cdot\vert\tilde{\mu}^x)$ for any $ t\in[0,T]$.
Choosing a jointly measurable version, we define
\begin{equation}\label{eq:thm_relation:flow_of_measures}
    \mu^*_t(\cdot):=\P((X_t,\lambda^*_t) \in \cdot \vert \tilde{\mu}^x).
\end{equation}
Since $\tilde{\mu}^x$ is $\cF_0$-measurable, $\mu^*$ is $\cF_0$-measurable. Moreover, $A$ is compact and $X$ has finite second moments, so $\mu^*$ takes values in $\dM(\cP_2)$.
Notice that $\mu_t^{*,x}=\P(X_t\in\cdot\vert \tilde{\mu}^x)=\tilde{\mu}^x_t$. Since $\tilde{\mu}^x_t=\mu^x_t$, $dt$-a.e., the convention introduced in \eqref{eq:relation:abuse_of_notation} yields
\[
    dX_t=b(t,X_t,\mu^*_t,\lambda_t^*)dt+\sigma dW_t,\quad X_0=\xi.
\]
Moreover, since $\sigma(\mu^*)\subseteq\sigma(\tilde{\mu}^x)$, the tower property gives
\begin{equation}\label{eq:thm_consistency:consistency_condition}
    \mu^*_t(\cdot) = \P( (X_t,\lambda^*_t) \in \cdot \vert \mu^*) \quad \P\text{-a.s., } \forall \, t \in [0,T].
\end{equation}
Finally, \eqref{eq:thm_relation:equality_laws} and \eqref{eq:thm_relation:equality_laws_terminal_time} follow by construction.

\smallskip
\textbf{Step 2}. We show that the correlated flow $(\lambda^*,\mu^*)$ defined by \eqref{eq:thm_relation:recommendation} and \eqref{eq:thm_relation:flow_of_measures} is a mean-field CCE.
Since the consistency condition is already granted by \eqref{eq:thm_consistency:consistency_condition}, we focus on optimality.
By the convention introduced in \eqref{eq:relation:abuse_of_notation}, we have 
\begin{equation}\label{eq:thm_relation:inequality1}
\begin{aligned}
    & J(\lambda^*,\mu^*) \\
    & = \E\bigg[ \int_{0}^{T}f(t,X_{t},\mu^*_{t},\lambda^*_t)dt+g(X_{T},\mu^*_{T}) \bigg] = \int_{0}^{T} \E [f(t,X_{t},\mu^{*,x}_{t},\lambda^*_t)]dt+\E[g(X_{T},\mu^{*,x}_{T})] \\
    &= \int_{0}^{T} \E [f(t,X_{t},\mu^x_{t},\lambda^*_t)]dt+\E[g(X_{T},\overline{\mu})] = \E \bigg[\int_{0}^{T} f(t,X_{t},\mu_{t},\lambda^*_t) dt+ g(X_{T},\overline{\mu}) \bigg] \\
    & = \E\bigg[ \int_{0}^{T}f(t,X_{t},\mu_{t},\alpha(t,X_{t},\mu,\overline{\mu}))dt+g(X_{T},\overline{\mu})\bigg] \\
    &\leq \E\bigg[ \int_{0}^{T}\int_{A}f(t,X_{t},\mu_{t},a)q_{t}(X_{t},\mu,\overline{\mu})(da)dt+g(X_{T},\overline{\mu})\bigg]
\end{aligned}
\end{equation}
where in the third equality we relied on \eqref{eq:thm_relation:equality_laws} and \eqref{eq:thm_relation:equality_laws_terminal_time}, in the fifth equality we relied on the definition of $\lambda^*$ (cf. \eqref{eq:thm_relation:recommendation}) and in the inequality we relied on \eqref{eq:thm_relation:cost_inequality1}.
By conditioning with respect to $(\mu,\overline{\mu})$ and relying on \eqref{eq:thm_relation:consistency}, we have 
\begin{equation}\label{eq:thm_relation:inequality2}
\begin{aligned}
    & J(\lambda^*,\mu^*)  \leq \int_{0}^{T} \E\bigg[ \E\bigg[ \int_{A}f(t,X_{t},\mu_{t},a)q_{t}(X_{t},\mu,\overline{\mu})(da) \bigg  \vert \mu,\overline{\mu} \bigg] \bigg]dt + \E\left [ \E [g(X_{T},\overline{\mu}) | \mu,\overline{\mu} ] \right] \\
    & = \int_{0}^{T} \E\bigg[ \int_{\R \times A} f(t,x,\mu_{t},a)\mu_t(dx,da) \bigg]dt + \E \bigg[ \int_{\R} g(x,\overline{\mu}) \overline{\mu}(dx) \bigg] \\
    &=\int_{\cR_0 }\bigg(\int_{0}^{T}\int_{\R\times A}f(t,x,m_{t},a)m_{t}(dx,da)dt+\int_{\R}g(x,\overline{m})\overline{m}(dx)\bigg)\rho^*(dm,d\overline{m}) \\
    & \leq \int_{\cR_0}\bigg(\int_{0}^{T}\int_{\R\times A}f(t,x,m_{t},a)\kappa_{t}(dx,da|m)dt+\int_{\R}g(x,\overline{m})\overline{\kappa}(dx|m)\bigg)\rho^*(dm,d\overline{m}) \\
    & = \Gamma^{dev}[\rho^*](\kappa,\overline{\kappa})
\end{aligned}
\end{equation}
for any $(\kappa,\overline{\kappa}) \in \cD$ admissible deviation for the LP formulation, since $\rho^* \in \cE$ is LP-CCE.
Let now $\beta \in \bbA$ be any admissible deviation and let $(\kappa,\overline{\kappa}) \in \cD$ be the associated admissible deviation in the LP formulation defined by Proposition \ref{LP:relation:prop_deviations}.
By Proposition \ref{LP:relation:prop:full_martingale_property} and by the convention introduced in \eqref{eq:relation:abuse_of_notation}, we get
\begin{align*}
    &\Gamma^{dev}[\rho^*](\kappa,\overline{\kappa}) \\
    & = \int_{\cR_0}\bigg(\int_0^T\int_{\R\times A} f(t,x,m_t,a)\kappa_t(dx,da|m)dt+\int_{\R}g(x,\overline{m})\overline{\kappa}(dx|m)\bigg)\rho^*(dm,d\overline{m}) \\
    & = \E\bigg[\int_0^T f(t,X_t^\beta,\mu_t^x,\beta_t)dt+g(X_T^\beta,\overline{\mu})\bigg] = \E\bigg[\int_0^T f(t,X_t^\beta,\mu_t^{*,x},\beta_t)dt+g(X_T^\beta,\mu_T^{*,x})\bigg] \\
    & = J(\beta,\mu^*),
\end{align*}
where, in the second to last equality, we relied on the fact that $\mu_t^{*,x}=\mu_t^x$, $dt$-a.e., and $\mu_T^{*,x}=\overline{\mu}$.
Combining \eqref{eq:thm_relation:inequality2} with the previous identity yields $J(\lambda^*,\mu^*)\leq J(\beta,\mu^*)$. Since $\beta\in\bbA$ was arbitrary, $(\lambda^*,\mu^*)$ satisfies the optimality condition and is therefore a mean-field CCE.

\smallskip
\textbf{Step 3}. We finally verify that \eqref{eq:thm_relation:inequality_value} holds.
Exploiting \eqref{eq:thm_relation:cost_inequality2}, by the same reasoning as in \eqref{eq:thm_relation:inequality1} and \eqref{eq:thm_relation:inequality2}, it is possible to show that
\begin{align*}
    J^0(\lambda^*,\mu^*) & \leq \E\bigg[ \int_{0}^{T}\int_{A}f^{0}(t,X_{t},\mu_{t},a)q_{t}(X_{t},\mu,\overline{\mu})(da)dt+g^{0}(X_{T},\overline{\mu})\bigg] \\
    = & \int_{\cR_0 }\bigg(\int_{0}^{T}\int_{\R\times A}f^{0}(t,x,m_{t},a)m_{t}(dx,da)dt+\int_{\R}g^{0}(x,\overline{m})\overline{m}(dx)\bigg)\rho^*(dm,d\overline{m}) \\
    = & \hspace{0,1cm}\Gamma^{0}(\rho^*) = V^{LP},
\end{align*}
since $\rho^* \in \cE$ is an optimal LP-CCE, and so we have $\Gamma^{0}(\rho^*) = V^{LP}$.
This concludes the proof.
\end{proof}

\subsection{The case of $m$-independent dynamics}
Let $(\lambda^*,\mu^*)$ be the mean-field CCE built from $\rho^* \in \cE$ in Theorem \ref{theorem: equivalence of formulations}.
Under stronger assumptions, it is possible to prove that the mean-field CCE $(\lambda^*,\mu^*)$ is optimal for the moderator's problem in the sense of Definition \ref{def:probabilistic_formulation:optimal_CCE}.
The first assumption that we need requires that the dynamics are independent of the measure argument $m \in \cP_2(\R \times A)$:
\begin{assumption}\label{ass: independent dynamics from m}
The drift is independent of the measure term $m$, i.e., it takes the form $b(t,x,a)$.
\end{assumption}
The following remark summarizes some useful consequences of Assumption \ref{ass: independent dynamics from m}.

\begin{remark}\label{remark:independent_dynamics_from_m}
\begin{enumerate}
    \item \label{remark: set of deviations in the case of independent dynamics}
    Under Assumption \ref{ass: independent dynamics from m}, the sets $\cR[m]$, $m\in V_2$, are all the same, independently of $m$, and so we denote them all by $\cR$.

    \item \label{remark: set of deviations in the case of independent dynamics D=R}
    Under Assumption \ref{ass: independent dynamics from m}, the deviation class $\cD$ can be identified with $\cR$.
    Indeed, if $(\kappa,\overline{\kappa})\in\cD$, then $d_b(m,m')=0$ for every $m,m'\in V_2$.
    By condition \eqref{LP:def:admissible_deviations:stability} in Definition \ref{LP:def:admissible_deviations},
    \[
        d_{\cR_0}\big((\kappa(m),\overline{\kappa}(m)),(\kappa(m'),\overline{\kappa}(m'))\big)=0,\quad m,m'\in V_2.
    \]
    By the definition of $d_{\cR_0}$, this implies $(\kappa(m),\overline{\kappa}(m))=(\kappa(m'),\overline{\kappa}(m'))$ in $\cR_0$.
    Since $m,m'\in V_2$ are arbitrary, the map $m \mapsto (\kappa(m),\overline{\kappa}(m))$ is constant.
    Conversely, every element $(\kappa,\overline{\kappa})\in\cR$ induces the constant admissible deviation $m\mapsto(\kappa,\overline{\kappa})$.
    Therefore, under Assumption \ref{ass: independent dynamics from m}, we identify $\cD$ with $\cR$.

    \item Whenever Assumption \ref{ass: independent dynamics from m} is in force, the set of LP-CCEs $\cE$ and the value $V^{LP}$ are understood with respect to the deviation class $\cD=\cR$.
\end{enumerate}
\end{remark}

For any $\rho\in\cP(\cR)$, define the functions
\begin{equation}
    \label{eq:aggregate_f_g_in_m}
    F^{\rho}(t,x,a):=\int_{\cR}f(t,x,m_{t},a)\rho(dm,d\overline{m}),\quad G^{\rho}(x):=\int_{\cR}g(x,\overline{m})\rho(dm,d\overline{m}).
\end{equation}
\begin{assumption}\label{ass: averaged_convexity}
For every $\rho\in\cM$ and $(t,x)\in[0,T]\times\R$, the set
\[
    \overline{K}^{\rho}(t,x):=\{(b(t,x,a),z):a\in A,\ F^{\rho}(t,x,a)\leq z\}
\]
is convex and closed.
\end{assumption}
Assumption \ref{ass: averaged_convexity} is automatically satisfied, for instance, if $A$ is convex, $a\mapsto b(t,x,a)$ is affine and $a\mapsto f(t,x,m,a)$ is convex and continuous for every $(t,x,m)$. Indeed, in this case $a\mapsto F^{\rho}(t,x,a)$ is convex as an average of convex functions, and the set $\overline{K}^{\rho}(t,x)$ is the epigraph of a convex cost along an affine drift map. Closedness follows from the compactness of $A$ and the continuity in $a$.

\begin{proposition}\label{prop:any_strong_CCE_induces_LP_CCE}
Suppose that Assumptions \ref{standing_assumptions}, \ref{ass: independent dynamics from m} and \ref{ass: averaged_convexity} hold. Let $(\lambda,\mu)$ be a mean-field CCE. Then, $\rho=\P\circ(\mu,\mu_T^x)^{-1}$ is an LP-CCE.
\end{proposition}
\begin{proof}
Let $(\lambda,\mu)$ be a mean-field CCE, we define $\rho:=\P\circ (\mu,\mu^{x}_{T})^{-1}$. From Proposition \ref{LP:relation:prop:consistent_strategies}, we know that $\rho\in \cM$.
It is easy to see that $J(\lambda,\mu)=\Gamma[\rho]$.
We show that for any $(\kappa,\overline{\kappa})\in \cR$, there exists $\beta^{(\kappa,\overline{\kappa})}\in \bbA$ such that
\begin{equation}
    J(\beta^{(\kappa,\overline{\kappa})},\mu)\leq \Gamma^{dev}[\rho](\kappa,\overline{\kappa}).
\end{equation}
This implies
\begin{equation*}
    \Gamma[\rho]=J(\lambda,\mu)\leq J(\beta^{(\kappa,\overline{\kappa})},\mu)\leq \Gamma^{dev}[\rho](\kappa,\overline{\kappa}),
\end{equation*}
for any $(\kappa,\overline{\kappa}) \in \cR$, where the first inequality is due to the optimality property of mean-field CCEs.
To show our claim, by Fubini's theorem, using the quadratic growth of $f,g$ and the uniform second-moment bounds on $\cR_0$, we have
\begin{align*}
    \Gamma^{dev}[\rho](\kappa,\overline{\kappa}) & =\int_{\cR}\bigg(\int_{0}^{T}\int_{\R\times A}f(t,x,m_{t},a)\kappa_{t}(dx,da)dt+\int_{\R}g(x,\overline{m})\overline{\kappa}(dx)\bigg)\rho(dm,d\overline{m}) \\
    & =\int_{0}^{T}\int_{\R\times A}F^{\rho}(t,x,a)\kappa_{t}(dx,da)dt+\int_{\R}G^{\rho}(x)\overline{\kappa}(dx).
\end{align*}
By Assumption \ref{ass: averaged_convexity}, the set $\overline{K}^{\rho}(t,x)$ is closed and convex.
Let $(\kappa,\overline{\kappa})\in \cR$. By \cite[Theorem C.6]{dumitrescu2021EJP}, there exist a measurable function $q^{(\kappa,\overline{\kappa})}:[0,T]\times\R\to \cP(A)$ and a complete filtered probability space $(\overline{\Omega},\overline{\cF},\overline{\bbF},\overline{\P})$ equipped with an $\overline{\bbF}$-Brownian motion $\overline{W}$ and an $\overline{\bbF}$-adapted process $X$ such that
\begin{equation}
    \label{eq:deviator_dynamics_ind}
    dX_{t}=\int_{A}b(t,X_{t},a)q^{(\kappa,\overline{\kappa})}_{t}(X_{t})(da)dt+\sigma d\overline{W}_{t},\quad \overline{\P} \circ X_{0}^{-1}=m^*_0,
\end{equation}
and, for any $C\in \cB_{\R}$ and $D\in \cB_{A}$ holds
\begin{align}
    \label{eq:prop_repr_dev_ind_dyn_t}
    \kappa_{t}(C\times D)&=\E^{\overline{\P}}\big[ \boldsymbol{1}_{C}(X_{t})q^{(\kappa,\overline{\kappa})}_{t}(X_{t})(D) \big],\; dt\text{-a.e.,} \\
    \label{eq:prop_repr_dev_ind_dyn}
    \overline{\kappa}(C)&=\E^{\overline{\P}}\big[ \boldsymbol{1}_{C}(X_{T})\big].
\end{align}
Since $b(t,x,a)$ is bounded by Assumption \ref{standing_assumptions} and $\sigma>0$ is constant, \cite[Theorem 1]{veretennikov1980strongsolutions} implies strong well-posedness for \eqref{eq:deviator_dynamics_ind}. Therefore, we may solve the same equation on the original probability space, driven by $W$ and with initial condition $\xi$. By uniqueness in law, the representation identities \eqref{eq:prop_repr_dev_ind_dyn_t}--\eqref{eq:prop_repr_dev_ind_dyn} remain valid for this solution.
Since $\overline{K}^{\rho}(t,x)$ is closed and convex, we have 
\begin{equation}
    \bigg( \int_{A}b(t,x,a)q^{(\kappa,\overline{\kappa})}_{t}(x)(da),\int_{A}F^{\rho}(t,x,a)q^{(\kappa,\overline{\kappa})}_{t}(x)(da) \bigg)\in \overline{K}^{\rho}(t,x).
\end{equation}
Hence, by a measurable selection argument (see again \cite[Theorem A.9]{HaussmannRelaxedControls}), there exists a deterministic measurable function $\alpha^{(\kappa,\overline{\kappa})}:[0,T]\times \R \to A$ such that
\begin{align}
    \label{eq:deterministic_control_via_meas_sel}
    b(t,x,\alpha^{(\kappa,\overline{\kappa})}(t,x))&=\int_{A}b(t,x,a)q^{(\kappa,\overline{\kappa})}_{t}(x)(da), \notag\\
    F^{\rho}(t,x,\alpha^{(\kappa,\overline{\kappa})}(t,x))&\leq\int_{A}F^{\rho}(t,x,a)q^{(\kappa,\overline{\kappa})}_{t}(x)(da).
\end{align}
Moreover, by employing again \cite{veretennikov1980strongsolutions}, the process $X$ is the unique strong solution to
\begin{equation}
    dX_{t}=b(t,X_{t},\alpha^{(\kappa,\overline{\kappa})}(t,X_{t}))dt+\sigma dW_{t},\quad X_{0}=\xi.
\end{equation}
In particular, the process $X$ can be taken $\F^{\xi,W}$-adapted, which implies that
\[
\beta^{(\kappa,\overline{\kappa})}:=(\alpha^{(\kappa,\overline{\kappa})}(t,X_{t}))_{t\in [0,T]}
\]
belongs to $\bbA$. Therefore,
\begin{align*}
    & J(\beta^{(\kappa,\overline{\kappa})},\mu) =\E\bigg[ \int_{0}^{T}f(t,X_{t},\mu_{t},\beta_{t}^{(\kappa,\overline{\kappa})})dt+g(X_{T},\mu_{T}) \bigg]\\
    &=\int_{0}^{T}\E\big[\E\big[ f(t,X_{t},\mu_{t},\beta_{t}^{(\kappa,\overline{\kappa})})\big|\xi,W\big] \big]dt+\E\big[\E\big[g(X_{T},\mu_{T}^{x}) \big|\xi,W\big] \big] \\
    &=\int_{0}^{T}\E\big[ F^{\rho}(t,X_{t},\beta_{t}^{(\kappa,\overline{\kappa})}) \big]dt+\E\big[ G^{\rho}(X_{T})\big] =\int_{0}^{T}\E\big[ F^{\rho}(t,X_{t},\alpha^{(\kappa,\overline{\kappa})}(t,X_t))\big]dt+\E\big[ G^{\rho}(X_{T})\big]\\
    &\leq \int_{0}^{T}\E\bigg[\int_{A}F^{\rho}(t,X_{t},a)q^{(\kappa,\overline{\kappa})}_{t}(X_{t})(da) \bigg]dt+\E[G^{\rho}(X_{T})] \\
    &=\int_{0}^{T}\int_{\R\times A}F^{\rho}(t,x,a)\kappa_{t}(dx,da)dt+\int_{\R}G^{\rho}(x)\overline{\kappa}(dx) =\Gamma^{dev}[\rho](\kappa,\overline{\kappa}),
\end{align*}
where the third equality follows from the tower property and the definition of $F^{\rho}$ and $G^{\rho}$ and the inequality is due to \eqref{eq:deterministic_control_via_meas_sel}.
Indeed, since $X$ and $\beta^{(\kappa,\overline{\kappa})}$ are $\F^{\xi,W}$-adapted and $(\mu,\mu_T^x)$ is independent of $(\xi,W)$ with law $\rho$, for every $t\in[0,T]$,
\[
    \E\big[f(t,X_t,\mu_t,\beta_t^{(\kappa,\overline{\kappa})})\vert\xi,W\big]
    =\int_{V_2\times\cP_2(\R)}f(t,X_t,m_t,\beta_t^{(\kappa,\overline{\kappa})})\rho(dm,d\overline{m})
    =F^{\rho}(t,X_t,\beta_t^{(\kappa,\overline{\kappa})}),
\]
and similarly
\[
    \E\big[g(X_T,\mu_T^x)\vert\xi,W\big] = \int_{V_2\times\cP_2(\R)}g(X_T,\overline{m})\rho(dm,d\overline{m}) =G^{\rho}(X_T).
\]
This concludes the result.
\end{proof}

\begin{theorem}\label{theorem: equivalence of formulations case of independent dynamics}
Suppose that Assumptions \ref{standing_assumptions}, \ref{ass: convexity relaxed control assumption}, \ref{ass: independent dynamics from m} and \ref{ass: averaged_convexity} hold.
Then, there exists a mean-field CCE $(\lambda^*,\mu^*)$ which is optimal for the moderator.
\end{theorem}
\begin{proof}
By Theorem \ref{theorem: equivalence of formulations}, there exists a mean-field CCE $(\lambda^*,\mu^*)$ such that $J^{0}(\lambda^*,\mu^*) \leq V^{LP}$.
Let $(\lambda,\mu)$ be any other mean-field CCE and define $\rho = \P \circ (\mu,\mu^x_T)^{-1}$. By Proposition \ref{prop:any_strong_CCE_induces_LP_CCE}, $\rho\in \cE$, which implies
\begin{equation}
    \label{eq:Gamma0_ineq_for_rho}
    V^{LP} = \inf_{\rho' \in \cE} \Gamma^{0}(\rho') \leq \Gamma^{0}(\rho) = J^{0}(\lambda,\mu).
\end{equation}
Combining \eqref{eq:thm_relation:inequality_value} and \eqref{eq:Gamma0_ineq_for_rho}, we conclude that
\begin{equation}
    J^{0}(\lambda^{*},\mu^{*})\leq J^{0}(\lambda,\mu)
\end{equation}
for any mean-field CCE $(\lambda,\mu)$. This completes the proof.
\end{proof}

\section{No-Regret Learning Algorithm}\label{section: Learning algorithm}
Regret-minimizing (or no-regret) algorithms are learning schemes in which an agent iteratively updates its decisions based on past performances. Over time, the agent learns so effectively that its long-run average performance is just as good as if it had known the single best fixed action from the very beginning. In the context of correlated equilibria, Hart and Mas-Colell \cite{MasCollel} showed that if agents make decisions based on regret-minimizing algorithms, the empirical distribution of joint play converges to the set of correlated equilibria. This convergence property was later extended to coarse correlated equilibria (see \cite{Blum_Mansour_2007}). The objective of this section is to design an adaptive regret-minimization algorithm that approximates an optimal mean-field LP-CCE. 

\smallskip
In the rest of the paper, we will assume that Assumptions \ref{standing_assumptions} and \ref{ass: independent dynamics from m} are in force, so in particular the state drift does not depend on the measure term $m$. 
Notice that, according to Remark \ref{remark:independent_dynamics_from_m}-(\ref{remark: set of deviations in the case of independent dynamics D=R}), we take $\cD = \cR$ as the set of admissible deviations.

\subsection{A Lagrangian formulation with External Regret}\label{subsection: Lagrangian approach}
Following \cite{Blum_Mansour_2007} and \cite[Definition 24]{muller2022learningcorrelatedequilibriameanfield}, we provide the definition of external regret in the mean-field LP framework.
Recall the notation $\boldsymbol{\eta} = (\eta,\overline{\eta})$ for a pair in $\cR$, the definition of $F[\boldsymbol{m}](\boldsymbol{\eta}) = F[\boldsymbol{m}](\eta,\overline{\eta})$ in \eqref{eq:function_F} and the expression of $\Gamma[\rho]$, $\Gamma^{dev}[\rho](\boldsymbol{\kappa})$ in terms of $F$ given in \eqref{eq:cost_functionals_integrated_form}.

\begin{definition}[LP-External Regret]\label{def: external regret}
The external regret $\cE\cR:\cP(\cR)\to \R$ is given by
\begin{equation}\label{eq: external regret}
    \cE\cR(\rho) :=\sup_{\boldsymbol{\kappa}\in\cR}\big(\Gamma[\rho]-\Gamma^{dev}[\rho](\boldsymbol{\kappa}) \big) = \sup_{\boldsymbol{\kappa}\in\cR} \int_{\cR} \big(F[\boldsymbol{m}](\boldsymbol{m}) - F[\boldsymbol{m}](\boldsymbol{\kappa})\big)\rho(d\boldsymbol{m}).
\end{equation}
\end{definition}
The following lemma is straightforward from Definition \ref{def: external regret}, and thus its proof is omitted:
\begin{lemma}\label{equivalence of external regret and CCE}
    Let $\rho\in\cM$, the following are equivalent:
    \begin{enumerate}
        \item $\cE\cR(\rho)\leq 0$;
        \item $\rho$ is a mean-field LP-CCE.
    \end{enumerate}
\end{lemma}
From Lemma \ref{equivalence of external regret and CCE} we reformulate LP mediator's problem \eqref{LP:mediator_value} as follows
\begin{equation}
    \label{eq: constrained problem external regret}
    \inf_{\rho\in \cM}\Gamma^{0}(\rho)\quad\text{ under the constraint }\quad \cE\cR(\rho)\leq 0.
\end{equation}

We need a further technical assumption, stating that one can find an LP-CCE which strictly outperforms all admissible deviations. This assumption has to be checked on a case-by-case basis (see the examples in Section \ref{section: examples} and Appendix \ref{app:strictCCE}).

\begin{assumption}\label{assumptions: extra ass on algorithmic part}
There exists $\overline{\rho}\in \cM$ such that $\cE\cR(\overline{\rho})<0$.
\end{assumption}

We introduce a dual problem for the constrained optimization problem (\ref{eq: constrained problem external regret}) by using Lagrange multipliers. To do so, we introduce the Lagrangian function $L: \cM\times \R_{+}\to \R\cup \{\infty\}$ as
\begin{equation}
    \label{eq: Lagrange equation}
    L(\rho,\lambda):=\Gamma^{0}(\rho)+\lambda\cE\cR(\rho) = \Gamma^{0}(\rho) + \lambda \sup_{\boldsymbol{\kappa} \in \cR} \left( \int_{\cR} (F[\boldsymbol{m}](\boldsymbol{m}) - F[\boldsymbol{m}](\boldsymbol{\kappa}))\rho(d\boldsymbol{m})  \right).
\end{equation}

\begin{lemma}\label{lemma:external_regret_continuous}
The external regret is convex, continuous and bounded.
\end{lemma}
\begin{proof}
The convexity of $\cE\cR(\rho)$ follows from the sub-linearity of the supremum.
As for continuity, we start by noticing that the map $\cR^2 \ni (\boldsymbol{m},\boldsymbol{\kappa}) \mapsto F[\boldsymbol{m}](\boldsymbol{m}) - F[\boldsymbol{m}](\boldsymbol{\kappa}) \in \R$ is jointly continuous by Lemma \ref{lemma:continuity_integrands}.
$\cR$ is a Polish space, since it is separable, complete and metrizable by \cite[Proposition 2.13]{LinProgFictDumitrescu}.
The continuity of $(\boldsymbol{\kappa},\rho) \mapsto \int_{\cR} ( F[\boldsymbol{m}](\boldsymbol{m}) - F[\boldsymbol{m}](\boldsymbol{\kappa}) )\rho(d\boldsymbol{m})$ then follows from \cite[Example 2.16]{lacker2018mean}.
Recalling that $\cR$ is compact, by Berge's maximum theorem (see \cite[Theorem 17.31]{aliprantis2013infinite}), $\cE\cR(\rho)$ is continuous.
Since $\cE\cR(\rho)$ is continuous over the compact set $\cP(\cR)$, it is bounded.
\end{proof}

\begin{proposition}\label{proposition: Lagrange formulation}
The following hold: 
\begin{enumerate}
    \item \label{prop: item: primal problem equal to dual} $\min_{\rho\in\cE}\Gamma^{0}(\rho)=\min_{\rho\in \cM}\sup_{\lambda\in \R_{+}}L(\rho,\lambda)$.
    
    \item \label{prop: item: existence of saddle-point} $L$ admits a saddle point, i.e. there exists $(\widehat{\rho},\widehat{\lambda})\in \cM\times \R_{+}$ such that
    \begin{equation}
        L(\widehat{\rho},\lambda)\leq L(\widehat{\rho},\widehat{\lambda})\leq L(\rho,\widehat{\lambda}).
    \end{equation}
    \item \label{prop: item: saddle point is minimum} If $(\widehat{\rho},\widehat{\lambda})$ is a saddle-point of $L$, it holds $L(\widehat{\rho},\widehat{\lambda})=\min_{\rho\in \cE}\Gamma^{0}(\rho)=\Gamma^{0}(\widehat{\rho})$.
\end{enumerate}
\end{proposition}
\begin{proof} (\ref{prop: item: primal problem equal to dual}) Let $\lambda\in \R_{+}$ be given. By the continuity of $\Gamma^0$ and Lemma \ref{lemma:external_regret_continuous}, the function $L(\rho,\lambda)$ is finite for any $(\rho,\lambda)\in\cM\times\R_+$. Moreover, for $\rho\in\cM$ fixed, we have that
\begin{align*}
    \sup_{\lambda\in \R_{+}}\big( \Gamma^{0}(\rho)+\lambda\cE\cR(\rho)\big)
    &=\Gamma^{0}(\rho)+\sup_{\lambda\in \R_{+}}\lambda \cE\cR(\rho).
\end{align*}
It is clear that
\begin{equation}
    \sup_{\lambda\in \R_{+}}\lambda\cE\cR(\rho)=\begin{cases}
        +\infty,\quad &\cE\cR(\rho)>0 \\
        0,\quad &\cE\cR(\rho)\leq 0.
    \end{cases}
\end{equation}
Hence, 
\begin{align*}
    \min_{\rho\in\cM}\sup_{\lambda\in \R_{+}}\big( \Gamma^{0}(\rho)+\lambda\cE\cR(\rho) \big)&=\min_{\rho\in\cM}\bigg(\Gamma^{0}(\rho)+\sup_{\lambda\in \R_{+}}\lambda \cE\cR(\rho)(\boldsymbol{1}_{\cE\cR(\rho)\leq 0}(\rho)+\boldsymbol{1}_{\cE\cR(\rho)> 0}(\rho)\big)\bigg) \\
    &=\min_{\rho\in\cM}\big(\Gamma^{0}(\rho)+ (\infty)\boldsymbol{1}_{\cE\cR(\rho) > 0}(\rho)\big)
    = \min_{\rho\in\cE}\Gamma^{0}(\rho).
\end{align*}
where we recall that the minima exist since $\cR$ and $\cM$ are compact by Lemmata \ref{lemma:compactness_a_priori_R0} and \ref{lemma:M_convex_compact} and $\Gamma^0$ is continuous.

(\ref{prop: item: existence of saddle-point}) We now turn to the existence of saddle points. To this end, we apply Sion's minimax theorem \cite{SionMinMax}. Note that $\cM\subset\cP(\cR)\subset\cM_{s}(\cR)$, where $\cM_{s}(\cR)$ denotes the set of finite signed measures on $\cR$. When equipped with the weak topology $\sigma(\cM_{s}(\cR),C(\cR))$, $\cM_{s}(\cR)$ is a Hausdorff linear topological space.
For any $\lambda \in \R_+$ fixed, since $\Gamma^{0}$ is linear and continuous and $\cE\cR$ is convex and continuous by Lemma \ref{lemma:external_regret_continuous}, 
 $\rho \mapsto L(\rho,\lambda)$ is convex and continuous.
On the other hand, for fixed $\rho \in \cM$, $\lambda \mapsto L(\rho,\lambda)$ is continuous and linear, so in particular it is concave.
Thus, since $\cM$ is compact, Sion's minimax theorem implies
\[
v := \min_{\rho \in \cM} \sup_{\lambda \in \R_+} L(\rho,\lambda) = \sup_{\lambda \in \R_+}\min_{\rho \in \cM}  L(\rho,\lambda) .
\]
We then use Assumption \ref{assumptions: extra ass on algorithmic part} to deduce the existence of a saddle point.
Let $\bfJ(\lambda) := \min_{\rho \in \cM} L(\rho,\lambda) = \min_{\rho \in \cM} (\Gamma^{0}(\rho) +\lambda\cE\cR(\rho))$. Since it is the infimum of linear functions, $\bfJ(\lambda)$ is concave and upper semi-continuous.
By definition of $\bfJ(\lambda)$, it holds
\[
\lim_{\lambda \to \infty}\bfJ(\lambda) \leq \lim_{\lambda \to \infty} (\Gamma^{0}(\overline{\rho}) + \lambda\cE\cR(\overline{\rho})) = -\infty,
\]
where we exploited the fact that $\cE\cR(\overline{\rho}) < 0$ by Assumption \ref{assumptions: extra ass on algorithmic part}.
Moreover, it holds
\[
    M := \sup_{\lambda \in \R_+} \bfJ(\lambda) = \sup_{\lambda \in \R_+} \min_{\rho \in \cM} L(\rho,\lambda) = \min_{\rho \in \cM} \sup_{\lambda \in \R_+} L(\rho,\lambda) = \min_{\rho \in \cE}\Gamma^{0}(\rho)<\infty,
\]
where we used part \ref{prop: item: primal problem equal to dual}.
Since $\bfJ(\lambda)$ is concave, upper-semicontinuous and $\sup_{\lambda \in \R_+}\bfJ(\lambda) < \infty$, we can apply Lemma \ref{appendix: lemma: minimum_concave_usc} to deduce that there exists $\widehat{\lambda} \in \R_+$ such that
\begin{equation}
    \bfJ(\widehat{\lambda}) = \max_{\lambda \in \R_+}\bfJ(\lambda) = \max_{\lambda \in \R_+}\min_{\rho \in \cM}L(\rho,\lambda).
\end{equation}
Let now $\widehat{\rho}$ be such that $\sup_{\lambda \in \R_+} L(\widehat{\rho},\lambda) = \min_{\rho \in \cM}\sup_{\lambda \in \R_+} L(\rho,\lambda)$, which exists by part \ref{prop: item: primal problem equal to dual}, since $\cE$ is compact and $\Gamma^0$ is continuous.
We show that $(\widehat{\rho},\widehat{\lambda})$ is a saddle-point of $L(\rho,\lambda)$.
It holds
\begin{equation*}
    \sup_{\lambda \in \R_+} L(\widehat{\rho},\lambda) = \min_{\rho \in \cM}\sup_{\lambda \in \R_+} L(\rho,\lambda) = v , \qquad  \inf_{\rho \in \cM} L(\rho,\widehat{\lambda}) = \max_{\lambda \in \R_+}\min_{\rho \in \cM} L(\rho,\lambda) = v,
\end{equation*}
which implies
\begin{equation*}
    L(\widehat{\rho},\widehat{\lambda}) \geq \inf_{\rho \in \cM}L(\rho,\widehat{\lambda}) = v, \qquad L(\widehat{\rho},\widehat{\lambda}) \leq \sup_{\lambda \in \R_+}L(\widehat{\rho},\lambda) = v,
\end{equation*}
i.e. $L(\widehat{\rho},\widehat{\lambda}) = v$  and $L(\widehat{\rho},\widehat{\lambda}) = \min_{\rho \in \cM}L(\rho,\widehat{\lambda}) = \max_{\lambda \in \R_+}L(\widehat{\rho},\lambda)$.
This is enough to conclude that $(\widehat{\rho},\widehat{\lambda})$ is a saddle-point for $L(\rho,\lambda)$.

(\ref{prop: item: saddle point is minimum})
It remains to show that, if $(\widehat{\rho},\widehat{\lambda})$ is a saddle-point of $L(\rho,\lambda)$, $\widehat{\rho}\in \cE$ and it attains $\inf_{\rho\in \cE}\Gamma^{0}(\rho)$. To this end, since $(\widehat{\rho},\widehat{\lambda})$ is saddle-point it holds
\begin{equation*}
    L(\widehat{\rho},\lambda)\leq L(\widehat{\rho},\widehat{\lambda}),\quad \text{for any }\lambda\in \R_{+},
\end{equation*}
or, equivalently, 
\begin{equation}
    \label{eq: lambda minus lambda star for external regret}
    (\lambda-\widehat{\lambda})\cE\cR(\widehat{\rho})\leq 0,\quad \text{for any }\lambda\in \R_{+}.
\end{equation}
Arguing by contradiction, we assume that $\cE\cR(\widehat{\rho})>0$. Then, for any $\lambda>\widehat{\lambda}$, we have that $(\lambda-\widehat{\lambda})\cE\cR(\widehat{\rho})>0$, which contradicts \eqref{eq: lambda minus lambda star for external regret}. This proves $\cE\cR(\widehat{\rho}) \leq 0$, and so $\widehat{\rho} \in \cE$.
Next, we show that $\widehat{\lambda}\cE\cR(\widehat{\rho})=0$.
Since $\cE\cR(\widehat{\rho})\leq 0$, we obtain $\widehat{\lambda}\cE\cR(\widehat{\rho})\leq 0$. It remains to show that $\widehat{\lambda}\cE\cR(\widehat{\rho})\geq 0$. From \eqref{eq: lambda minus lambda star for external regret}, we have
\begin{equation*}
    \lambda\cE\cR(\widehat{\rho})\leq \widehat{\lambda}\cE\cR(\widehat{\rho}),\quad \text{for any }\lambda\in \R_{+},
\end{equation*}
therefore, taking $\lambda=0$, we conclude $\widehat{\lambda}\cE\cR(\widehat{\rho})\geq 0$.
Since $\widehat{\lambda}\cE\cR(\widehat{\rho}) = 0$, we have
\begin{equation}
    \inf_{\rho\in \cE}\Gamma^{0}(\rho)=\inf_{\rho\in\cM}\sup_{\lambda\in \R_{+}}\big( \Gamma^{0}(\rho)+\lambda\cE\cR(\rho) \big)=\Gamma^{0}(\widehat{\rho})+\widehat{\lambda}\cE\cR(\widehat{\rho})=\Gamma^{0}(\widehat{\rho})
\end{equation}
thus concluding the proof.
\end{proof}

\begin{remark}
A similar reformulation of the optimal CCE problem appears in \cite{FarinaOptimalCCE}. There, the authors focus on static finite games with $N$ players. Leveraging the linearity of the underlying functionals and the bilinear form of external regret, they demonstrate the existence of optimal equilibria in extensive-form games under mild assumptions.
\end{remark}

\subsection{Primal-Dual Algorithm}\label{section: algorithm}
The aim of this subsection is to develop an algorithm to approximate an optimal LP-CCE. Based on the findings of subsection \ref{subsection: Lagrangian approach}, the optimal LP-CCE problem has been transformed to a saddle-point problem. The latter admits a solution which can be approximated via \textit{primal-dual schemes}.
We draw inspiration from the works of Chambolle and Pock \cite{ChambollePockFirstOrderPrimalDual} and Chambolle and Contreras \cite{Chambolle_Contreras}. Therein, the authors introduce primal-dual proximal-type schemes and show the theoretical convergence to the solution of the saddle-point problem.
It is important to notice that they work on a finite-dimensional space,  while our problem is infinite-dimensional. 

\smallskip
Proximal-type schemes involve an extra regularized term which guarantees the convergence of the algorithm.
To this end, we introduce the Bregman divergence.
For more details about Bregman divergence, we refer to \cite{Mirror_Descent_with_Relative_Smoothness, Mirror_Descent-Ascent_Mean-Field}.
\begin{definition}\label{algo:def:Bregman_divergence}
Let $\psi:\cP(\cR) \to \R$ be a continuous and linearly differentiable function (in the sense of Definition \ref{appendix: definition: flat derivative}), with linear derivative $\frac{\delta \psi}{\delta \rho}:\cP(\cR) \times \cR \to \R$.
The Bregman divergence is a map $D_{\psi}:\cP(\cR) \times \cP(\cR) \to \R$, given by
\begin{equation}
    \label{eq: Bregman divergence for the algo}
    D_{\psi}(\rho',\rho):=\psi(\rho')-\psi(\rho)-\int_{\cR}\frac{\delta\psi}{\delta\rho}(\rho,\boldsymbol{m})(\rho'-\rho)(d\boldsymbol{m}),\quad \rho',\rho\in \cP(\cR).
\end{equation}
We refer to $\psi$ as a Bregman potential.
\end{definition}

The numerical scheme is given by Algorithm \ref{algorithm: Primal Dual for Optimal LP-CCE}. 
\begin{algorithm}
\caption{Primal-Dual for Optimal LP-CCE}
\label{algorithm: Primal Dual for Optimal LP-CCE}

\textbf{Initialize:} number of iterations $N\in\N$, $\lambda^{(0)}>0$, $\rho^{(0)}\in\cM$, $\psi:\cP(\cR) \to \R$ convex and linearly differentiable.

\For{$n=0$ to $N-1$}
{Calculate:
    \begin{align*}
        \label{algo:primal part} \quad\rho^{(n+1)}  & \in \text{argmin}_{\rho \in \cM}\big(\Gamma^{0}(\rho)+\lambda^{(n)}\cE\cR(\rho)+\frac{1}{2}D_{\psi}(\rho,\rho^{(n)})\big) \tag{Primal} \\
        \label{algo: dual part} \quad \lambda^{(n+1)}&=\text{argmin}_{\lambda\in \R_{+}}\big(-\lambda\cE\cR(\rho^{(n+1)})+\frac{\sqrt{N}}{2}|\lambda-\lambda^{(n)}|^{2} \big), \tag{Dual} 
    \end{align*}}
    \textbf{Return:} $(\rho^{(n)},\lambda^{(n)})_{0\leq n\leq N}$.
\end{algorithm}

Algorithm \ref{algorithm: Primal Dual for Optimal LP-CCE} is well-defined.\ Indeed, the \eqref{algo:primal part} update is performed over the compact set $\cM$ (cf. Lemma \ref{lemma:M_convex_compact}), and its objective function is continuous (cf. Lemmata \ref{lemma:continuity_integrands}, and \ref{lemma:external_regret_continuous}); hence, the minimum in the \eqref{algo:primal part} part is attained.\ The \eqref{algo: dual part} update is well-defined as well, since its objective is a continuous, coercive, and strictly convex quadratic function on $\R_{+}$. Consequently, the corresponding argmin sets are nonempty at every iteration.

\smallskip
To show the convergence of Algorithm \ref{algorithm: Primal Dual for Optimal LP-CCE} we need the notion of \textit{primal-dual gap}. Following \cite{Chambolle_Contreras}, let $(\rho,\lambda)\in \cM\times \R_{+}$ be arbitrary, and define the map $\text{Gap}_{(\rho,\lambda)}:\cM\times \R_{+} \to \R$ by 
\begin{equation}
    \label{eq: Primal-Dual gap}
    \text{Gap}_{(\rho,\lambda)}(\widehat{\rho},\widehat{\lambda}):=L(\widehat{\rho},\lambda)-L(\rho,\widehat{\lambda}),
\end{equation}
where $L(\rho,\lambda)$ is defined by \eqref{eq: Lagrange equation}.
The following results establish the three-point inequality for both the dual and primal parts of the algorithm, and will play a key role in the proof of the convergence result. 

\begin{lemma}\label{lemma: displacement convexity of dual part}
Let $G:\R_{+}\to \R\cup \{+\infty\}$ be a convex continuous function and $\sigma > 0$. Let also $\overline{\lambda}\in \R_{+}$, and $\widehat{\lambda}\in\text{argmin}_{\lambda\in \R_{+}}\big( G(\lambda)+\sigma |\lambda-\overline{\lambda}|^{2}\big)$. Then, for any $\lambda\in \R_{+}$, we have
\begin{equation*}
    G(\widehat{\lambda})-G(\lambda)\leq \sigma \big( |\overline{\lambda}-\lambda|^{2}-|\widehat{\lambda}-\lambda|^{2}-|\widehat{\lambda}-\overline{\lambda}|^{2} \big).
\end{equation*}
\end{lemma}
\begin{proof}
Following \cite{ConvergenceAnalysisUsingBregman}, we know that the Bregman function $\psi(\lambda):=\sigma |\lambda|^{2}$, gives Bregman divergence $D_{\psi}(\lambda',\lambda)=\sigma |\lambda'-\lambda|^{2}$. Hence, by applying \cite[Lemma 3.2]{ConvergenceAnalysisUsingBregman}, we conclude.
\end{proof}

For the (\ref{algo:primal part}) part, we use the Bregman divergence to derive the three-point inequality, similar to \cite[Lemma 2.5]{Mirror_Descent-Ascent_Mean-Field}; see Appendix \ref{appendix: technical results: learning part} for details.

\begin{lemma}\label{lemma: three-points inequality for primal part}
Let $\psi:\cP(\cR)\to\mathbb R$ be convex and linearly differentiable, and let $F:\cP(\cR)\to\mathbb R$ be convex and linearly differentiable.    
Let $\overline{\rho}\in \cP(\cR)$, and $\widehat{\rho}\in \operatorname{argmin}_{\rho\in \cM} \big( F(\rho)+ \tau D_{\psi}(\rho,\overline{\rho}) \big)$. Then, for any $\rho\in \cM$, the following holds:
\begin{equation}\label{eq:three-points inequality for primal part}
    F(\widehat{\rho})-F(\rho)\leq \tau \big( D_{\psi}(\rho,\overline{\rho})-D_{\psi}(\rho,\widehat{\rho})-D_{\psi}(\widehat{\rho},\overline{\rho}) \big).
\end{equation}
\end{lemma}

In order to apply the three-point inequality to $F(\rho) =  \Gamma^{0}(\rho) + \lambda^{(n)}\cE\cR(\rho)$, we need to ensure that the external regret is linearly differentiable.
To this end, we will rely on an envelope-theorem-like argument, which requires the following extra assumptions:

\begin{assumption}\label{ass: uniqueness maximizer}
For any $\rho \in \cP(\cR)$, there exists a unique $\boldsymbol{\kappa}^*(\rho):=(\kappa^*(\rho),\overline{\kappa}^{*}(\rho)) \in \cR$ such that
\[
    \int_{\cR} ( F[\boldsymbol{m}](\boldsymbol{m}) - F[\boldsymbol{m}](\boldsymbol{\kappa}^*(\rho)) )\rho(d\boldsymbol{m}) = \max_{\boldsymbol{\kappa} \in \cR} \int_{\cR} ( F[\boldsymbol{m}](\boldsymbol{m}) - F[\boldsymbol{m}](\boldsymbol{\kappa}) )\rho(d\boldsymbol{m})   = \cE\cR(\rho).
\]
\end{assumption}
As for Assumption \ref{assumptions: extra ass on algorithmic part}, this assumption has to be checked on a case-by-case basis (see again the examples in Section \ref{section: examples} and Appendix \ref{app:strictCCE}).
The next lemma establishes that, under Assumption \ref{ass: uniqueness maximizer}, $\cE\cR(\rho)$ is linearly differentiable.
\begin{lemma}\label{lemma: external regret derivative}
The external regret is linearly differentiable, with linear derivative given by
\begin{equation}\label{eq:external_regret:linear_derivative}
    \frac{ \delta \cE\cR }{\delta \rho}(\rho,\boldsymbol{m}) = F[\boldsymbol{m}](\boldsymbol{m}) - F[\boldsymbol{m}](\boldsymbol{\kappa}^*(\rho)).
\end{equation}
\end{lemma}
\begin{proof}
In the following, to ease the notation, we set
\[
    g(\boldsymbol{m},\boldsymbol{\kappa}) := F[\boldsymbol{m}](\boldsymbol{m}) - F[\boldsymbol{m}](\boldsymbol{\kappa}), \quad J(\rho,\boldsymbol{\kappa}) := \int_{\cR} g(\boldsymbol{m},\boldsymbol{\kappa})\rho(d\boldsymbol{m}),
\]
so that
\[
    \cE\cR(\rho) = \sup_{\boldsymbol{\kappa} \in \cR} J(\rho,\boldsymbol{\kappa}).
\]
We recall that $g$ is jointly continuous by Lemma \ref{lemma:continuity} and so is $J$.
Convexity is straightforward from the definition of $\cE\cR$, by exploiting the sub-linearity of the supremum.
We rely again on Berge's maximum theorem (see \cite[Theorem 17.31]{aliprantis2013infinite}) to deduce that the correspondence
\[
    \rho \mapsto \mu(\rho) = \{\boldsymbol{\kappa} \in \cR: \, \cE\cR(\rho) = J(\rho,\boldsymbol{\kappa}) \}
\]
is nonempty, compact and upper hemi-continuous.
By Assumption \ref{ass: uniqueness maximizer}, we have $\mu(\rho) = \{ \boldsymbol{\kappa}^*(\rho) \}$. From the upper hemi-continuity of $\mu(\rho)$, we deduce that the map $\rho \mapsto \boldsymbol{\kappa}^*(\rho)$ is continuous.

For $t \in [0,1]$, $\rho, \rho' \in \cP(\cR)$, set $\rho_t = \rho + t (\rho' - \rho)$.
We prove that
\begin{equation}\label{eq:regret:first_order_variation}
    \lim_{t \to 0} \frac{ \cE\cR(\rho_t) - \cE\cR(\rho) }{t} = \int_{\cR} g(\boldsymbol{m},\boldsymbol{\kappa}^*(\rho))(\rho' - \rho)(d\boldsymbol{m}). 
\end{equation}
Consider $\boldsymbol{\kappa}^*(\rho_t)$ and $\boldsymbol{\kappa}^*(\rho)$ and recall that the map $t \mapsto \boldsymbol{\kappa}^*(\rho_t)$ is continuous from $[0,1]$ in $\cR$ endowed with the topology $\tau^{(2)}\otimes \tau^{(2)}$.
Since $\boldsymbol{\kappa}^*(\rho_t)$ is sub-optimal for $\cE\cR(\rho)$, we have
\begin{align}\label{eq:lemma:linear_derivative:estimate1}
    &\frac{1}{t}\left( \cE\cR(\rho_t) - \cE\cR(\rho) \right) \notag\\
    &= \frac{1}{t} \left( \sup_{\boldsymbol{\kappa} \in \cR}\int_{\cR}g(\boldsymbol{m},\boldsymbol{\kappa})\rho_t(d\boldsymbol{m}) - \sup_{\boldsymbol{\kappa} \in \cR}\int_{\cR}g(\boldsymbol{m},\boldsymbol{\kappa})\rho(d\boldsymbol{m}) \right) \notag\\
    &\leq \frac{1}{t} \int_{\cR}g(\boldsymbol{m},\boldsymbol{\kappa}^*(\rho_t))(\rho_t - \rho)(d\boldsymbol{m}) =  \int_{\cR}g(\boldsymbol{m},\boldsymbol{\kappa}^*(\rho_t))(\rho' - \rho)(d\boldsymbol{m}).
\end{align}
Therefore, by taking the $\limsup$ in \eqref{eq:lemma:linear_derivative:estimate1}, we obtain
\begin{align*}
    &\limsup_{t \downarrow 0} \frac{1}{t}\left( \cE\cR(\rho_t) - \cE\cR(\rho) \right)  \leq \limsup_{t \downarrow 0} \int_{\cR}g(\boldsymbol{m},\boldsymbol{\kappa}^*(\rho_t))(\rho' - \rho)(d\boldsymbol{m}) \\
    &= \lim_{t \downarrow 0}\int_{\cR}g(\boldsymbol{m},\boldsymbol{\kappa}^*(\rho_t))\rho'(d\boldsymbol{m}) -  \lim_{t \downarrow 0}\int_{\cR}g(\boldsymbol{m},\boldsymbol{\kappa}^*(\rho_t))\rho(d\boldsymbol{m}) \\
    &= \int_{\cR}g(\boldsymbol{m},\boldsymbol{\kappa}^*(\rho))(\rho' - \rho)(d\boldsymbol{m}).
\end{align*}
The last equality follows by dominated convergence theorem: indeed, as $g(\boldsymbol{m},\boldsymbol{\kappa})$ is jointly continuous and $t \mapsto \boldsymbol{\kappa}^*(\rho_t)$ is continuous, the integrands in the limits converge to $g(\boldsymbol{m},\boldsymbol{\kappa}^*(\rho))$ pointwise.
Since $g(\boldsymbol{m},\boldsymbol{\kappa})$ is bounded, as it is jointly continuous over the compact space $\cR^2$, the dominated convergence theorem can be applied, implying that we can exchange the limit and integral over $\cR$.
As for the lower bound, we exploit the fact that $\boldsymbol{\kappa}^*(\rho)$ is sub-optimal for $\cE\cR(\rho_t)$, to get
\begin{align*}
    \liminf_{t \downarrow 0} \frac{1}{t}\left( \cE\cR(\rho_t) - \cE\cR(\rho) \right)  &\geq \liminf_{t \downarrow 0} \int_{\cR}g(\boldsymbol{m},\boldsymbol{\kappa}^*(\rho))(\rho' - \rho)(d\boldsymbol{m}) \\
    &= \int_{\cR}g(\boldsymbol{m},\boldsymbol{\kappa}^*(\rho))(\rho' - \rho)(d\boldsymbol{m}).
\end{align*}
Finally, we verify that \eqref{eq:regret:first_order_variation} is enough to conclude the linear differentiability of $\cE\cR(\rho)$, i.e., that the identity
\begin{equation}\label{eq:lemma:linear_derivative:def_linear_derivative}
    \cE\cR(\rho') - \cE\cR(\rho) = \int_0^1 \int_{\cR} \frac{\delta \cE\cR}{\delta \rho}(\rho + t(\rho'-\rho) ,\boldsymbol{m})(\rho'-\rho)(d\boldsymbol{m})dt
\end{equation}
holds for any $\rho, \rho' \in \cP(\cR)$.
To this extent, fix $\rho, \rho' \in \cP(\cR)$ and consider again $\rho_t=\rho + t(\rho'-\rho)$.
Set $f(t) := \cE\cR(\rho_t)$. By the same reasoning as above, we have
\begin{align*}
    f'(t) &= \lim_{h \to 0} \frac{\cE\cR(\rho_{t+h}) - \cE\cR(\rho_t)}{h} =  \lim_{h \to 0} \frac{\cE\cR(\rho_{t} + h(\rho'-\rho) ) - \cE\cR(\rho_t)}{h} \\
    &= \int_{\cR} \frac{\delta \cE\cR}{\delta \rho}(\rho_t,\boldsymbol{m})(\rho'-\rho)(d\boldsymbol{m}) = \int_{\cR} g(\boldsymbol{m},\boldsymbol{\kappa}^*(\rho_t))(\rho'-\rho)(d\boldsymbol{m})
\end{align*}
which is bounded and continuous in $t$ since $g(m,\kappa)$ is bounded, continuous and $t \mapsto \boldsymbol{\kappa}^*(\rho_t)$ is continuous.
Thus, by the fundamental theorem of calculus, we have
\[
\cE\cR(\rho') - \cE\cR(\rho) = f(1) - f(0)= \int_0^1 f'(t)dt = \int_0^1 \int_{\cR} \frac{\delta \cE\cR}{\delta \rho}(\rho + t(\rho'-\rho) ,\boldsymbol{m})(\rho'-\rho)(d\boldsymbol{m})dt,
\]
which concludes the proof.
\end{proof}

We are now in the position to prove the convergence of Algorithm \ref{algorithm: Primal Dual for Optimal LP-CCE}.
\begin{theorem}\label{theorem: Convergence of Primal Dual algorithm}
Assume that Assumptions \ref{assumptions: extra ass on algorithmic part} and \ref{ass: uniqueness maximizer} hold.
Set
\[
\overline{C} = \sup_{\rho \in \cP(\cR)}\cE\cR^2(\rho) < \infty.
\]
Let $(\rho^{(n)},\lambda^{(n)})_{n}$ be a sequence generated by Algorithm \ref{algorithm: Primal Dual for Optimal LP-CCE}. 
Then, for any $(\rho,\lambda) \in \cM \times \R_+$, the sequences
\begin{equation}\label{eq:thm:algo_convergence:series}
    \widehat{\rho}^{(N)}:=\frac{1}{N}\sum_{n=0}^{N-1}\rho^{(n+1)}, \qquad \widehat{\lambda}^{(N)}:=\frac{1}{N}\sum_{n=0}^{N-1}\lambda^{(n)}
\end{equation}
satisfy the following estimate
\begin{equation}\label{eq: convergence rate of primal-dual gap}
    \text{Gap}_{(\rho,\lambda)}(\widehat{\rho}^{(N)},\widehat{\lambda}^{(N)})\leq \frac{1}{ 2\sqrt{N}}\bigg( |\lambda-\lambda^{(0)}|^{2}+D_{\psi}(\rho,\rho^{(0)}) + \overline{C}\bigg).
\end{equation}
Moreover, for any optimal LP-CCE $\rho^{*}\in \cE$, the external regret of $\widehat{\rho}^{(N)}$ satisfies the following estimate:
\begin{equation}\label{eq: external regret inequality}
    \cE\cR(\widehat{\rho}^{(N)})\leq \frac{1}{N}\sum_{n=0}^{N-1}\cE\cR(\rho^{(n+1)})\leq \frac{1}{2\sqrt{N}}\bigg( D_{\psi}(\rho^*,\rho^{(0)}) + |\lambda^{(0)}-1|^{2} + \overline{C}\bigg) +\Gamma^{0}(\rho^{*})-\Gamma^{0}(\widehat{\rho}^{(N)}).
\end{equation}
Finally, the sequence $(\widehat{\rho}^{(N)},\widehat{\lambda}^{(N)})_{N\geq 1}$ converges (up to a subsequence) to a saddle-point $(\widehat{\rho}, \widehat{\lambda})$ of $L(\rho,\lambda)$. 
In particular, $\widehat{\rho}$ is an optimal LP-CCE.
\end{theorem}

\begin{proof}
We split the proof into four steps.
    
\vspace{0.25cm}
\textbf{Step 1}: Take any arbitrary $(\rho,\lambda)\in \cM\times \R_{+}$. For $n=0,\ldots,N-1$, we apply Lemma \ref{lemma: displacement convexity of dual part} to $G(\lambda)=-\lambda\cE\cR(\rho^{(n+1)})$ with $\overline{\lambda}=\lambda^{(n)}$ and $\sigma = \frac{\sqrt{N}}{2}$ and $\widehat{\lambda}=\lambda^{(n+1)}$, we obtain
    \begin{equation}
        \label{eq: displacement convexity of dual part}
        (\lambda-\lambda^{(n+1)})\cE\cR(\rho^{(n+1)})\leq \frac{\sqrt{N}}{2}\big( |\lambda^{(n)}-\lambda|^{2}-|\lambda^{(n+1)}-\lambda|^{2}-|\lambda^{(n+1)}-\lambda^{(n)}|^{2} \big).
    \end{equation}
    Similarly, applying Lemma \ref{lemma: three-points inequality for primal part} to $F(\rho)=\Gamma^{0}(\rho)+\lambda^{(n)}\cE\cR(\rho)$ with $\tau=\frac{1}{2}$, $\overline{\rho}=\rho^{(n)}$, $\widehat{\rho}=\rho^{(n+1)}$ and $\rho\in \cM$ yields
    \begin{align}
        \label{eq: displacement convexity of primal part I}
        \big(\Gamma^{0}(\rho^{(n+1)})+\lambda^{(n)}\cE\cR(\rho^{(n+1)}) \big)&-\big(\Gamma^{0}(\rho)+\lambda^{(n)}\cE\cR(\rho) \big) \notag \\
        &\leq \frac{1}{2}\bigg( D_{\psi}(\rho,\rho^{(n)})-D_{\psi}(\rho,\rho^{(n+1)})-D_{\psi}(\rho^{(n+1)},\rho^{(n)}) \bigg).
    \end{align}
    
    \textbf{Step 2:}
    The goal of this step is to use inequalities from Step 1 in order to bound $\text{Gap}_{(\rho,\lambda)}(\rho^{(n+1)},\lambda^{(n)})$ for any $n\in\N$, where $\text{Gap}_{(\rho,\lambda)}$ is defined by \eqref{eq: Primal-Dual gap}.
    Throughout the following computations, we use the notation $\pm$ to indicate addition and subtraction of terms. We have
    \begin{equation}\label{eq:thm:algorithm:estimate1}
    \begin{aligned}
        \text{Gap}_{(\rho,\lambda)} &(\rho^{(n+1)},\lambda^{(n)}) = L(\rho^{(n+1)},\lambda)-L(\rho,\lambda^{(n)}) \\
        & = \Gamma^{0}(\rho^{(n+1)}) + \lambda\cE\cR(\rho^{(n+1)}) - \Gamma^{0}(\rho) - \lambda^{(n)}\cE\cR(\rho) \\
        & \quad \pm \big(\Gamma^{0}(\rho^{(n+1)})+\lambda^{(n)}\cE\cR(\rho^{(n+1)}) - \Gamma^{0}(\rho) -\lambda^{(n)}\cE\cR(\rho) \big) \\
        & \quad \pm (\lambda-\lambda^{(n+1)})\cE\cR(\rho^{(n+1)}) \\
        & = \big(\Gamma^{0}(\rho^{(n+1)})+\lambda^{(n)}\cE\cR(\rho^{(n+1)}) - \Gamma^{0}(\rho) -\lambda^{(n)}\cE\cR(\rho) \big) + (\lambda-\lambda^{(n+1)})\cE\cR(\rho^{(n+1)}) \\
        & \quad + (\lambda^{(n+1)} - \lambda^{(n)})\cE\cR(\rho^{(n+1)}) \\
        & \leq \frac{1}{2}\bigg( D_{\psi}(\rho,\rho^{(n)})-D_{\psi}(\rho,\rho^{(n+1)})-D_{\psi}(\rho^{(n+1)},\rho^{(n)}) \bigg) \\
        & \quad + \frac{\sqrt{N}}{2}\big( |\lambda^{(n)}-\lambda|^{2}-|\lambda^{(n+1)}-\lambda|^{2}-|\lambda^{(n+1)}-\lambda^{(n)}|^{2} \big) \\
        & \quad + \frac{\sqrt{N}}{2} | \lambda^{(n+1)} - \lambda^{(n)} |^{2} + \frac{1}{2\sqrt{N}} \cE\cR^2(\rho^{(n+1)})
    \end{aligned}
    \end{equation}
    where in the last inequality we used \eqref{eq: displacement convexity of dual part}, \eqref{eq: displacement convexity of primal part I} and Young's inequality.
    Since $\psi$ is convex and differentiable, Lemma \ref{appendix:lemma: convexity in terms of first variation} implies $D_\psi(\rho^{(n+1)},\rho^{(n)}) \geq 0$, so that we can bound the last term in \eqref{eq:thm:algorithm:estimate1} as follows:
    \begin{equation}\label{eq:thm:algorithm:estimate2}
    \begin{aligned}
        \text{Gap}_{(\rho,\lambda)} & (\rho^{(n+1)},\lambda^{(n)}) \\
        & \leq \frac{1}{2}\bigg( D_{\psi}(\rho,\rho^{(n)})-D_{\psi}(\rho,\rho^{(n+1)}) \bigg) + \frac{\sqrt{N}}{2}\big( |\lambda^{(n)}-\lambda|^{2}-|\lambda^{(n+1)}-\lambda|^{2} \big) + \frac{1}{2\sqrt{N}} \overline{C} \\
        & = \frac{1}{2} \left( \sqrt{N} |\lambda - \lambda^{(n)} |^2 + D_{\psi}(\rho,\rho^{(n)})  \right) - \frac{1}{2} \left( \sqrt{N} |\lambda - \lambda^{(n+1)} |^2 + D_{\psi}(\rho,\rho^{(n+1)})  \right) + \frac{\overline{C}}{2\sqrt{N}} \\
        & = R^{(n)} -R^{(n+1)} + \frac{\overline{C}}{2\sqrt{N}},
        \end{aligned}
    \end{equation}
    where we have set $R^{(n)} =  \frac{1}{2} ( \sqrt{N} |\lambda - \lambda^{(n)} |^2 + D_{\psi}(\rho,\rho^{(n)}) )$.
    Notice that $R^{(n)} \geq 0$ for any $n \in \N$, since $D_{\psi}(\rho,\rho^{(N)}) \geq 0$, for all $N \in \N$.
    Summing from $n=0$ to $N-1$ and dividing by $N$, we get
    \begin{equation}\label{eq: bounded sum of GAPs}
    \begin{aligned}
        & \text{Gap}_{(\rho,\lambda)}\left(\frac{1}{N}\sum_{n=0}^{N-1}\rho^{(n+1)},\frac{1}{N}\sum_{n=0}^{N-1}\lambda^{(n)} \right) \leq \frac{1}{N}\sum_{n=0}^{N-1}\text{Gap}_{(\rho,\lambda)}(\rho^{(n+1)},\lambda^{(n)})\\
        & \leq \frac{1}{N}\sum_{n=0}^{N-1}\big(R^{(n)}-R^{(n+1)}\big) + \frac{\overline{C}}{2\sqrt{N}} =\frac{1}{N} ( R^{(0)}-R^{(N)}) + \frac{\overline{C}}{2\sqrt{N}} \leq \frac{1}{N} R^{(0)} + \frac{\overline{C}}{2\sqrt{N}},
    \end{aligned}
    \end{equation}
    where we exploited the fact that $\rho \mapsto L(\rho,\lambda)$ is convex and $\lambda\mapsto L(\rho,\lambda)$ is linear, the fact that the sum $\sum_{n=0}^{N-1}(R^{(n)} -R^{(n +1)})$ is telescopic and the positivity of $R^{(N)}$.
    Expanding on $R^{(0)}$, we get
    \begin{equation}\label{eq:thm:algo_convergene:final_estimate}
    \begin{aligned}
        & \text{Gap}_{(\rho,\lambda)}\left(\frac{1}{N}\sum_{n=0}^{N-1}\rho^{(n+1)},\frac{1}{N}\sum_{n=0}^{N-1}\lambda^{(n)} \right)  \leq \frac{1}{N}\sum_{n=0}^{N-1}\text{Gap}_{(\rho,\lambda)}(\rho^{(n+1)},\lambda^{(n)})\\
        & \leq \frac{1}{2N}D_{\psi}(\rho,\rho^{(0)}) + \frac{1}{2\sqrt{N}}| \lambda - \lambda^{(0)} |^2 + \frac{\overline{C}}{2\sqrt{N}} \leq \frac{1}{2\sqrt{N}} \left( D_{\psi}(\rho,\rho^{(0)}) + | \lambda - \lambda^{(0)} |^2 + \overline{C} \right),
    \end{aligned}
    \end{equation}
    i.e. \eqref{eq: convergence rate of primal-dual gap}.
    
    \vspace{0.25cm}
    \textbf{Step 3:}
    We now turn our attention to proving \eqref{eq: external regret inequality}. Taking $(\rho,\lambda)=(\rho^{*},1)$ with $\rho^{*}\in \cE$ an optimal LP-CCE, by definition of $L(\rho,\lambda)$ (cf. \eqref{eq: Lagrange equation}) we get 
    \begin{equation}\label{eq: regret ineq Gap}
        \text{Gap}_{(\rho^{*},1)}(\rho^{(n+1)},\lambda^{(n)})=\Gamma^{0}(\rho^{(n+1)})+\cE\cR(\rho^{(n+1)})-\Gamma^{0}(\rho^{*})-\lambda^{(n)}\cE\cR(\rho^{*}).
    \end{equation}
    Now, since $\cE\cR(\rho^{*})\leq 0$, rearranging terms in \eqref{eq: regret ineq Gap}, we get
    \begin{equation*}
        \Gamma^{0}(\rho^{(n+1)})+\cE\cR(\rho^{(n+1)})-\Gamma^{0}(\rho^{*})\leq \text{Gap}_{(\rho^{*},1)}(\rho^{(n+1)},\lambda^{(n)}).
    \end{equation*}
    Summing from $n=0$ to $N-1$, multiplying by $\frac{1}{N}$, and using \eqref{eq:thm:algo_convergene:final_estimate}, we deduce
    \begin{multline*}
        \frac{1}{N}\sum_{n=0}^{N-1}\big( \Gamma^{0}(\rho^{(n+1)})+\cE\cR(\rho^{(n+1)})-\Gamma^{0}(\rho^{*})\big) \leq\frac{1}{N}\sum_{n=0}^{N-1}\text{Gap}_{(\rho^{*},1)}(\rho^{(n+1)},\lambda^{(n)}) \\
        \leq \frac{1}{2\sqrt{N}}\bigg( D_{\psi}(\rho^*,\rho^{(0)}) + |\lambda^{(0)}-1|^{2} + \overline{C}\bigg).
    \end{multline*}
    Then, rearranging terms and utilizing linearity of the map $\rho\mapsto \Gamma^{0}(\rho)$, we conclude
    \[
        \cE\cR(\widehat{\rho}^{(N)}) \leq \frac{1}{N}\sum_{n=0}^{N-1}\cE\cR(\rho^{(n+1)}) \leq \frac{1}{2\sqrt{N}} \bigg( D_{\psi}(\rho^*,\rho^{(0)}) + |\lambda^{(0)}-1|^2 +\overline{C} \bigg) +\Gamma^0(\rho^*)-\Gamma^0(\widehat{\rho}^{(N)})
    \]
    where in the first inequality we exploited the convexity of $\rho\mapsto\cE\cR(\rho)$ (cf. Lemma \ref{lemma:external_regret_continuous}).

    \textbf{Step 4:}
    We are now in the position to show the convergence of the primal-dual algorithm.
    We first show that the sequence $(\widehat{\lambda}^{(N)})_{N}$ is bounded.\ Arguing by contradiction, suppose that $(\widehat{\lambda}^{(N)})_{N}$ admits a subsequence $(\widehat{\lambda}^{(N_k)})_{k \geq 1}$ such that $\widehat{\lambda}^{(N_k)}\to \infty$ as $k\to\infty$. 
    Let $\overline{\rho} \in \cM$ be such that $\cE\cR(\overline{\rho})<0$, which exists by Assumption \ref{assumptions: extra ass on algorithmic part}.
    Fix any $\lambda\in\R_+$.
    We have
    \begin{equation}
        \text{Gap}_{(\overline{\rho},\lambda)}(\widehat{\rho}^{(N_k)},\widehat{\lambda}^{(N_k)})=  L(\widehat{\rho}^{(N_k)},\lambda) -L(\overline{\rho},\widehat{\lambda}^{(N_k)})\leq \epsilon_{N_k}
    \end{equation}
    with $\epsilon_{N}:=\frac{1}{2\sqrt{N}}\big( D_{\psi}(\rho^{(0)},\overline{\rho})+|\lambda^{(0)}-\lambda|^{2} + \overline{C}\big)$. Relying on the definition of $L$ (cf. (\ref{eq: Lagrange equation})) and rearranging terms, we obtain 
    \begin{equation*}
        -\widehat{\lambda}^{(N_k)}\cE\cR(\overline{\rho})\leq \epsilon_{N_k}+\Gamma^{0}(\overline{\rho})-L(\widehat{\rho}^{(N_k)},\lambda).
    \end{equation*}
    Since $\lambda$ is fixed, $\Gamma^0$ and $\cE\cR$ are continuous on the compact set $\cM$, the sequence $(L(\widehat{\rho}^{(N_k)},\lambda))_{k\geq 1}$ is bounded.
    Since $\epsilon_{N_k} \to 0$, we deduce $\lim_{k \to \infty } -\widehat{\lambda}^{(N_k)}\cE\cR(\overline{\rho}) < \infty$. On the other hand, since $\cE\cR(\overline{\rho})< 0$, we have $\lim_{k \to \infty } -\widehat{\lambda}^{(N_k)}\cE\cR(\overline{\rho}) = \infty$, which leads us to a contradiction.
    
    Since $\cM$ is compact and by boundedness of $(\widehat{\lambda}^{(N)})_{N \geq 1}$, we can extract a convergent subsequence of $(\widehat{\rho}^{(N)},\widehat{\lambda}^{(N)})_{N}$ with cluster point $(\widehat{\rho},\widehat{\lambda})$. By continuity of the gap, passing to the limit in \eqref{eq: convergence rate of primal-dual gap}, we obtain
    \begin{equation}
        \label{eq: primal-dual gap for the cluster point}
        \text{Gap}_{(\rho,\lambda)}(\widehat{\rho},\widehat{\lambda})\leq 0,
    \end{equation}
    for any $(\rho,\lambda)\in \cM\times \R_{+}$. Taking arbitrary $\rho\in\cM$ and $\lambda=\widehat{\lambda}$ in (\ref{eq: primal-dual gap for the cluster point}) we obtain $L(\widehat{\rho},\widehat{\lambda})\leq L(\rho,\widehat{\lambda})$, and taking $\rho=\widehat{\rho}$ and arbitrary $\lambda\in \R_{+}$ we have $L(\widehat{\rho},\lambda)\leq L(\widehat{\rho},\widehat{\lambda})$, which shows that $(\widehat{\rho},\widehat{\lambda})$ is a saddle-point of $L(\rho,\lambda)$. From Proposition \ref{proposition: Lagrange formulation} we deduce that
    \begin{equation*}
        \Gamma^{0}(\widehat{\rho})=\min_{\rho\in \cE}\Gamma^{0}(\rho)=\inf_{\rho\in\cM}\sup_{\lambda\geq 0}L(\rho,\lambda)=L(\widehat{\rho},\widehat{\lambda}),
    \end{equation*}
    i.e. $\widehat{\rho}$ is an optimal LP-CCE.
\end{proof}

Theorem \ref{theorem: Convergence of Primal Dual algorithm} implies that Algorithm \ref{algorithm: Primal Dual for Optimal LP-CCE} can be used to learn an LP-CCE, even in the case in which the moderator does not have any optimality criterion.
We state this result in the following corollary:

\begin{corollary}\label{algo:corollary:CCE_learning}
Let $\Gamma^0(\rho) \equiv c$, for some constant $c \in \R$.
For any $\rho \in \cE$, the sequence $\widehat{\rho}^{(N)}$ generated by Algorithm \ref{algorithm: Primal Dual for Optimal LP-CCE} satisfies
\begin{equation}\label{eq:corollary:external_regret_bound}
    \cE\cR(\widehat{\rho}^{(N)})\leq \frac{1}{N}\sum_{n=0}^{N-1}\cE\cR(\rho^{(n+1)})\leq \frac{1}{2\sqrt{N}}\bigg(D_{\psi}(\rho,\rho^{(0)}) + |\lambda^{(0)}-1|^{2} + \overline{C}\bigg)
\end{equation}
and it converges (up to a subsequence) to an LP-CCE.
\end{corollary}

We conclude with some useful remarks.

\begin{remark}
\begin{enumerate}
    \item In Fictitious Play algorithms, an agent updates her strategy at each iteration by computing a best response, assuming the other agents play their previously estimated average strategies (see, for instance, \cite{CardaliaguetHadikhanloo,LinProgFictDumitrescu,PerrinLauriereEliePerolat}, among others). In contrast, our approach belongs to the class of \textit{no-regret} algorithms. Instead of evaluating the best response against the previous estimation at each time step, we recommend a strategy that performs better than any other fixed strategy. Consequently, the time-averaged sequence converges to a true CCE, as established in Theorem~\ref{theorem: Convergence of Primal Dual algorithm}. To the best of our knowledge, our algorithm is the first no-regret framework introduced in the context of continuous time mean-field games.
    \item Compared to other algorithms for mean-field games, our approach does not require any monotonicity condition. This aligns with existing literature on correlated MFGs (see, for instance, \cite{campi2024coarse,muller2022learningcorrelatedequilibriameanfield}). A crucial question remains as to when no-regret algorithms converge to a Nash equilibrium. For example, the authors of \cite{ConvergenceNo-RegretToNash} design a no-regret algorithm for potential games and demonstrate its convergence. We believe this approach could be extended to a continuous time and state framework for mean-field games, a direction we leave for future research.
    \item A no-regret bound similar to \eqref{eq: external regret inequality} appears in \cite[Theorem D.1]{FarinaOptimalCCE} in the context of computing optimal CCEs for zero-sum extensive-form games. Additionally, analogous results are presented in \cite{AnagnostidesFarinaPanageasSandholm}, where the authors demonstrate that (online) mirror descent schemes converge to a CCE in two-player general-sum games, and in \cite{CaiDaskalakis2025proximal}, where the authors introduce the concept of \textit{proximal regret} for online learning.
\end{enumerate}
\end{remark}

\section{Parametrized Primal-Dual Scheme}\label{Section: Implementation and Examples}

The aim of this section is to apply the no-regret primal-dual algorithm developed in Section \ref{section: Learning algorithm} to a parametrized class of LP-CCEs, thereby making numerical computations feasible.
Throughout the section, Assumption \ref{ass: independent dynamics from m} is in force, so that the drift is of the form $b(t,x,a)$ and, according to Remark \ref{remark:independent_dynamics_from_m}, the deviation class is identified with $\cR$.

To parametrize the consistent martingale property \eqref{eq:consistent_mtg_property}, we combine neural networks with randomized Markovian policies.
The overall methodology can be summarized as follows:
\begin{enumerate}
    \item First, the moderator draws a correlation signal $\zeta\in Z$ according to a distribution $\Xi^\varphi$, parameterized by $\varphi\in\Phi$.
    
    \item Given the realized signal $\zeta$, we consider a neural network with parameters $\theta\in\Theta$ and introduce a randomized recommendation whose density with respect to the Lebesgue measure on $A$ is denoted by $\pi^\theta_t(a|x,\zeta)$ (see \eqref{eq: randomized parametrized policy}).
    
    \item For each randomized recommendation $\pi^\theta_t(a|x,\zeta)$, we solve the associated Fokker--Planck equation with aggregate drift $B^\theta(t,x;\zeta)$ (see \eqref{eq:parameterized:aggregate_drift}).
    This allows us to construct a related pair $\boldsymbol{m}^{\theta,\zeta}=(m^{\theta,\zeta},\overline{m}^{\theta,\zeta})\in\cR$.
    Under the assumption that the map $(\theta,\zeta)\mapsto\boldsymbol{m}^{\theta,\zeta}$ is continuous (see Assumption \ref{ass:parametric_FP_solver_continuity}) randomization with respect to $\Xi^\varphi$ over these parametrized pairs generates a consistent LP-correlated flow $\rho^{\varphi,\theta}\in\cM$.
    
    \item Finally, for each parametrized consistent LP-correlated flow, we compute the moderator's cost, the representative player's cost, and the external regret.
    This gives parametrized functionals $J^0(\varphi,\theta)$, $J(\varphi,\theta)$ and $ER(\varphi,\theta)$.
\end{enumerate}

The result is a parametrized Primal-Dual Gradient Descent Scheme (see Algorithm \ref{algorithm:Primal_Dual_Gradient_Descent}).\
In this parametrized formulation, the optimization is performed only over the finite-dimensional parameters $\varphi$ and $\theta$.
The working flow of the scheme is displayed in Figure \ref{fig:network_architecture_compact}.

\begin{figure}[htbp]
    \centering
    \resizebox{0.78\textwidth}{!}{
    \begin{tikzpicture}[
        block/.style={rectangle, draw, fill=blue!10, text width=2.8cm, align=center, minimum height=0.9cm, font=\scriptsize, rounded corners},
        nn/.style={draw, fill=orange!20, text width=2.6cm, align=center, minimum height=1.25cm, font=\scriptsize\bfseries},
        midblock/.style={rectangle, draw, fill=purple!10, text width=3.4cm, align=center, minimum height=1cm, font=\scriptsize\bfseries, rounded corners},
        rhoblock/.style={rectangle, draw, fill=cyan!12, text width=3.6cm, align=center, minimum height=1cm, font=\scriptsize\bfseries, rounded corners},
        proc/.style={rectangle, draw, fill=green!10, text width=3.9cm, align=center, minimum height=1.15cm, font=\scriptsize, rounded corners},
        loss/.style={rectangle, draw, fill=red!15, text width=11.8cm, align=center, minimum height=1.6cm, font=\scriptsize, rounded corners},
        opt/.style={rectangle, draw, fill=yellow!20, text width=8.6cm, align=center, minimum height=2.1cm, font=\scriptsize, rounded corners},
        arrow/.style={-Stealth, semithick},
        line/.style={-, semithick}
    ]

        \node [block] (input_zeta) at (-2.6, 1.7) {Mechanism $\zeta\sim \Xi^{\varphi}$};
        \node [block] (input_tx) at (2.6, 1.7) {Input $(t,x)$};

        \node [nn] (nn) at (0,0) {NN Policy \\ $\pi^\theta_t(a|x,\zeta)$ via \eqref{eq: randomized parametrized policy}};

        \draw [arrow] (input_zeta.south) -- ++(0,-0.25) -| (nn.135);
        \draw [arrow] (input_tx.south) -- ++(0,-0.25) -| (nn.45);

        \node [midblock] (fp_solver) at (0, -1.9) {Parametric FP solver \\ Find $p^{\theta,\zeta}$ and $\boldsymbol m^{\theta,\zeta}$};

        \draw [arrow] (nn.south) -- (fp_solver.north);

        \node [rhoblock] (rho_construct) at (0, -3.45) {Correlated flow \\ Construct $\rho^{\varphi,\theta}$ as the law of $\boldsymbol m^{\theta,\cZ^\varphi}$};

        \draw [arrow] (fp_solver.south) -- (rho_construct.north);

        \node [proc] (rep_cost) at (-2.6, -5.7) {\textbf{Representative player} \\ Calculate $J(\varphi,\theta)$ };
        \node [proc] (best_dev) at (2.6, -5.7) {\textbf{Best deviation} \\ Calculate $V^{dev}(\varphi,\theta)=\inf_{\boldsymbol{\kappa}\in \cR}\Gamma^{dev}[\rho^{\varphi,\theta}](\boldsymbol{\kappa})$};

        \coordinate (branch_down) at ($(rho_construct.south) + (0,-0.35)$);
        \draw [arrow] (rho_construct.south) -- (branch_down);
        \draw [arrow] (branch_down) -| (rep_cost.north);
        \draw [arrow] (branch_down) -| (best_dev.north);

        \node [loss] (loss) at (0, -7.8) {$P((\varphi,\theta),(\varphi^{(n)},\theta^{(n)});\lambda^{(n)})=J^0(\varphi,\theta)+\lambda^{(n)}ER(\varphi,\theta)+\frac12D_\psi((\varphi,\theta),(\varphi^{(n)},\theta^{(n)}))$ \\ $ER(\varphi,\theta)=J(\varphi,\theta)-V^{dev}(\varphi,\theta)$};

        \draw [arrow] (rep_cost.south) -- ++(0,-0.25) -| ([xshift=-1.4cm]loss.north);
        \draw [arrow] (best_dev.south) -- ++(0,-0.25) -| ([xshift=1.4cm]loss.north);

        \node [opt] (optim) at (0, -11.6) {\textbf{Primal-Dual update} \\ $\widehat{\theta}\leftarrow \operatorname{Proj}_{\Theta}\left(\theta^{(n)}-\nabla_{\theta}P((\varphi^{(n)},\theta),(\varphi^{(n)},\theta^{(n)});\lambda^{(n)})\right)$ \\ $\widehat{\varphi}\leftarrow \operatorname{Proj}_{\Phi}\left(\varphi^{(n)}-\nabla_{\varphi}P((\varphi,\widehat{\theta}),(\varphi^{(n)},\theta^{(n)});\lambda^{(n)})\right)$ \\ $\lambda^{(n+1)}=\big(\lambda^{(n)}+\frac{1}{\sqrt{N}}ER(\widehat{\varphi},\widehat{\theta})\big)_{+}$};

        \draw [arrow] (loss.south) -- (optim.north);

        \coordinate (left_corridor) at ($(optim.west) + (-4.6, 0)$);

        \draw [line] (optim.west) -- (left_corridor);

        \draw [arrow] (left_corridor) |- (nn.west)
            node[pos=0.25, right, font=\footnotesize, align=left] {Params \\ Update};

        \draw [arrow] (left_corridor) |- (input_zeta.west);

    \end{tikzpicture}}
    \caption{Architecture of the Primal-Dual Gradient Descent Scheme.}
    \label{fig:network_architecture_compact}
\end{figure}
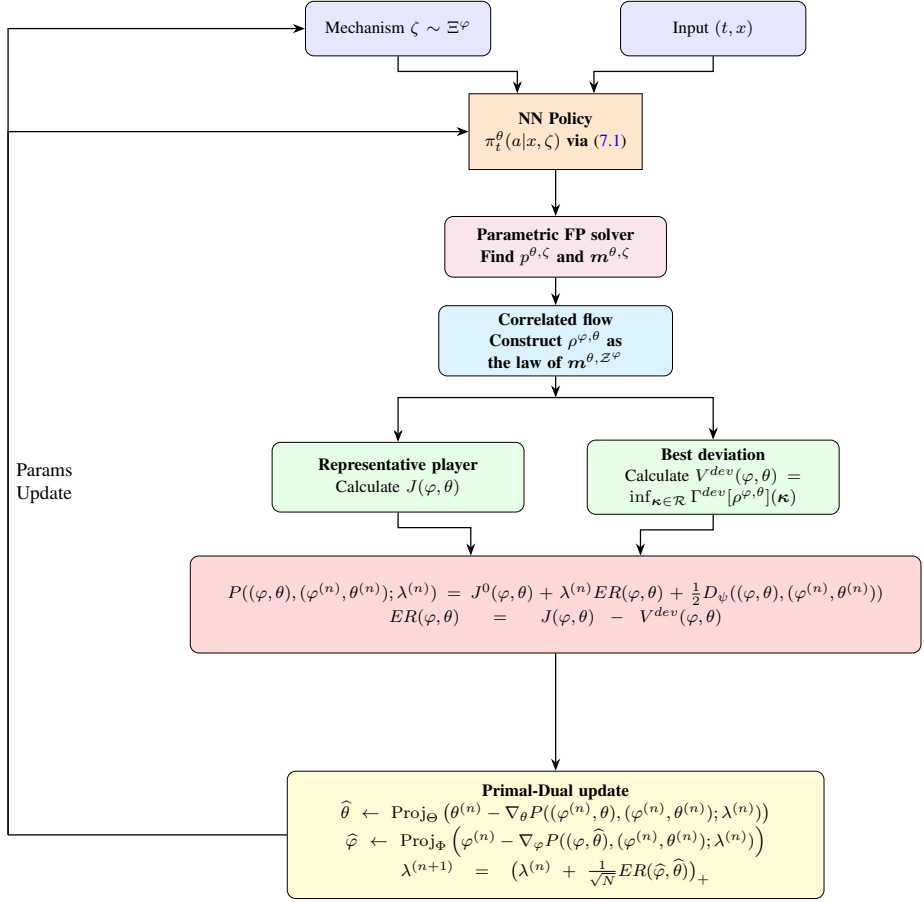

Let $k,\ell,m\in\N$ and let $\Phi\subset\R^k$ and $\Theta\subset\R^\ell$ be compact and convex sets, and $Z\subset\R^m$, possibly equal to the $\R^m$.
The set $Z$ is the set of correlation signals.
For each $\varphi\in\Phi$, let $\Xi^\varphi$ be a probability distribution on $(Z,\cB_Z)$.
The parameter $\varphi$ controls the mediator's randomization over signals, while the parameter $\theta$ controls the randomized recommendation sent after the signal has been realized.

\begin{assumption}\label{ass:parametrization_setting}
\begin{enumerate}
    \item The action space is a compact convex interval $A=[a_{\min},a_{\max}]\subset\R$.
    \item The initial distribution $m_0^*\in\cP_2(\R)$ admits a square-integrable density with respect to the Lebesgue measure, still denoted by $m_0^*$.
\end{enumerate}
\end{assumption}

\subsection{Neural parametrization of the martingale constraint}\label{subsection:neural_parametrization_martingale_constraint}

For each realization $\zeta\in Z$ of the mediator's correlation device, we parameterize the recommended action through a neural network with parameters $\theta\in\Theta$.
Let $L\geq 1$ be the number of hidden layers and let $\phi$ be a continuously differentiable activation function.
Starting from the input
\[
	h_0(t,x,\zeta):=(t,x,\zeta)^\top,
\]
we define recursively
\[
	h_{r+1}(t,x,\zeta;\theta):=\phi(\theta_r h_r(t,x,\zeta;\theta)+\theta_r^0),\quad r=0,\ldots,L-1,
\]
and set
\[
	h(t,x,\zeta;\theta):=\theta_L h_L(t,x,\zeta;\theta).
\]
Here $\theta:=((\theta_r)_{r\leq L},(\theta_r^0)_{r\leq L-1})$ collects the weight matrices and bias vectors.
Since the output $h(t,x,\zeta;\theta)$ is real-valued, we project it smoothly onto the compact action set $A=[a_{\min},a_{\max}]$.
Let $\phi_A:\R\to A$ be a smooth projection-type map and define
\[
	\mu_\theta(t,x,\zeta):=\phi_A(h(t,x,\zeta;\theta)).
\]
The randomized recommendation is then obtained by sampling around this mean action.
Let $\lambda_{\mathrm{pol}}\in\R$ be a learnable log-standard deviation parameter, which is included among the components of $\theta$, and set $\sigma_{\lambda_{\mathrm{pol}}}:=\exp(\lambda_{\mathrm{pol}})$.
For $a\in A$, define
\begin{equation}\label{eq: randomized parametrized policy}
	\pi^\theta_t(a|x,\zeta):=\left( \int_A\exp\left(-\frac{(y-\mu_\theta(t,x,\zeta))^2}{2\sigma_{\lambda_{\mathrm{pol}}}^2}\right)dy \right)^{-1} \exp\left(-\frac{(a-\mu_\theta(t,x,\zeta))^2}{2\sigma_{\lambda_{\mathrm{pol}}}^2}\right) .
\end{equation}
Finally, we set
\[
	q^\theta_t(da|x,\zeta):=\pi^\theta_t(a|x,\zeta)da.
\]
Thus, for each $(\theta,\zeta)\in\Theta\times Z$, $q^\theta(\cdot|\cdot,\zeta)$ is a Markovian relaxed control.
The role of $\zeta$ is to encode the mediator's randomization: different realizations of the correlation signal generate different randomized recommendations, while the distribution of $\zeta$ is controlled by $\varphi$.

\begin{lemma}\label{lemma:regularity_parametrized_policy}
Suppose that $\phi$ and $\phi_A$ are continuously differentiable.
Then the map $(t,x,\zeta,\theta)\mapsto\mu_\theta(t,x,\zeta)$ is continuously differentiable.
Moreover, for every $a\in A$, the map $(t,x,\zeta,\theta)\mapsto\pi^\theta_t(a|x,\zeta)$ is continuously differentiable.
\end{lemma}

\begin{proof}
The claim follows from the differentiability of the neural-network layers, the smoothness of $\phi_A$, and the strict positivity of the normalizing factor in the definition of $\pi^\theta$.
\end{proof}

For $(\theta,\zeta)\in\Theta\times Z$, define the aggregate drift induced by the randomized recommendation by
\begin{equation}\label{eq:parameterized:aggregate_drift}
	B^\theta(t,x;\zeta):=\int_A b(t,x,a)q^\theta_t(da|x,\zeta)=\int_A b(t,x,a)\pi^\theta_t(a|x,\zeta)da.
\end{equation}

\begin{proposition}\label{prop:parametric_FP_wellposedness}
Suppose that Assumptions \ref{standing_assumptions}, \ref{ass: independent dynamics from m} and \ref{ass:parametrization_setting} hold.
Assume moreover that the activation function $\phi$ and the projection map $\phi_A$ are continuously differentiable.
Then, for every $(\theta,\zeta)\in\Theta\times Z$, the Fokker-Planck equation
\begin{equation}\label{eq:parametric_FP_equation}
\left\{
\begin{aligned}
    \partial_t p^{\theta,\zeta}(t,x) &=  -\partial_x\left(B^\theta(t,x;\zeta)p^{\theta,\zeta}(t,x)\right) +\frac{\sigma^2}{2}\partial_{xx}p^{\theta,\zeta}(t,x), \\
    p^{\theta,\zeta}(0,x) &= m_0^*(x),
\end{aligned} \right.
\end{equation}
admits a unique weak probability-density solution with finite second moments.
\end{proposition}
\begin{proof}
Fix $(\theta,\zeta)\in\Theta\times Z$.
Let $\xi$ have law $m_0^*(x)dx$, independently of a Brownian motion $W$, and consider
\[
	dX_t^{\theta,\zeta}=B^\theta(t,X_t^{\theta,\zeta};\zeta)dt+\sigma dW_t,\qquad X_0^{\theta,\zeta}=\xi.
\]
By \eqref{eq:parameterized:aggregate_drift}, Lemma \ref{lemma:regularity_parametrized_policy}, and the standing assumptions on $b$, the aggregate drift $B^\theta(\cdot,\cdot;\zeta)$ is bounded and locally Lipschitz in $x$.
Thus, the SDE admits a unique non-explosive strong solution.
Setting $\mu_t^{\theta,\zeta}:=\cL(X_t^{\theta,\zeta})$, It\^o's formula shows that $\mu^{\theta,\zeta}$ is a weak probability solution of \eqref{eq:parametric_FP_equation} (see also \cite[Proposition 1.3.1]{bogachev2015fokker}).
Since $\sigma>0$ is constant, the equation is uniformly non-degenerate, and \cite[Theorem 6.3.1 and Corollary 6.3.2]{bogachev2015fokker} imply that $\mu_t^{\theta,\zeta}$ admits a density $p^{\theta,\zeta}(t,\cdot)$. Moreover, boundedness of $B^\theta$ and $m_0^*\in\cP_2(\R)$ give
\[
	\sup_{t\in[0,T]}\int_{\R}|x|^2\mu_t^{\theta,\zeta}(dx)<\infty.
\]
Finally, uniqueness of the weak probability solution follows from \cite[Theorem 9.4.3]{bogachev2015fokker}, because the diffusion coefficient is constant and uniformly elliptic, while $B^\theta(\cdot,\cdot;\zeta)$ is bounded and locally Lipschitz in $x$.
Hence \eqref{eq:parametric_FP_equation} admits a unique weak probability-density solution with finite second moments.
\end{proof}
For every $(\theta,\zeta)\in\Theta\times Z$, we define $\boldsymbol{m}^{\theta,\zeta} \in V_2 \times \cP_2(\R)$ by setting 
\[
	m_t^{\theta,\zeta}(dx,da):=q^\theta_t(da|x,\zeta)p^{\theta,\zeta}(t,x)dx,
	\qquad
	\overline{m}^{\theta,\zeta}(dx):=p^{\theta,\zeta}(T,x)dx.
\]

\begin{lemma}\label{lemma:FP_induces_martingale_constraint}
For every $(\theta,\zeta)\in\Theta\times Z$, $\boldsymbol{m}^{\theta,\zeta}$ belongs to $\cR$.
\end{lemma}

\begin{proof}
Let $u\in \dC_b^{1,2}([0,T]\times\R)$.
By the weak formulation of the Fokker--Planck equation,
\begin{multline*}
	\int_{\R}u(T,x)\overline{m}^{\theta,\zeta}(dx)-\int_{\R}u(0,x)m_0^*(x)dx \\
    =
    \int_0^T\int_{\R}
    \left(
    \partial_tu(t,x)
    +\frac{\sigma^2}{2}\partial_{xx}u(t,x)
    +B^\theta(t,x;\zeta)\partial_xu(t,x)
    \right)
    p^{\theta,\zeta}(t,x)dxdt.
\end{multline*}
Using the definition of $B^\theta$ and $m^{\theta,\zeta}_t(dx,da)$, the right-hand side is equal to
\[
	\int_0^T\int_{\R\times A}(\partial_t+\cL)u(t,x,a)m^{\theta,\zeta}_t(dx,da)dt.
\]
Hence $\boldsymbol{m}^{\theta,\zeta}$ satisfies the martingale property, and therefore $\boldsymbol{m}^{\theta,\zeta}\in\cR$.
\end{proof}

\subsection{Parametrized LP-correlated flows}\label{subsection:parametrized_LP_correlated_flows}

In order to associate to a pair $(\varphi,\theta)$ a measure $\rho^{\varphi,\theta}\in\cM$, we need the following continuity assumption.

\begin{assumption}\label{ass:parametric_FP_solver_continuity}
The map
\[
	\Theta \times Z \ni (\theta,\zeta) \longmapsto \boldsymbol{m}^{\theta,\zeta} \in \cR
\]
is continuous with respect to the topology $\tau^{(2)} \otimes \tau^{(2)}$.
\end{assumption}

For each $\varphi\in\Phi$, let $\cZ^\varphi$ be a $Z$-valued random variable with law $\Xi^\varphi$.
For any $(\varphi,\theta)\in\Phi\times\Theta$, consider the map $\boldsymbol{m}^{\theta,\cZ^\varphi}$.
Notice that, thanks to Assumption \ref{ass:parametric_FP_solver_continuity}, the map $\boldsymbol{m}^{\theta,\cZ^\varphi}$ is an $\cR$-valued random variable.
Define
\[
    \rho^{\varphi,\theta} = \P \circ \big( \boldsymbol{m}^{\theta,\cZ^\varphi} \big)^{-1} \in \cP(\cR).
\]
We denote the family of parametrized LP-correlated flows by
\[
	\cP\cM:=\{\rho^{\varphi,\theta}:(\varphi,\theta)\in\Phi\times\Theta\}.
\]

\begin{lemma}\label{lemma:parametrized_flows_are_admissible}
Under Assumption \ref{ass:parametric_FP_solver_continuity}, $\rho^{\varphi,\theta}\in\cM$ for every $(\varphi,\theta)\in\Phi\times\Theta$.
In particular, $\cP\cM\subseteq\cM$.
\end{lemma}

\begin{proof}
We just have to check the consistent martingale property.
Let $\Psi\in \dC_b(\cR)$ and $u\in \dC_b^{1,2}([0,T]\times\R)$.
For every $\zeta\in Z$, Lemma \ref{lemma:FP_induces_martingale_constraint} gives
\[
	\int_{\R}u(T,x)\overline{m}^{\theta,\zeta}(dx)
	-\int_{\R}u(0,x)m_0^*(x)dx
	-\int_0^T\int_{\R\times A}(\partial_t+\cL)u(t,x,a)m_t^{\theta,\zeta}(dx,da)dt
	=
	0.
\]
Multiplying this identity by $\Psi(\boldsymbol{m}^{\theta,\zeta})$ and integrating with respect to $\Xi^\varphi(d\zeta)$ gives the consistent martingale property in Definition \ref{LP:def:consistent_mtg_property}.
Thus $\rho^{\varphi,\theta}\in\cM$.
\end{proof}

\begin{remark}
The parametrized class $\cP\cM$ is generally a strict subset of $\cM$.
Thus the parametrized problem below searches for the best LP-CCE within the family generated by the neural-network recommendations and the parametrized correlation device.
The parametrization sacrifices convexity and full generality in exchange for numerical tractability.
\end{remark}

\subsection{Parametrized primal-dual gradient scheme}\label{subsection:parametrized_objective_primal_dual_gradient_scheme}

For $(\varphi,\theta)\in\Phi\times\Theta$, define the parametrized representative player's cost, the value of the best deviation against the parametrized correlated flow $\rho^{\varphi,\theta}$ and the parametrized moderator cost by
\[
	J(\varphi,\theta):=\Gamma[\rho^{\varphi,\theta}], \qquad V^{dev}(\varphi,\theta):=\inf_{\boldsymbol{\kappa}\in\cR}\Gamma^{dev}[\rho^{\varphi,\theta}](\boldsymbol{\kappa}),   \qquad J^0(\varphi,\theta):=\Gamma^0(\rho^{\varphi,\theta}).
\]
The external regret associated with $(\varphi,\theta)$ is then
\[
	ER(\varphi,\theta):=\cE\cR(\rho^{\varphi,\theta})=J(\varphi,\theta)-V^{dev}(\varphi,\theta).
\]
The parametrized approximation of the moderator's problem is therefore
\[
	\min_{(\varphi,\theta)\in\Phi\times\Theta}J^0(\varphi,\theta)
	\quad\text{under the constraint}\quad
	ER(\varphi,\theta)\leq 0.
\]
Let $\psi:\cP(\cR)\to\R$ be a Bregman potential, as introduced in Definition \ref{algo:def:Bregman_divergence}.
For $(\varphi,\theta),(\overline{\varphi},\overline{\theta})\in\Phi\times\Theta$, we write
\[
	D_\psi((\varphi,\theta),(\overline{\varphi},\overline{\theta})) := D_\psi(\rho^{\varphi,\theta},\rho^{\overline{\varphi},\overline{\theta}}).
\]

The following Algorithm \ref{algorithm:Primal_Dual_Gradient_Descent} is the parametrized counterpart of Algorithm \ref{algorithm: Primal Dual for Optimal LP-CCE}: the minimization over $\rho\in\cM$ is replaced by a minimization over the finite-dimensional family $\rho^{\varphi,\theta}\in\cP\cM$.

\begin{algorithm}
\caption{Parametrized Primal-Dual Scheme}
\label{algorithm:Primal_Dual_Gradient_Descent}

\textbf{Initialize:} number of iterations $N\in\N$, $\lambda^{(0)}>0$, $(\varphi^{(0)},\theta^{(0)})\in\Phi\times\Theta$, and a convex linearly differentiable function $\psi:\cP(\cR)\to\R$.

\For{$n=0$ to $N-1$}
{
Calculate:
\begin{align*}
    \label{algo:parametrized primal part}
    (\varphi^{(n+1)},\theta^{(n+1)}) &\in \mathrm{argmin}_{(\varphi,\theta)\in\Phi\times\Theta}\bigg(J^0(\varphi,\theta)+\lambda^{(n)}ER(\varphi,\theta)+\frac{1}{2}D_\psi((\varphi,\theta),(\varphi^{(n)},\theta^{(n)}))\bigg), \tag{Parametrized Primal} \\
    \label{algo:parametrized dual part}
    \lambda^{(n+1)} &= \mathrm{argmin}_{\lambda\in\R_+}\bigg(-\lambda ER(\varphi^{(n+1)},\theta^{(n+1)})+\frac{\sqrt{N}}{2}|\lambda-\lambda^{(n)}|^2\bigg). \tag{Dual}
\end{align*}
}
\textbf{Return:} $(\varphi^{(n)},\theta^{(n)},\lambda^{(n)})_{0\leq n\leq N}$.
\end{algorithm}

\begin{remark}
The scheme in Algorithm \ref{algorithm:Primal_Dual_Gradient_Descent} is formulated using a neural-network parametrization of the occupation measure. In Section \ref{section: examples}, the corresponding argmin step is implemented by gradient-based optimization (Adam optimizer). This implementation is used only as a practical realization of the parametrized scheme. A rigorous convergence analysis of the gradient-based implementation of the neural parametrized argmin scheme, including differentiability with respect to the network parameters, is left for future work.
\end{remark}

\section{Numerical Results}\label{section: examples}
In this section, we implement the Primal-Dual Gradient Scheme (Algorithm~\ref{algorithm:Primal_Dual_Gradient_Descent}) on two different mean-field games.
Our examples are based on flocking systems \cite{carmona2018probabilistic} and emission abatement games \cite{campi2025LQ}.

The neural-network architecture is specified in the tables below.
When SiLU networks are used, the activation function is $\phi(y):=\frac{y}{1+e^{-y}}$.
Whenever the network output has to be projected onto the compact action space $A=[a_{\min},a_{\max}]$, we use the smooth projection $\phi_A(x):=c+r\tanh\left(\frac{x}{r}\right)$, $r:=\frac{a_{\max}-a_{\min}}{2}$, $c:=\frac{a_{\max}+a_{\min}}{2}$.
When randomized policies over a discretized action grid are used, the network output is instead passed through a softmax layer, as described in Section \ref{Section: Implementation and Examples}\footnote{In Algorithm \ref{algorithm:Primal_Dual_Gradient_Descent}, differentiability with respect to the parameters is required. For this reason, we use smooth activation functions and smooth projections onto the action space.}.
We take $Z=\R$ as the set of correlation signals and use the Gaussian parametrized family
\[
    \Xi^\varphi=\cN(\varphi_1,\varphi_2^2),\qquad \varphi=(\varphi_1,\varphi_2)\in\R\times\R_+.
\]
Furthermore, in the following examples, we use the Bregman divergence associated with the Bregman potential
\[
    \psi\big(\rho^{(\theta,\varphi)}\big):=\left(\int_Z\left(\int_0^T\int_{\R}x\,m_t^{\theta,x}(dx|\zeta)dt\right)\Xi^\varphi(d\zeta)\right)^2,\qquad \rho^{(\theta,\varphi)}\in\cP\cM.
\]

Computing the best-deviation value 
\(V^{\mathrm{dev}}(\varphi,\theta)\) in practice requires
solving the deviating player's optimization problem over the set \(\mathcal R\)
of admissible deviations. In our examples, this problem can be reduced to the
solution of the Hamilton--Jacobi--Bellman equation associated with the deviating
player's control problem. The resulting value is then used to evaluate the
external-regret term appearing in the primal--dual scheme.

The neural-network parameters, which are the primal variables, are updated using the Adam optimizer, while the Lagrange multiplier, which is the dual variable, is updated by projected gradient ascent.

All experiments were conducted on a single NVIDIA A100 Tensor Core GPU (16GB VRAM) via Google Colab.
The models were implemented using the JAX machine learning library.\footnote{The source code used for these experiments can be found at: \hyperlink{https://github.com/JannTzou/Learning-Algorithm-for-Mean-Field-CCE.git}{https://github.com/JannTzou/Learning-Algorithm-for-Mean-Field-CCE.git}.}

Before moving to the examples, we notice that Assumptions \ref{assumptions: extra ass on algorithmic part} and \ref{ass: uniqueness maximizer} are satisfied in the examples below, provided that the numerical values of the coefficients are suitably chosen (see Appendix \ref{app:strictCCE}).

\subsection{Example 1: Simple flocking model}

We begin with an explicitly solvable flocking system, with interactions of scalar type, studied in \cite[Section 2.4.1]{carmona2018probabilistic}.
We cast the example in the framework of Section \ref{section: Prob formulation}.
As the interaction depends on the measure flow only through the average state, for any flow $\mu=(\mu_t)_{t\in[0,T]}$ with values in $\cP(\R\times A)$, we write
\begin{equation}\label{eq:examples:interaction_term}
    \bar{\mu}^x_t:=\int_{\R\times A}x\,\mu_t(dx,da).
\end{equation}
Given a correlated flow $(\lambda,\mu)$, the representative player's state evolves according to
\[
    dX_t=\lambda_tdt+dW_t,\quad X_0=\xi \sim \cN(0.1,0.11^{2}),
\]
and the representative player seeks to minimize the cost functional
\[
    J(\lambda,\mu):=\E\bigg[\int_0^T\left(\frac{1}{2}\lambda_t^2+\frac{1}{2}(X_t-\bar{\mu}^x_t)^2\right)dt\bigg].
\]
In this example, the moderator's objective is to minimize the same payoff as the representative player. Therefore, we set
\[
    J^0(\lambda,\mu):=J(\lambda,\mu).
\]
When $A=\R$, the MFG admits an explicit unique solution by \cite[Section 2.4.1]{carmona2018probabilistic}.
Since the model is a MFG of linear-quadratic type, it is enough to formulate the consistency condition in terms of the flow of means.
Therefore, we denote the MFG solution by $(\alpha^\star,\bar{\mu}^{\star,x})$, where $\alpha^\star$ is the optimal control and $\bar{\mu}^{\star,x}=(\bar{\mu}^{\star,x}_t)_{t\in[0,T]}$ is the flow of means, satisfying $\bar{\mu}^{\star,x}_t=\E[X^\star_t]$, with $X^\star$ being the state process controlled by $\alpha^\star$.
The MFG solution is given by
\[
    \alpha^\star_t = K_t(\bar{\mu}^{\star,x}_t-X^\star_t),\qquad \bar{\mu}^{\star,x}_t \equiv  \E[\xi], \quad t\in[0,T],
\]
where
\[
    K_t=\tanh(T-t).
\]
Moreover, by \cite[Section 6.7.1]{carmona2018probabilistic}, the MFC solution is unique and explicit as well, and it coincides with the MFG solution.
We stress that the explicit solution above corresponds to the unconstrained problem $A=\R$, whereas our model is formulated with the compact action space $A=[a_{\min},a_{\max}]=[-3,3]$, consistently with Assumption \ref{standing_assumptions}.
Thus, the explicit solution is not the exact solution of the compact-action MFG considered here, but it provides a useful benchmark for the numerical results (cf. Figure \ref{fig: flocking system}).

In Figure \ref{fig: flocking system}, we compare the average cost generated by Algorithm \ref{algorithm:Primal_Dual_Gradient_Descent} with the explicit unconstrained benchmark and with the discretized compact-action solution. The figure shows that the algorithm rapidly converges to the MFG solution.
In particular, the algorithm selects the parameters $\varphi_1 = 0$ and $\varphi_2 \simeq 0$.
Thus, the optimal mean-field CCE does not feature any additional randomization.
This is consistent with the findings of \cite[Section 4]{campi2025LQ}: when the MFG and MFC solutions coincide, the best performing mean-field CCE is given by the MFG solution, and no extra randomization is needed.

\begin{table}[htpb]
    \centering
    \caption{Mathematical parameters for the discretized flocking system.}
    \label{table: math params flocking}
    \begin{tabular}{@{}lccccccccc@{}}
        \toprule
        \textbf{Parameter} & $T$ & $dt$ & $\sigma$ & $[x_{\min},x_{\max}]$ & $dx$ & $[a_{\min},a_{\max}]$ & $da$\\
        \midrule
        \textbf{Value} & 5.0 & 0.02 & 1 & [-3,3] & 0.1 & [-3,3] & 0.1\\
        \bottomrule
    \end{tabular}
\end{table}

\begin{table}[htpb]
    \centering
    \caption{Hyperparameters and model constants for the flocking system.}
    \label{table: nn params flocking}
    \begin{tabular}{@{}lccccccccc@{}}
        \toprule
        \textbf{HP} & Hidden Layers & Act. & LR (Adam)  & Batch & Epochs \\
        \midrule
        \textbf{Value} & [64, 64] & SiLU & $10^{-3}$ &  256 & 10,000 \\
        \bottomrule
    \end{tabular}
\end{table}

\begin{figure}[ht]
    \centering
    \includegraphics[width=1\linewidth]{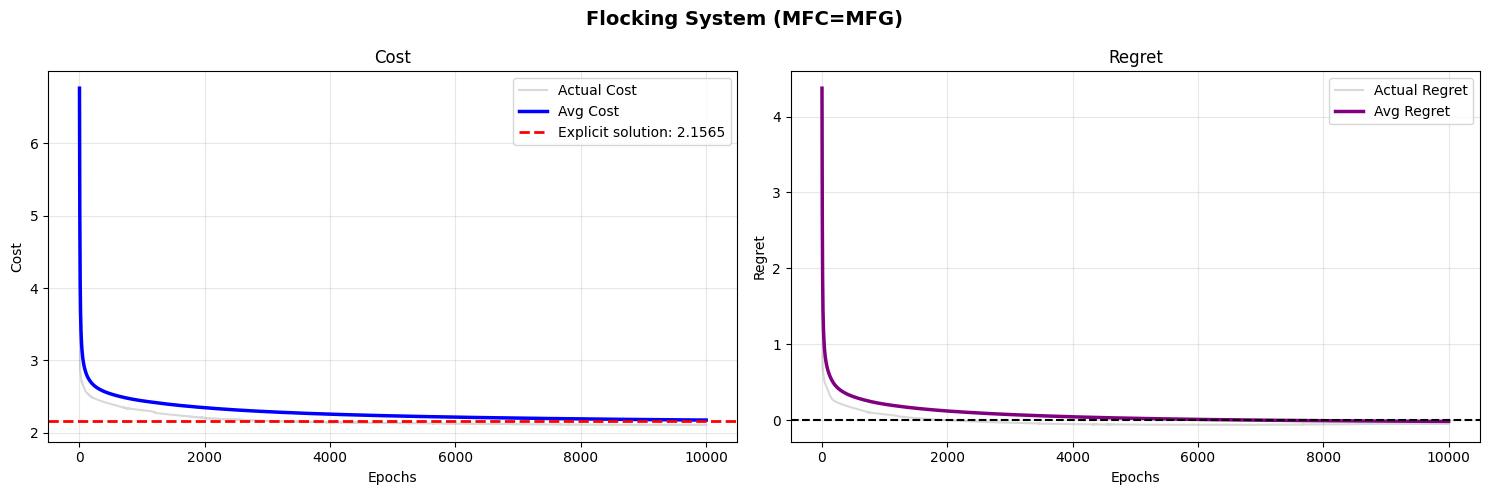}
    \caption{Approximation of the solution to the mean-field flocking system. The average cost produced by Algorithm \ref{algorithm:Primal_Dual_Gradient_Descent} converges to the true value of the MFG. Moreover, the average regret remains slightly below zero.}
    \label{fig: flocking system}
\end{figure}

\subsection{Example 2: Emission Abatement Game}

In this subsection we apply our learning algorithm to the emission abatement game introduced in \cite[Section 5]{campi2025LQ}.
We refer to \cite{campi2025LQ} for a discussion of the economic interpretation of the model.
In order to remain consistent with the convention used in \cite{campi2025LQ}, we switch here from the cost-minimization formulation adopted in the previous sections to the equivalent payoff-maximization formulation.
We consider two different objectives for the moderator: first, maximizing the representative player's payoff, and then maximizing the terminal average state $\E[\bar{\mu}^x_T]$, which, in line with \cite[Section 5]{campi2025LQ}, we interpret as terminal cumulative abatement.

As in the previous example, the interaction depends on the measure flow only through the average state.
Therefore, for any flow $\mu=(\mu_t)_{t\in[0,T]}$ with values in $\cP(\R\times A)$, we consider its flow of state averages $(\bar{\mu}^{x}_t)_{t \in [0,T]}$ defined by \eqref{eq:examples:interaction_term}.
Given a correlated flow $(\lambda,\mu)$, the representative agent's state evolves according to
\begin{equation}\label{eq: level of Emission Abatement}
    dX_t=\lambda_tdt+dW_t,\quad X_0=\xi.
\end{equation}
The representative agent seeks to maximize the payoff
\begin{equation}\label{eq: emission abatement LQ}
    J(\lambda,\mu):=\E\bigg[\int_0^T\bigg(a\bar{\mu}^x_t-\frac{b}{2}(\bar{\mu}^x_t)^2-\frac{1}{2}\lambda_t^2-\frac{1}{2}(X_t-\bar{\mu}^x_t)^2\bigg)dt\bigg].
\end{equation}
Our first goal is to apply the algorithm to learn a CCE that outperforms the standard Nash equilibrium in terms of payoff. Therefore, we set
\[
    J^0(\lambda,\mu):=J(\lambda,\mu)
\]
under the payoff convention.

\smallskip
In \cite{campi2025LQ}, the authors introduce a tractable class of linear CCEs.
More precisely, given a random variable $Z$ with finite first and second moments, they consider a recommendation-flow of state averages pair $(\lambda,\bar{\mu}^x)$ given by 
\[
    \bar{\mu}^x_t=\bar{\mu}^x_0+tZ,\quad \lambda_t=K_t(\bar{\mu}^x_t-X_t)+Z,\quad t\in[0,T],
\]
where $\bar{\mu}^x_0=\E[\xi]$ and $K_t=\tanh(T-t)$.
Writing $z_1=\E[Z]$ and $\sigma_z^2=\mathrm{Var}(Z)$, \cite[Propositions 5.3 and 5.4]{campi2025LQ} characterize the set of admissible moments for which the corresponding linear correlated flow is a CCE and improves on the Nash equilibrium payoff.
Moreover, under the condition $T>\sqrt{3}$, \cite[Proposition 5.7]{campi2025LQ} identifies the payoff-maximizing element in this linear class.
For the parameters in Table \ref{table: math params Emission Abatement}, this gives $z_1\approx 0.4954$, $\sigma_z^2\approx 0.00669$.
We refer to the corresponding equilibrium as the \textit{linear optimal CCE}.
We stress that the linear optimal CCE of \cite{campi2025LQ} is obtained in the unconstrained linear-quadratic setting, whereas our model is formulated with compact action space $A=[-3,3]$.
Thus, the linear optimal CCE is not an admissible competitor for the compact-action problem considered here, but it provides a useful benchmark for the numerical comparisons below.

In contrast with \cite{campi2025LQ}, we do not restrict the search to linear CCEs.
Instead, the algorithm jointly learns the recommendation policy and the parameters of the distribution of the correlation signal.

\smallskip
Evaluating the external regret $\cE\cR(\rho^{\varphi,\theta})$ requires computing the best-deviation value.
We do this by solving the HJB equation associated with the deviating player's control problem.
The Monte Carlo approximation is then used only to integrate over the correlation signal.
For each realization $\zeta$ of the correlation signal, we denote by
\[
    \bar m_t^{\theta,\zeta}:=\int_{\R\times A}x\,m_t^{\theta,\zeta}(dx,da)
\]
the corresponding average state.
Expectations with respect to the correlation signal are approximated by
\[
    \E_{\zeta\sim\Xi^\varphi}\left[\bar m_t^{\theta,\zeta}\right]\approx \frac{1}{B}\sum_{i=1}^B\bar m_t^{\theta,\zeta_i},\quad \zeta_i\sim\Xi^\varphi.
\]
Finally, Tables~\ref{table: math params Emission Abatement} and~\ref{table: nn params Emission Abatement} summarize the mathematical data of the problem together with the algorithmic hyperparameters.

\begin{table}[htpb]
    \centering
    \caption{Mathematical parameters for the Emission Abatement Game.}
    \label{table: math params Emission Abatement}
    \begin{tabular}{@{}lcccccccc@{}}
        \toprule
        \textbf{Parameter} & $T$ & $dt$ & $X_0$ & $[x_{\min},x_{\max}]$ & $dx$ & $[a_{\min},a_{\max}]$ & $a$ & $b$\\
        \midrule
        \textbf{Value} & 5.0 & 0.01 & $\cN(0.1,0.11^2)$ & [-4,4] & 0.1 & [-3,3] & 2.0 & 1.0\\
        \bottomrule
    \end{tabular}
\end{table}

\begin{table}[htpb]
    \centering
    \caption{Hyperparameters and model constants for the Emission Abatement Game.}
    \label{table: nn params Emission Abatement}
    \begin{tabular}{@{}lccccccc@{}}
        \toprule
        \textbf{HP} & Hidden Layers & Act. & LR (Adam) & MC Samples ($B$) & Epochs \\
        \midrule
        \textbf{Value} & [64, 64] & SiLU & $10^{-4}$ & 30 & 10,000 \\
        \bottomrule
    \end{tabular}
\end{table}

\begin{figure}[ht]
    \centering
    \includegraphics[width=1.0\linewidth]{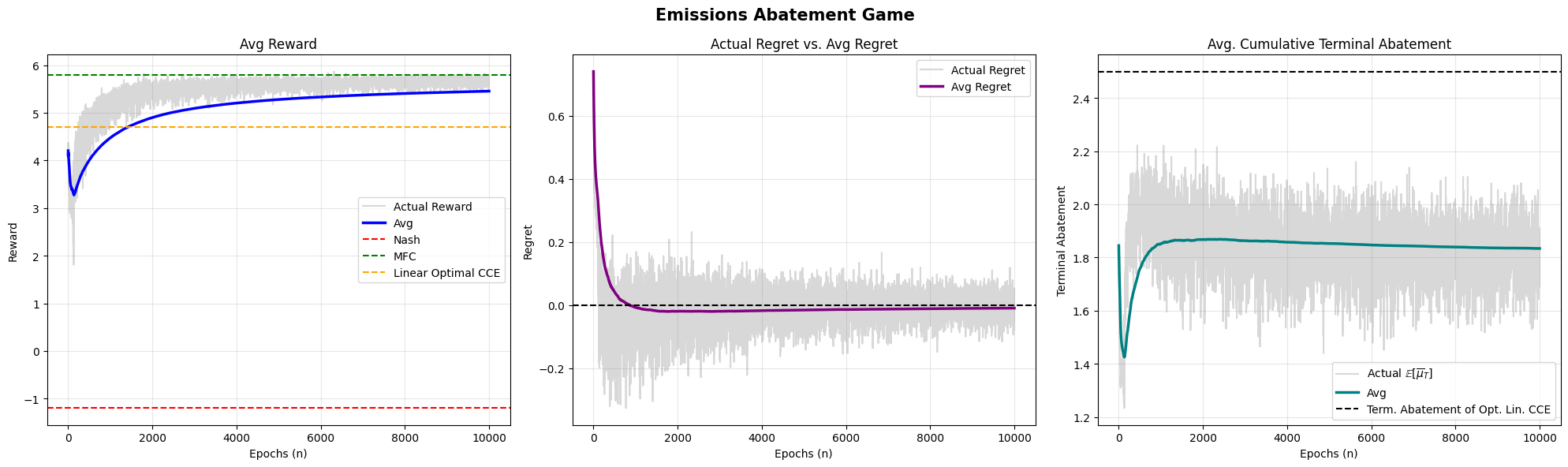}
    \caption{Optimal CCE, external regret and average terminal emission abatement for the mean-field emission abatement game.}
    \label{fig: optimal CCE emission abatement}
\end{figure}

As illustrated in Figure \ref{fig: optimal CCE emission abatement}, the average payoff generated by Algorithm \ref{algorithm:Primal_Dual_Gradient_Descent} converges to a higher level than the unconstrained linear benchmark from \cite{campi2025LQ}.
At the same time, the average external regret remains nonpositive, indicating that the learned recommendation is approximately a CCE for the system.
In the third plot, we report the terminal cumulative abatement $\E[\bar{\mu}^x_T]$.
For the learned payoff-oriented CCE, this quantity is approximately $1.8$, which is lower than the corresponding value for the unconstrained optimal linear CCE, approximately $2.5$; see Figure 5(b) in \cite{campi2025LQ}.
Thus, in this experiment, the CCE with higher payoff produces lower terminal cumulative abatement. 

\smallskip
Next, we consider the opposite case, where the coordinator aims to design an equilibrium that maximizes terminal cumulative abatement.
Equivalently, this corresponds to choosing the moderator objective
\[
    J^0(\lambda,\mu):=\E[\bar{\mu}^x_T]
\]
and maximizing it over all mean-field CCEs.
Using the same parametrization as in Tables \ref{table: math params Emission Abatement} and \ref{table: nn params Emission Abatement}, we obtain the results displayed in Figure \ref{fig: emission abatement social welfare}.

\begin{figure}[ht]
    \centering
    \includegraphics[width=1.0\linewidth]{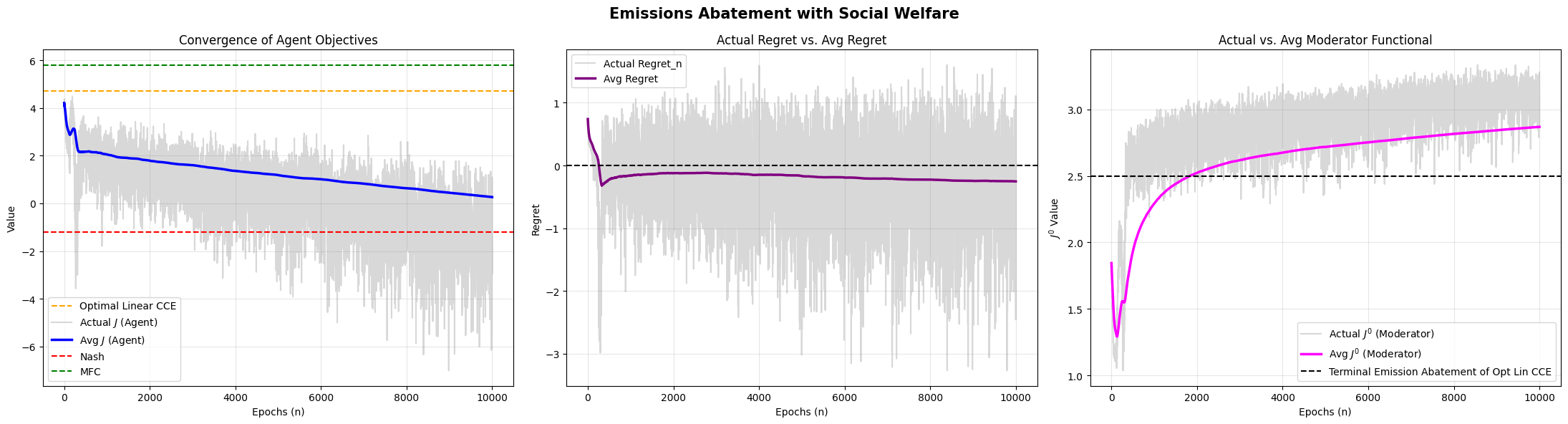}
    \caption{Approximation of a terminal-abatement-maximizing CCE for the emission abatement game.}
    \label{fig: emission abatement social welfare}
\end{figure}

The algorithm approximates an equilibrium with lower expected payoff than the payoff-oriented equilibrium reported in Figure \ref{fig: optimal CCE emission abatement}, while achieving a higher terminal cumulative abatement.
On the other hand, consistently with the findings of \cite{campi2025LQ}, the unconstrained linear benchmark yields a lower terminal cumulative abatement but higher expected payoff than the payoff-oriented equilibrium found by our learning algorithm.

\appendix
\section{Existence Part: Technical Results}\label{sec:appendix:technical_results_existence}

\begin{theorem}[Probabilistic Representation of LP-CCEs]\label{theorem:consistent_mtg_representation}
Let $\rho\in \cM$. There exist a measurable function $q:[0,T] \times \R \times V_{2}\times \cP_{2}(\R) \to \cP(A)$, a complete filtered probability space $(\Omega,\cF,\F,\P)$, with $(\Omega,\cF,\P)$ Polish, equipped with a Brownian motion $W$, an $\cF_0$-measurable random variable $\xi$ with law $m_0^*$, and a pair $(\mu,\overline{\mu})$, independent of $(W,\xi)$, consisting of a $V_2$-valued random variable $\mu$ and $\cP_2(\R)$-valued random variable $\overline{\mu}$ such that $\rho = \P \circ (\mu,\overline{\mu})^{-1}$, and an $\F$-adapted process $X$ such that
\begin{equation}\label{eq:theorem:consistent_mtg_representation:sde}
    dX_{t}=\int_{A}b(t,X_{t},\mu_{t},a)q_{t}(X_{t},\mu,\overline{\mu})(da)dt+\sigma dW_{t},\quad X_{0}=\xi.
\end{equation}
Moreover, for any $B\in \cB_{\R},\; C\in \cB_{A}$ we have
\begin{equation}\label{eq:consistency:appendix}
\begin{aligned}
    \mu_{t}(B\times C) & =\E \big[ \boldsymbol{1}_{B}(X_{t})q_{t}(X_{t},\mu,\overline{\mu})(C) |\mu,\overline{\mu}\big],\; dt\text{-a.e.},\; \P\text{-a.s.}, \\
    \overline{\mu}(B) & = \E \big[ \boldsymbol{1}_{B}(X_{T})\big| \mu,\overline{\mu} \big],\; \P\text{-a.s.}
\end{aligned}
\end{equation}
\end{theorem}
\begin{proof}
The proof is divided in four steps.

\textbf{Step 1}.
Denote by $m^x_t(dx)$ the marginal of $m_t(dx,da)$ over $\R$.
The first step is to show that there exists a measurable map $q:[0,T] \times \R \times V_2 \times \cP_2(\R) \to \cP(A)$ such that, for $\rho$-a.e. $(m,\overline{m})$, it holds
\begin{equation}\label{eq:disintegration:measure_m_t}
        m_t(dx,da) dt =  q_t(x, m , \overline{m})(da) m^x_t(dx)dt.
\end{equation}
To this extent, define a positive measure $\Q$ on $[0,T] \times \R \times A \times V_2 \times \cP_2(\R)$ by setting
\[
\Q(dt,dx,da,dm,d\overline{m}) := m_t(dx,da)dt\rho(dm,d\overline{m}).
\]
Denote the marginal of $\Q$ over $[0,T] \times \R \times V_2 \times \cP_2(\R)$ by $\Q^{t,x,m,\overline{m}}$.
Since $(A,\mathcal{B}_A)$ is a Polish space, by the disintegration theorem (see, e.g. \cite[Theorem 14.D.10]{RandomMeasuresBook}) there exists a measurable function $q:[0,T] \times \R \times V_2 \times \cP_2(\R) \to \cP(A)$ such that
\begin{equation}\label{eq:disintegration_Q1}
    \Q(dt,dx,da,dm,d\overline{m}) = m_t(dx,da) dt \rho(dm,d\overline{m}) = q_t(x,m,\overline{m})(da)\Q^{t,x,m,\overline{m}}(dt,dx,dm,d\overline{m}).
\end{equation}
Equation \eqref{eq:disintegration_Q1} implies that the stochastic kernels $q_t(x,m,\overline{m})m^{x}_t(dx)dt$ and $m_t(dx,da)dt$ are both disintegration kernels of $\Q$ with respect to $\rho(dm,d\overline{m})$.
Since the space $([0,T] \times \R \times A, \mathcal{B}_{[0,T] \times \R \times A})$ is Polish, uniqueness of the disintegration kernel implies that the kernels $q_t(x,m,\overline{m})m^{x}_t(dx)dt$ and $m_t(dx,da)dt$ are equal $\rho$-a.e., i.e. \eqref{eq:disintegration:measure_m_t} holds for $\rho$-a.e. $(m,\overline{m})$ in $V_2 \times \cP_2(\R)$.

\smallskip
\noindent
\textbf{Step 2}.
Define the function $B(t,x,m,\overline{m})$ by setting
\begin{equation}\label{eq:appendix:aggregate_drift}
    B(t,x,m,\overline{m}) :=  \int_A b(t,x,m_t,a)q_t(x,m,\overline{m})(da)
\end{equation}
Notice that $B(t,x,m,\overline{m})$ is jointly measurable in $(t,x,m,\overline{m})$, since $q_{t}(x,m,\overline{m})(da)$ is jointly measurable in $(x,m,\overline{m})$ by the disintegration theorem.
Moreover, since $b$ is bounded by Assumption \ref{standing_assumptions}, so is $B$.
Let $(\Omega,\F,\mathbb{F},\P)$ be a filtered probability space satisfying the usual assumptions, supporting a $\mathbb{F}$-Brownian motion $W$, an $\cF_0$-measurable $\R$-valued random $\xi$ with law $m^*_{0}$ and a $\F_0$-measurable pair of $\cF_0$-measurable random variables $(\mu,\overline{\mu})$ with values in $V_2 \times \cP_2(\R)$ so that $\P \circ (\mu,\overline{\mu})^{-1} = \rho$.
Assume that $\xi$, $W$ and $(\mu,\overline{\mu})$ are mutually independent.
This can be realized on the canonical probability space $\R \times \dC([0,T],\R) \times V_2 \times \cP_2(\R)$, which is Polish.
Then, the following holds:
\begin{enumerate}
    \item \label{mimicking:lemma_strong_existence:punto_eq_deterministica}
    For every $(m,\overline{m}) \in V_2 \times \cP_2(\R)$, the equation
    \begin{equation} \label{mimicking:lemma_strong_existence:eq_deterministica}
        dX^{(m,\overline{m})}_{t}=B(t,X^{(m,\overline{m})}_{t},m,\overline{m})dt+\sigma dW_t,\quad X_0^{(m,\overline{m})}=\xi.
    \end{equation}
    admits a pathwise unique strong solution.
    Moreover, let 
    \begin{equation}\label{eq:appendix:def_P_m}
        P^{(m,\overline{m})}:=\P\circ (X^{(m,\overline{m})})^{-1}.
    \end{equation}
    Then, the map $V_2 \times \cP_2(\R) \ni (m,\overline{m}) \mapsto P^{(m,\overline{m})} \in \cP(\dC([0,T];\R))$ is measurable.
    \item \label{mimicking:lemma_strong_existence:esistenza_sol_forte}
    There exists a continuous $\mathbb{F}$-adapted process $X$ solution to
    \begin{equation}\label{mimicking:lemma_strong_existence:eq_aleatoria}
        dX_t=B(t,X_t,\mu,\overline{\mu})dt + \sigma dW_t, \quad X_0=\xi.
    \end{equation}
    $X$ can be taken adapted to the $\P$-augmentation of the filtration $\mathbb{F}^{\xi,\mu,\overline{\mu},W}$, where $\cF^{\xi,W,\mu,\overline{\mu}}_t = \sigma( \xi,\mu,\overline{\mu}, W_s: \, 0 \leq s \leq t )$.
    
    \item \label{mimicking:lemma_strong_existence:punto_legge_soluzione}
    The joint law of $X$ and $(\mu,\overline{\mu})$ is given by 
    \begin{equation}\label{eq:appendix:disintegration_X_mu}
        \P\circ(X,\mu,\overline{\mu})^{-1}=P^{(m,\overline{m})}(dx)\rho(dm,d\overline{m}).
    \end{equation}
\end{enumerate}
These properties can be proved by the same methods as in \cite[Appendix~A]{LackerCloseLoop} and \cite[Appendix~A]{lacker_leflem2022} (see also \cite[Lemma E.11]{campi2024coarse}).
Notice that, in particular, \eqref{eq:appendix:disintegration_X_mu} implies that
\begin{multline}\label{eq:appendix:disintegration_time_t}
    \E[ \Phi(X_t,\mu,\overline{\mu}) ] = \int_{V_2 \times \cP_2(\R)} \int_{\dC([0,T];\R)} \Phi(x_t,m,\overline{m})P^{(m,\overline{m})}(dx)\rho(dm,d\overline{m}) \\
    = \int_{V_2 \times \cP_2(\R)} \E[\Phi(X^{(m,\overline{m})}_t,m,\overline{m})]\rho(dm,d\overline{m}) 
\end{multline}
for any $t \in [0,T]$ and any $\Phi:\R \times V_2 \times \cP_2(\R) \to \R$ bounded and measurable.

\smallskip
\noindent
\textbf{Step 3}.
We introduce the set
\[
    \cA:=\left\{(m,\overline{m})\in\operatorname{supp}(\rho):m_t(dx,da)dt=q_t(x,m,\overline{m})(da)m_t^x(dx)dt\right\}.
\]
By Step 1, we have $\rho(\cA) = 1$.
Since $\cA\subseteq\operatorname{supp}(\rho)$, Lemma \ref{LP:lemma:restriction_to_R_0} implies that, for every $(m,\overline{m})\in\cA$, it holds $(m,\overline{m})\in\cR[m]$ and therefore
\[
    (q_t(x,m,\overline{m})(da)m_t^x(dx),\overline{m})\in\cR[m].
\]
Then, for any $(m,\overline{m}) \in \cA$, the unique strong solution $X^{(m,\overline{m})}$ of \eqref{mimicking:lemma_strong_existence:eq_deterministica} is such that 
\begin{equation}\label{eq:consistency_fixed_flow}
\begin{aligned}
    m_{t}(C\times D) &=\E^{\P}\big[\boldsymbol{1}_{C}(X_{t}^{(m,\overline{m})})q_{t}(X_{t}^{(m,\overline{m})},m,\overline{m})(D)\big],\quad C\in \cB_{\R},\; D\in \cB_{A}, \quad dt\text{-a.e.,}\\
    \overline{m}(C)&= \E^{\P}[\boldsymbol{1}_{C}(X_{T}^{(m,\overline{m})})],\; C\in \cB_{\R}.
\end{aligned}
\end{equation}
To see our claim, we rely on \cite[Theorem C.6]{dumitrescu2021EJP}.
According to such result, for any $(m,\overline{m}) \in \cA$, since the disintegration of measure $m_t(dx,da)dt = q_t(x,m,\overline{m})(da)m^x_t(dx)dt$ holds, there exists a filtered probability space $(\tilde{\Omega},\tilde{\cF},\tilde{\bbF},\tilde{\P})$, possibly dependent on the pair $(m,\overline{m})$, equipped with an $\tilde{\bbF}$-Brownian motion $\tilde{W}$ and an $\tilde{\bbF}$-adapted process $\tilde{X}^{(m,\overline{m})}$ such that 
\begin{equation}\label{eq:SDE with aggregate drift}
    d\tilde{X}^{(m,\overline{m})}_{t}=\int_{A}b(t,\tilde{X}^{(m,\overline{m})}_{t},m_{t},a)q_{t}(\tilde{X}^{(m,\overline{m})}_{t},m,\overline{m})(da)dt+\sigma d\tilde{W}_{t}, \quad \tilde{\P}\circ (\tilde{X}_{0}^{(m,\overline{m})})^{-1}=m_0^*
\end{equation}
and such that 
\begin{equation}\label{eq:consistency_fixed_flow:space_Omega_tilde}
\begin{aligned}
    m_{t}(C\times D) &=\E^{\tilde{\P}}\big[\boldsymbol{1}_{C}(\tilde{X}_{t}^{(m,\overline{m})})q_{t}(\tilde{X}_{t}^{(m,\overline{m})},m,\overline{m})(D)\big],\quad C\in \cB_{\R},\; D\in \cB_{A}, \quad dt\text{-a.e.,}\\
    \overline{m}(C)&= \E^{\tilde{\P}}[\boldsymbol{1}_{C}(\tilde{X}_{T}^{(m,\overline{m})})],\; C\in \cB_{\R}.
\end{aligned}
\end{equation}
By definition of $B$ (cf. \eqref{eq:appendix:aggregate_drift}), equations \eqref{eq:SDE with aggregate drift} and \eqref{mimicking:lemma_strong_existence:eq_deterministica} coincide.
Thus, by point \ref{mimicking:lemma_strong_existence:punto_eq_deterministica} of Step 2, \eqref{eq:SDE with aggregate drift} admits a pathwise unique strong solution for any $(m,\overline{m}) \in V_2 \times \cP_2(\R)$.
Since strong uniqueness implies uniqueness in law, we deduce that $\tilde{\P} \circ (\tilde{X}^{(m,\overline{m})})^{-1} = \P \circ (X^{(m,\overline{m})})^{-1} = P^{(m,\overline{m})}$ for any $(m,\overline{m}) \in V_2 \times \cP_2(\R)$.
Therefore, for any $(m,\overline{m}) \in \cA$, we have
\begin{equation}\label{eq:appendix:equality_laws}
    m_t(C \times D) = \E^{\tilde{\P}}\big[\boldsymbol{1}_{C}(\tilde{X}_{t}^{(m,\overline{m})})q_{t}(\tilde{X}_{t}^{(m,\overline{m})},m,\overline{m})(D)\big] = \E^{\P}\big[\boldsymbol{1}_{C}(X_{t}^{(m,\overline{m})})q_{t}(X_{t}^{(m,\overline{m})},m,\overline{m})(D)\big]
\end{equation}
for $dt$-a.e. $t \in [0,T]$, $C \in \cB_{\R}$, $D \in \cB_{A}$.

\smallskip
\noindent
\textbf{Step 4}. We are now ready to prove \eqref{eq:consistency:appendix}.
Let $g:[0,T]\to \R$ and $\psi:V_{2}\times \cP_{2}(\R)\to \R$ be measurable and bounded functions.
Then, we have
\begin{equation}\label{eq:appendix:equalities_integrals}
\begin{aligned}
    \int_{0}^{T}\E \big[ \boldsymbol{1}_{B}(X_{t})&q_{t}(X_{t},\mu,\overline{\mu})(C)\psi(\mu,\overline{\mu})\big]g(t)dt \\
    &=\int_{V_{2}\times \cP_{2}(\R)}\int_{0}^{T}\E \big[ \boldsymbol{1}_{B}(X^{(m,\overline{m})}_{t})q_{t}(X^{(m,\overline{m})}_{t},m,\overline{m})(C)\big]g(t)dt\psi(m,\overline{m})\rho(dm,d\overline{m}) \\
    &=\int_{V_{2}\times \cP_{2}(\R)}\int_{0}^{T}m_{t}(B\times C)g(t)dt\psi(m,\overline{m})\rho(dm,d\overline{m}),
\end{aligned}
\end{equation}
where in the first equality we relied on \eqref{eq:appendix:disintegration_time_t} with $f(x,m,\overline{m}) = \boldsymbol{1}_{B}(x)q_t(x,m,\overline{m})(C)\psi(m,\overline{m})$ and in the second equality we relied on \eqref{eq:appendix:equality_laws}.
Since \eqref{eq:appendix:equalities_integrals} holds for any $g$ and $\psi$, we deduce that
\begin{equation*}
    \mu_t(B \times C) = \E[ \boldsymbol{1}_{B}(X_t) q_t(X_t,\mu,\overline{\mu})(C) \vert \mu,\overline{\mu} ] \quad \P\text{-a.s., } dt\text{-a.e.}
\end{equation*}
for any $B \in \cB_{\R}$, $C \in \cB_{A}$.
By the same reasoning, one can show that
\begin{equation*}
    \overline{\mu}(B)=\E^{\P}\big[ \boldsymbol{1}_{B}(X_{T})\big| \mu,\overline{\mu} \big],\; \P\text{-a.s., } \; B\in \cB_{\R}.
\end{equation*}
Thus, \eqref{eq:consistency:appendix} holds.
\end{proof}

\begin{lemma}\label{lemma:continuous_version_state_marginal}
Let $(m,\overline{m}) \mapsto P^{(m,\overline{m})}$ be defined by \eqref{eq:appendix:def_P_m}.
Let $\rho\in\cM$, and let $(\Omega,\cF,\F,\P)$, $X$ and $(\mu,\overline{\mu})$ with law $\rho$ be such that
\[
    \P\circ(X,\mu,\overline{\mu})^{-1}=P^{(m,\overline{m})}(dx)\rho(dm,d\overline{m}).
\]
Then, there exists an $\cF_0$-measurable random variable $\tilde{\mu}^x=(\tilde{\mu}^x_t)_{t\in[0,T]}$ with values in $\dC([0,T];\cP_2(\R))$ such that
\[
    \tilde{\mu}^x_t=\mu^x_t,\quad dt\text{-a.e.},\qquad \tilde{\mu}^x_T=\overline{\mu},
\]
and
\[
    \tilde{\mu}^x_t=\P(X_t\in\cdot\vert\tilde{\mu}^x),\quad t\in[0,T].
\]
\end{lemma}
\begin{proof}
For every $(m,\overline{m})$, define the marginal flow
\[
    \tilde{m}^x_t:=P^{(m,\overline{m})}\circ e_t^{-1},\quad t\in[0,T],
\]
where $e_t:\dC([0,T];\R)\to\R$ is the evaluation map.
Since $P^{(m,\overline{m})}$ is the law of a continuous solution with finite second moments, standard moment estimates imply that the map $t\mapsto \tilde{m}^x_t$ is continuous in $\cP_2(\R)$.
Moreover, by the consistency identities \eqref{eq:consistency_fixed_flow} in Theorem \ref{theorem:consistent_mtg_representation}, for $\rho$-a.e. $(m,\overline{m})$ it holds $\tilde{m}^x_t=m^x_t$ $dt$-a.e. and $\tilde{m}^x_T=\overline{m}$.
Using the measurable dependence of $(m,\overline{m})\mapsto P^{(m,\overline{m})}$, we can therefore define an $\cF_0$-measurable random flow $\tilde{\mu}^x$ by setting
\[
    \tilde{\mu}^x_t:=P^{(\mu,\overline{\mu})}\circ e_t^{-1},\quad t\in[0,T].
\]
Then $t\mapsto\tilde{\mu}^x_t$ is continuous $\P$-a.s., $\tilde{\mu}^x_t=\mu^x_t$, $dt$-a.e., and $\tilde{\mu}^x_T=\overline{\mu}$.
By the definition of $\tilde{\mu}^x$ and the conditional law identity above, we have $\tilde{\mu}^x_t(\cdot) = \P(X_t\in\cdot\vert\mu,\overline{\mu})$, $t \in [0,T]$.
Since $\tilde{\mu}^x$ is a measurable function of $(\mu,\overline{\mu})$, for every bounded measurable function $\varphi:\R\to\R$ and every $t\in[0,T]$, we have
\[
    \E[\varphi(X_t)\vert \tilde{\mu}^x] = \E\left[\E[\varphi(X_t)\vert \mu,\overline{\mu}]\vert \tilde{\mu}^x\right] = \int_{\R}\varphi(x)\tilde{\mu}^x_t(dx).
\]
Therefore, $\tilde{\mu}^x_t=\P(X_t\in\cdot\vert\tilde{\mu}^x)$, for any $t\in[0,T]$.
\end{proof}

\section{Learning Part: Technical Results}\label{section: Appendix: Learning Part: Technical Results}

\subsection{Technical results for Section \ref{section: Learning algorithm}}\label{appendix: technical results: learning part}

This section summarizes the technical results used in Section \ref{section: Learning algorithm}.

\subsubsection*{Auxiliary results for Section \ref{subsection: Lagrangian approach} }
\begin{lemma}\label{appendix: lemma: minimum_concave_usc}
Let $\bfJ:\R_+ \to \R$ be concave, upper semi-continuous and so that $\lim_{\lambda \to \infty} \bfJ(\lambda) = -\infty$ and $\sup_{\lambda \in \R_+}\bfJ(\lambda) < \infty$.
Then, there exists $\lambda^* \in \R_+$ which attains the maximum.
\end{lemma}
\begin{proof}
We define $M:=\sup_{\lambda\in \R_+}\bfJ(\lambda)$. Since $\lim_{\lambda\to\infty}\bfJ(\lambda)=-\infty$, there exists $R>0$ such that
\[
\bfJ(\lambda)\leq M-1,\quad \lambda\geq R.
\]
Hence, we have that
\[
    M=\sup_{\lambda\in \R_{+}}\bfJ(\lambda)=\sup_{\{\lambda:\bfJ(\lambda)>M-1\}}\bfJ(\lambda)\leq \sup_{[0,R]}\bfJ(\lambda)\leq M.
\]
Therefore, the maximum, if it exists, is attained in the compact set $[0,R]$. Since $\bfJ$ is upper-semi continuous over $[0,R]$, we deduce that there exists a maximum point $\lambda^*$.
\end{proof}

\subsubsection*{Auxiliary results for Section \ref{section: algorithm} }
We recall that $(\cR,\tau^{(2)} \otimes \tau^{(2)})$ is a compact separable and metrizable space (for the metrizability, see \cite[Proposition 2.13]{LinProgFictDumitrescu}); in particular, it is a Polish space. We endow $\cP(\cR)$ with the topology of weak convergence of probability measures on $\cR$. Since $\cR$ is compact and metrizable, $\cP(\cR)$ is compact and metrizable as well.
In this context, we introduce a tailor-made definition of linear derivative, inspired from \cite[Definition 5.43]{carmona2018probabilistic} (see also \cite[Definition 5.3]{GuZhang2026information} for flat derivatives on $\cP_{2}(\cP_{2}(\R))$).

\begin{definition}[Flat derivative]\label{appendix: definition: flat derivative}
Let $\phi:\cP(\cR)\to\R$. We say that $\phi$ admits a flat derivative (or linear derivative) if there exists a bounded and continuous function $\frac{\delta \phi}{\delta\rho}:\cP(\cR) \times \cR\to \R$, such that, for all $\rho,\rho' \in \cP(\cR)$, it holds that
\[
    \phi(\rho')-\phi(\rho)=\int_{0}^{1}\int_{\cR}\frac{\delta \phi}{\delta\rho}(\rho+\epsilon(\rho'-\rho),\boldsymbol{m})(\rho'-\rho)(d\boldsymbol{m})d\epsilon.
\]
\end{definition}

\begin{lemma}\label{appendix:lemma: linear derivative linear functional}
Let $\phi:\cR\to \R$ be continuous. Set $\Phi(\rho):=\int_{\cR}\phi(\boldsymbol{m})\rho(d\boldsymbol{m}),\; \rho\in \cP(\cR)$, then the linear derivative of $\Phi$ has the form
\[
    \frac{\delta \Phi}{\delta \rho}(\rho,\boldsymbol{m})=\phi(\boldsymbol{m}).
\]
\end{lemma}
The proof is straightforward from Definition \ref{appendix: definition: flat derivative} and it is therefore omitted.

\begin{lemma}\label{appendix:lemma: convexity in terms of first variation}
Let $F:\cP(\cR)\to \R$ admit first-order linear derivative. Then $F$ is convex if and only if for any $\rho,\rho'\in \cP(\cR)$, we have 
\begin{equation}\label{eq:appendix:lemma:convexity wrt first variation}
    F(\rho')-F(\rho)\geq \int_{\cR}\frac{\delta F}{\delta \rho}(\rho,\boldsymbol{m})(\rho'-\rho)(d\boldsymbol{m}).
\end{equation}
Moreover for Bregman divergence with Bregman potential $F$ (cf. Definition \ref{algo:def:Bregman_divergence}) holds
\begin{equation}\label{eq:appendix:lemma:convexity:positive_Bregman_divergence}
    D_{F}(\rho,\rho') \geq 0.
\end{equation}
\end{lemma}
\begin{proof}
We start by proving that convexity implies \eqref{eq:appendix:lemma:convexity wrt first variation}.
Let $\rho,\rho'\in \cP(\cR)$. For $t\in [0,1]$, we define $\rho_{t}:=(1-t)\rho+t\rho'$. We define the function $f:[0,1]\to \R$ to be $f(t):=F(\rho_{t})$ with $f(0)=F(\rho)$ and $f(1)=F(\rho')$. Notice that, $f$ is continuous and convex, moreover by using the definition of linear derivative, we have
\[
\frac{1}{t}(F(\rho_t)-F(\rho))=\int_0^1\int_{\cR}\frac{\delta F}{\delta \rho}\big(\rho+\epsilon t(\rho'-\rho),\boldsymbol{m}\big)(\rho'-\rho)(d\boldsymbol{m})d\epsilon.
\]
Since the linear derivative $\frac{\delta F}{\delta \rho}(\rho,\boldsymbol{m})$ is jointly continuous in $(\rho,\boldsymbol{m})$ and bounded, we rely on the dominated convergence theorem to exchange the integral and the limit to get
\begin{multline*}
    f'(0^{+}) = \lim_{t\downarrow 0}\frac{1}{t}(f(t)-f(0)) =\lim_{t \downarrow 0} \frac{1}{t}(F(\rho_t) - F(\rho)) \\
    = \int_0^1\int_{\cR} \lim_{t \downarrow 0} \frac{\delta F}{\delta \rho}\big( \rho + \epsilon t (\rho' - \rho) ,\boldsymbol{m} \big)(\rho' - \rho)(d\boldsymbol{m}) d\epsilon = \int_{\cR} \frac{\delta F}{\delta \rho} ( \rho,\boldsymbol{m} )(\rho' - \rho)(d\boldsymbol{m}).
\end{multline*}
On the other hand, by convexity of $f$, we have
\[
    f'(0^{+}) = \lim_{t \downarrow 0} \frac{f(t)-f(0)}{t}\leq f(1)-f(0),
\]
which yields
\[
    \int_{\cR} \frac{\delta F}{\delta \rho} ( \rho,\boldsymbol{m} )(\rho' - \rho)(d\boldsymbol{m})\leq F(\rho')-F(\rho).
\]
To show that \eqref{eq:appendix:lemma:convexity wrt first variation} implies convexity of $F$, we apply twice \eqref{eq:appendix:lemma:convexity wrt first variation} for $(\rho',\rho)=(\rho,\rho_{t})$ and $(\rho',\rho)=(\rho',\rho_{t})$, so that we get 
\[
    F(\rho)-F(\rho_{t})\geq \int_{\cR}\frac{\delta F}{\delta\rho}(\rho_{t},\boldsymbol{m})(\rho-\rho_{t})(d\boldsymbol{m}),
\]
and,
\[
    F(\rho')-F(\rho_{t})\geq \int_{\cR}\frac{\delta F}{\delta\rho}(\rho_{t},\boldsymbol{m})(\rho'-\rho_{t})(d\boldsymbol{m}).
\]
Multiplying the first equation by $(1-t)$ and the second equation by $t$, adding them up, and recalling definition of $\rho_{t}$, we find
\begin{align}
    (1-t)F(\rho)+tF(\rho')-F(\rho_{t})&\geq \int_{\cR}\frac{\delta F}{\delta\rho}(\rho_{t},\boldsymbol{m})\big((1-t)(\rho-\rho_{t})+t(\rho'-\rho_{t})\big)(d\boldsymbol{m}) \notag \\
    &=\int_{\cR}\frac{\delta F}{\delta\rho}(\rho_{t},\boldsymbol{m})\big((1-t)\rho+t\rho'-\rho_{t})(d\boldsymbol{m})=0,
\end{align}
which implies 
\begin{equation}
    (1-t)F(\rho)+tF(\rho')-F(\rho_{t})\geq 0.
\end{equation}
Finally, \eqref{eq:appendix:lemma:convexity:positive_Bregman_divergence} follows by definition of Bregman divergence (cf. Definition \ref{algo:def:Bregman_divergence}), by rearranging the terms in \eqref{eq:appendix:lemma:convexity wrt first variation}.
\end{proof}

\begin{proposition}[First order optimality condition on convex subsets of $\cP(\cR)$]\label{appendix: proposition: optimality condition}
    Let $\cK\subset \cP(\cR)$ be convex. Let $G:\cP(\cR) \to \R$ be a continuous function which admits a flat derivative.
    Assume that there exists $\widehat{\rho}\in \cK$ such that
    \begin{equation}
        \widehat{\rho}\in \operatorname{argmin}_{\rho \in \cK}G(\rho).
    \end{equation}
    Then, for any $\rho\in \cK$, it holds
    \begin{equation}
        \label{eq: optimality condition for proximal functional}
        \int_{\cR}\frac{\delta G}{\delta\rho}(\widehat{\rho},\boldsymbol{m})(\rho-\widehat{\rho})(d\boldsymbol{m})\geq 0.
    \end{equation}
\end{proposition}

\begin{proof}
Let $\rho\in \cK$ be arbitrary. We define $\rho_{t}:=(1-t)\widehat{\rho}+t\rho = \widehat{\rho} + t(\rho - \widehat{\rho})$, $t\in [0,1]$. Since $\cK$ is convex, $\rho_t\in \cK$ for every $t\in[0,1]$. Since $\widehat{\rho}$ minimizes $G$ over $\cK$, we have
\begin{equation}
    G(\rho_t)-G(\widehat{\rho})\geq 0,\quad t\in[0,1].
\end{equation}
Dividing by $t>0$ and taking the limit as $t\downarrow 0$, we obtain
\begin{equation}
    0 \leq \lim_{t\downarrow 0}\frac{G(\rho_t)-G(\widehat{\rho})}{t} = \int_{\cR}\frac{\delta G}{\delta\rho}(\widehat{\rho},\boldsymbol{m})(\rho-\widehat{\rho})(d\boldsymbol{m}),
\end{equation}
where the last equality follows by the same computations as in the proof of Lemma \ref{appendix:lemma: convexity in terms of first variation}.
\end{proof}

\begin{theorem}[Three-point inequality]\label{appendix: theorem: three-point inequality}
Let $\cK\subset \cP(\cR)$ be convex. Let $G:\cP(\cR)\to \R \cup \{+\infty\}$ be a convex and continuous function, which admits a flat derivative on $\cP(\cR)$, and let $\tau>0$ be fixed. Let $D_{\psi}$ be the Bregman divergence with Bregman potential $\psi$, and let $\overline{\rho} \in \cP(\cR)$.
Assume that there exists $\widehat{\rho}\in \cK$ such that 
\[
    \widehat{\rho}\in \operatorname{argmin}_{\rho\in \cK}\bigg( G(\rho)+\tau D_{\psi}(\rho,\overline{\rho}) \bigg).
\]
Then, for any $\rho\in \cK$, it holds
\begin{equation}\label{eq:theorem: three-point inequality}
    G(\widehat{\rho})-G(\rho)\leq \tau \left( D_{\psi}(\rho,\overline{\rho})-D_{\psi}(\rho,\widehat{\rho})-D_{\psi}(\widehat{\rho},\overline{\rho}) \right).
\end{equation}
\end{theorem}

\begin{proof}
The proof follows the same approach as in \cite[Lemma 2.5]{Mirror_Descent-Ascent_Mean-Field}; we also refer the reader to the proof of \cite[Lemma 3]{Mirror_Descent_with_Relative_Smoothness}. 
Set
\begin{equation}\label{appendix: eq: minimizing movement scheme with Bregman divergence}
    G_{\tau}(\rho;\overline{\rho}):=G(\rho)+\tau D_{\psi}(\rho,\overline{\rho}),\quad \rho,\overline{\rho}\in \cP(\cR).
\end{equation}
Let $\rho\in\cP(\cR)$ and $\rho'\in\cK$.
We start by noticing that
\begin{equation}\label{appendix: eq: linear derivative of Bregman divergence}
    \frac{\delta D_{\psi}}{\delta\rho}(\rho,\overline{\rho},\boldsymbol{m})=\frac{\delta\psi}{\delta\rho}(\rho,\boldsymbol{m})-\frac{\delta\psi}{\delta\rho}(\overline{\rho},\boldsymbol{m}),\quad \rho,\overline{\rho}\in \cP(\cR).
\end{equation}
Indeed, using the definition of linear derivative and Lemma \ref{appendix:lemma: linear derivative linear functional}, we obtain
\begin{align}
    \frac{\delta D_{\psi}}{\delta\rho}(\rho,\overline{\rho},\boldsymbol{m})&=\frac{\delta }{\delta \rho}\bigg( \psi(\rho)-\psi(\overline{\rho})-\int_{\cR}\frac{\delta \psi}{\delta\rho}(\overline{\rho},\boldsymbol{m})(\rho-\overline{\rho})(d\boldsymbol{m}) \bigg) \notag \\
    &=\frac{\delta\psi}{\delta\rho}(\rho,\boldsymbol{m})-\frac{\delta\psi}{\delta\rho}(\overline{\rho},\boldsymbol{m}).
\end{align}
Then, given $\overline{\rho}\in \cP(\cR)$, we show that
\begin{equation}\label{appendix: eq: Bregman divergence of Bregman divergence}
    D_{D_{\psi}(\cdot,\overline{\rho})}(\rho',\rho)=D_{\psi}(\rho',\rho),\quad \rho',\rho\in \cP(\cR).
\end{equation}
Indeed, by direct computations, it holds 
\begin{equation*}
\begin{aligned}
    D_{D_{\psi}(\cdot,\overline{\rho})}(\rho',\rho)&=D_{\psi}(\rho',\overline{\rho})-D_{\psi}(\rho,\overline{\rho})-\int_{\cR}\frac{\delta D_{\psi}}{\delta\rho}(\rho,\overline{\rho},\boldsymbol{m})(\rho'-\rho)(d\boldsymbol{m})  \\
    &=D_{\psi}(\rho',\overline{\rho})-D_{\psi}(\rho,\overline{\rho})-\int_{\cR}\frac{\delta \psi}{\delta \rho}(\rho,\boldsymbol{m})(\rho'-\rho)(d\boldsymbol{m})  \\
    &\quad +\int_{\cR}\frac{\delta \psi}{\delta \rho}(\overline{\rho},\boldsymbol{m})(\rho'-\rho)(d\boldsymbol{m})  \\
    &=\psi(\rho')-\psi(\rho)-\int_{\cR}\frac{\delta \psi}{\delta \rho}(\overline{\rho},\boldsymbol{m})(\rho'-\rho)(d\boldsymbol{m})  \\
    &\quad-\int_{\cR}\frac{\delta \psi}{\delta \rho}(\rho,\boldsymbol{m})(\rho'-\rho)(d\boldsymbol{m})+\int_{\cR}\frac{\delta \psi}{\delta \rho}(\overline{\rho},\boldsymbol{m})(\rho'-\rho)(d\boldsymbol{m}) \\
    &=\psi(\rho')-\psi(\rho)-\int_{\cR}\frac{\delta \psi}{\delta \rho}(\rho,\boldsymbol{m})(\rho'-\rho)(d\boldsymbol{m})=D_{\psi}(\rho',\rho).
\end{aligned}
\end{equation*}
From the additivity of the Bregman divergence with respect to the potential and \eqref{appendix: eq: Bregman divergence of Bregman divergence}, we obtain
\begin{align}\label{appendix:eq:inequality with diff Bregman diver.}
    D_{G_{\tau}(\cdot,\overline{\rho})}(\rho',\rho)&=D_{G}(\rho',\rho)+\tau D_{D_{\psi}(\cdot,\overline{\rho})}(\rho',\rho)\notag \\
    &=D_{G}(\rho',\rho)+\tau D_{\psi}(\rho',\rho)\geq \tau D_{\psi}(\rho',\rho),
\end{align}
where the last inequality holds by convexity of $G$ and \eqref{eq:appendix:lemma:convexity:positive_Bregman_divergence}. Expanding $D_{G_{\tau}(\cdot,\overline{\rho})}$ in \eqref{appendix:eq:inequality with diff Bregman diver.}, we find that
\begin{equation}\label{appendix:eq:inequality with the integral and G}
    G_{\tau}(\rho';\overline{\rho})-G_{\tau}(\rho;\overline{\rho})-\int_{\cR}\frac{\delta G_{\tau}}{\delta \rho}(\rho,\boldsymbol{m};\overline{\rho})(\rho'-\rho)(d\boldsymbol{m})\geq \tau D_{\psi}(\rho',\rho).
\end{equation}
Let now $\rho$ be equal to $\widehat{\rho}\in \operatorname{argmin}_{\rho\in \cK}G_{\tau}(\rho;\overline{\rho})$. By \eqref{appendix: eq: linear derivative of Bregman divergence}, the map $\rho\mapsto G_{\tau}(\rho;\overline{\rho})$ admits a flat derivative. Since $\widehat{\rho}$ minimizes $G_{\tau}(\rho;\overline{\rho})$ over $\cK$, from Proposition \ref{appendix: proposition: optimality condition} we obtain
\begin{equation}\label{appendix:eq:optimality condition in action}
    \int_{\cR}\frac{\delta G_{\tau}}{\delta\rho}(\widehat{\rho},\boldsymbol{m};\overline{\rho})(\rho'-\widehat{\rho})(d\boldsymbol{m})\geq 0,\quad\text{for any }\rho'\in \cK.
\end{equation}
Therefore, using \eqref{appendix:eq:optimality condition in action} in \eqref{appendix:eq:inequality with the integral and G}, we deduce that
\begin{equation}
    G_{\tau}(\rho';\overline{\rho})-G_{\tau}(\widehat{\rho};\overline{\rho})\geq \tau D_{\psi}(\rho',\widehat{\rho}),\quad \rho'\in\cK.
\end{equation}
By the definition of $G_\tau$, this gives
\begin{equation}
    G(\widehat{\rho})-G(\rho')\leq \tau \left( D_{\psi}(\rho',\overline{\rho})-D_{\psi}(\rho',\widehat{\rho})-D_{\psi}(\widehat{\rho},\overline{\rho}) \right),\quad \rho'\in\cK.
\end{equation}
Since $\rho'\in\cK$ is arbitrary, we conclude \eqref{eq:theorem: three-point inequality}.
\end{proof}

\section{Numerical results: how to check Assumptions \ref{assumptions: extra ass on algorithmic part} and \ref{ass: uniqueness maximizer}}\label{app:strictCCE}
We show here that Assumptions \ref{assumptions: extra ass on algorithmic part} and Assumption \ref{ass: uniqueness maximizer} holds true in the setting of the emission abatement game (Example 2), as the flocking model (Example 1) is a particular case of the emission abatement game.

\begin{lemma}\label{lemma:strict_feasibility_emission_example}
Consider the emission abatement example with the parameters of Table \ref{table: math params Emission Abatement}.
Then Assumption \ref{assumptions: extra ass on algorithmic part} is satisfied.
\end{lemma}

\begin{proof}
We work in the cost-minimization convention, i.e. we minimize the cost
\[
    J(\lambda,\mu)=\E\left[\int_0^T\left(-a\bar{\mu}^x_t+\frac{b}{2}(\bar{\mu}^x_t)^2+\frac{1}{2}\lambda_t^2+\frac{1}{2}(X^\lambda_t-\bar{\mu}^x_t)^2\right)dt\right].
\]
It is enough to construct a mean-field CCE $(\lambda,\mu)$ such that $J(\lambda,\mu)<\inf_{\beta\in\bbA}J(\beta,\mu)$.
Then, we set $\overline{\rho} = \P \circ (\mu,\mu^x_T)^{-1}$. By Proposition \ref{prop:any_strong_CCE_induces_LP_CCE}, $\overline{\rho} \in \cE$ and
\[
\Gamma[\overline{\rho}] = J(\lambda,\mu) < \inf_{\beta \in \bbA} J(\beta,\mu) \leq \inf_{\boldsymbol{\kappa} \in \cR} \Gamma^{dev}[\overline{\rho}](\boldsymbol{\kappa}),
\]
i.e. $\cE\cR(\overline{\rho}) < 0$ and Assumption \ref{assumptions: extra ass on algorithmic part} follows.

Let $\varepsilon$ be independent of $(\xi,W)$, with $\P(\varepsilon=1)=\P(\varepsilon=-1)=1/2$, and set $c=1$.
Define $\lambda_t:=\varepsilon c$ and $X^\lambda_t:=\xi+\varepsilon ct+W_t$, for $t\in[0,T]$, and let $\mu_t$ be the conditional law of $(X^\lambda_t,\lambda_t)$ given $\varepsilon$.
Since $c=1\in[-3,3]$, the recommendation $\lambda$ is admissible and, by construction, the correlated flow $(\lambda,\mu)$ satisfies the consistency condition \eqref{eq:CCE:consistency}.
Since $\bar{\mu}^x_t=\E[\xi]+\varepsilon ct$, we have $\E[\bar{\mu}^x_t]=\E[\xi]$, $\E[(\bar{\mu}^x_t)^2]=(\E[\xi])^2+c^2t^2$, $\lambda_t^2=c^2$, and $X^\lambda_t-\bar{\mu}^x_t=\xi-\E[\xi]+W_t$, so that 
\[
    J(\lambda,\mu)=-a\E[\xi]T+\frac{b}{2}(\E[\xi])^2T+\frac{b}{6}c^2T^3+\frac{1}{2}c^2T+\frac{1}{2}\mathrm{Var}(\xi)T+\frac{T^2}{4}.
\]
Let now $\beta\in\bbA$ be an admissible deviation.
Since $\beta$ is $\bbF^{\xi,W}$-progressively measurable and $\varepsilon$ is independent of $(\xi,W)$, the process $X^\beta$ is independent of $\varepsilon$.
Therefore,
\[
    \E\left[(X^\beta_t-\bar{\mu}^x_t)^2\right]=\E\left[(X^\beta_t-\E[\xi])^2\right]+c^2t^2.
\]
Hence,
\[
    \inf_{\beta \in \bbA} J(\beta,\mu) = \inf_{\beta\in\bbA}\E\left[\frac{1}{2}\int_0^T\left(\beta_t^2+(X^\beta_t-\E[\xi])^2\right)dt\right]+\frac{c^2T^3}{6} -a\E[\xi]T+\frac{b}{2}(\E[\xi])^2T+\frac{b}{6}c^2T^3.
\]
Let $J^{LQ}(\beta,0;y)$ denote the cost of the same linear-quadratic control problem with initial condition $Y_0=y$, zero mean-field flow, and controls taking values in the whole real line, with $\E[\int_0^T \beta_t^2dt]<\infty$.
Set
\[
    v^{LQ}(y,0):=\inf_{\beta}J^{LQ}(\beta,0;y).
\]
Since the unconstrained control class is larger than $\bbA$, we have
\[
    \inf_{\beta\in\bbA}\E\left[\frac{1}{2}\int_0^T\left(\beta_t^2+(X^\beta_t-\E[\xi])^2\right)dt\right]\geq \E\left[v^{LQ}(\xi-\E[\xi],0)\right].
\]
The unconstrained linear-quadratic problem is explicit and gives
\[
    v^{LQ}(y,0)=\frac{1}{2}\tanh(T)y^2+\frac{1}{2}\ln(\cosh T)\geq \frac{1}{2}\frac{T}{T+1}y^2.
\]
Therefore, since $\E[(\xi-\E[\xi])^2]=\mathrm{Var}(\xi)$,
\[
    \E\left[v^{LQ}(\xi-\E[\xi],0)\right]\geq \frac{1}{2}\frac{T}{T+1}\mathrm{Var}(\xi).
\]
Consequently,
\[
    \inf_{\beta \in \bbA}J(\beta,\mu)\geq \frac{1}{2}\frac{T}{T+1}\mathrm{Var}(\xi)+\frac{c^2T^3}{6} -a\E[\xi]T+\frac{b}{2}(\E[\xi])^2T+\frac{b}{6}c^2T^3.
\]
For the parameters of Table \ref{table: math params Emission Abatement}, we have $T=5$, $\mathrm{Var}(\xi)=0.11^2$, and $c=1$.
Thus, it holds
\[
    \frac{1}{2}c^2T+\frac{1}{2}\mathrm{Var}(\xi)T+\frac{T^2}{4}<\frac{1}{2}\frac{T}{T+1}\mathrm{Var}(\xi)+\frac{c^2T^3}{6}.
\]
Consequently, $J(\lambda,\mu)<J(\beta,\mu)$ for every $\beta\in\bbA$, which concludes the proof.
\end{proof}

\begin{lemma}\label{lemma:examples_uniqueness_maximizer}
Assumption \ref{ass: uniqueness maximizer} is satisfied in the emission abatement game.
\end{lemma}

\begin{proof}
We use the equivalent cost-minimization formulation, obtained by changing the sign of the payoff.
Fix $\rho\in\cP(\cR)$.
We show that there exists a unique $\boldsymbol{\kappa}\in\cR$ which minimizes $\Gamma^{dev}[\rho](\boldsymbol{\kappa})$.
This is equivalent to the uniqueness of the external-regret maximizer.

For $\boldsymbol{m}\in\cR$, write $\bar m^x_t:=\int_{\R\times A}x\,m_t(dx,da)$ and set $\frm^\rho_t:=\int_{\cR}\bar m^x_t\,\rho(d\boldsymbol{m})$.
The terms depending only on $\boldsymbol{m}$ are irrelevant for the optimal deviation problem.
Therefore, up to additive constants independent of the deviation, the best-deviation problem associated with $\rho$ is
\begin{equation}\label{eq:appendix:LP_ctrl_problem}
    \inf_{\boldsymbol{\kappa}\in\cR}\int_0^T\int_{\R\times A} F^{\rho}(t,x,a)\kappa_t(dx,da)dt,
\end{equation}
where $F^{\rho}(t,x,a):=\frac{1}{2}a^2+\frac{1}{2}x^2-\frm^{\rho}_t x$.
The corresponding strong control problem is
\begin{equation}\label{eq:appendix:ctrl_problem}
    \inf_{\beta\in\bbA}J^\rho(\beta),\quad J^\rho(\beta):=\E\left[\int_0^T F^{\rho}(t,X^{\beta}_t,\beta_t)dt \right], \quad dX^\beta_t=\beta_tdt+dW_t, \; X^\beta_0=\xi.
\end{equation}
Since $F^{\rho}(t,x,a)$ is strictly convex jointly in $(x,a)$, and since the dynamics are linear and $A$ is compact and convex, the strong control problem \eqref{eq:appendix:ctrl_problem} has a unique minimizer, up to $d\P\otimes dt$-a.e. equality.
Let $\beta^{\star}\in\bbA$ denote this minimizer.
We now show that, if $\boldsymbol{\kappa}^{\rho}$ is a minimizer of the LP deviation problem \eqref{eq:appendix:LP_ctrl_problem}, then
\begin{equation}\label{eq:appendix:conclusion}
    \kappa^{\rho}_t(dx,da)dt=\P\circ(X^{\beta^\star}_t,\beta^\star_t)^{-1}(dx,da)dt,\qquad \overline{\kappa}^{\rho}(dx)=\P\circ(X^{\beta^\star}_T)^{-1}(dx).
\end{equation}
This will imply the uniqueness of the LP minimizer.
Let $\boldsymbol{\kappa}^{\rho}$ be an LP minimizer.
The emission abatement game satisfies Assumptions \ref{ass: convexity relaxed control assumption} and \ref{ass: averaged_convexity}: the action space $A$ is compact and convex, the drift is affine in the control, and the cost is convex in the control.
Thus, by the measurable-selection argument in the proof of Proposition \ref{prop:any_strong_CCE_induces_LP_CCE}, there exists a strong deviation $\beta^{\boldsymbol{\kappa}^{\rho}}$ such that
\[
    J^\rho(\beta^{\boldsymbol{\kappa}^{\rho}})\leq \Gamma^{dev}[\rho](\boldsymbol{\kappa}^{\rho}).
\]
Disintegrate
\[
    \kappa^{\rho}_t(dx,da)=\kappa^{\rho,x}_t(dx)q^{\rho}_t(x)(da).
\]
Then $\beta^{\boldsymbol{\kappa}^{\rho}}_t=\alpha^{\rho}(t,X^{\rho}_t)$, where
\begin{equation}\label{eq:appendix:inequalities}
\begin{aligned}
    \alpha^\rho(t,x)&=\int_A a\,q^{\rho}_t(x)(da), \\&
    F^\rho(t,x,\alpha^{\rho}(t,x))=F^\rho\left(t,x,\int_A a\,q^{\rho}_t(x)(da)\right)\leq\int_A F^\rho(t,x,a)q^{\rho}_t(x)(da),
\end{aligned}
\end{equation}
and
\[
    dX^{\rho}_t=\alpha^{\rho}(t,X^{\rho}_t)dt+dW_t,\qquad X^{\rho}_0=\xi.
\]

Let $\boldsymbol{\kappa}^{\star}\in\cR$ be the LP deviation induced by $\beta^\star$, as in Proposition \ref{LP:relation:prop_deviations}.
Then
\begin{equation}\label{eq:appendix:chain_of_inequalities}
    \Gamma^{dev}[\rho](\boldsymbol{\kappa}^{\star})=J^\rho(\beta^\star)\leq J^\rho(\beta^{\boldsymbol{\kappa}^{\rho}})\leq\Gamma^{dev}[\rho](\boldsymbol{\kappa}^{\rho})=\inf_{\boldsymbol{\kappa}\in\cR}\Gamma^{dev}[\rho](\boldsymbol{\kappa})\leq\Gamma^{dev}[\rho](\boldsymbol{\kappa}^{\star}).
\end{equation}
All inequalities are therefore equalities. In particular, $\beta^{\boldsymbol{\kappa}^{\rho}}$ is a minimizer of the strong problem, and by uniqueness $\beta^{\boldsymbol{\kappa}^{\rho}}=\beta^\star$, $d\P\otimes dt$-a.e.
Moreover, $J^\rho(\beta^{\boldsymbol{\kappa}^{\rho}})=\Gamma^{dev}[\rho](\boldsymbol{\kappa}^{\rho})$.
We claim that this forces $q^\rho_t(x)(da)$ to be a Dirac mass for $dt\otimes\kappa^{\rho,x}_t(dx)$-a.e. $(t,x)$.
Indeed, for each fixed $(t,x)$, the map $a\mapsto F^\rho(t,x,a)$ is strictly convex.
Hence, by the equality case in Jensen's inequality, the inequality in \eqref{eq:appendix:inequalities} is strict unless $q^\rho_t(x)(da)$ is a Dirac mass.
If strict inequality held on a set of positive $dt\otimes\kappa^{\rho,x}_t(dx)$-measure, then the measurable-selection construction would give $J^\rho(\beta^{\boldsymbol{\kappa}^{\rho}})<\Gamma^{dev}[\rho](\boldsymbol{\kappa}^{\rho})$, contradicting \eqref{eq:appendix:chain_of_inequalities}.
Therefore $q^\rho_t(x)(da)$ is a Dirac mass for $dt\otimes\kappa^{\rho,x}_t(dx)$-a.e. $(t,x)$.
It follows that $\boldsymbol{\kappa}^{\rho}$ is induced by the strong deviation $\beta^{\boldsymbol{\kappa}^{\rho}}$.
Since $\beta^{\boldsymbol{\kappa}^{\rho}}=\beta^\star$ $d\P\otimes dt$-a.e., we obtain \eqref{eq:appendix:conclusion}.
Thus every LP minimizer coincides with the LP deviation induced by $\beta^\star$, and the LP best-deviation problem has a unique minimizer.
This concludes the proof.
\end{proof}

\subsection*{Acknowledgements}
Federico Cannerozzi and Ioannis Tzouanas gratefully acknowledge financial support from Deutsche Forschungsgemeinschaft (DFG, German Research Foundation) – Project-ID 317210226 – SFB 1283.

\bibliographystyle{siam}

\bibliography{bibliography}

\end{document}